\documentclass[11pt,a4paper,leqno]{article}
\usepackage{a4wide}
\usepackage{latexsym}
\usepackage{amsmath}
\usepackage{amssymb}
\newtheorem{defin}{Definition}
\newtheorem{lemma}{Lemma}
\newtheorem{prop}{Proposition}
\newtheorem{theo}{Theorem}
\newtheorem{corol}{Corollary}
\newenvironment{proof}{\medskip\par\noindent{\bf Proof}}{\hfill $\Box$
\medskip\par}
\begin{document}
\title{On parametric Gevrey asymptotics for some Cauchy problems in quasiperiodic function spaces}
\author{{\bf A. Lastra\footnote{The author is partially supported by the project MTM2012-31439 of Ministerio de Ciencia e
Innovacion, Spain}, S. Malek\footnote{The author is partially supported by the french ANR-10-JCJC 0105 project and the PHC
Polonium 2013 project No. 28217SG.}}\\
University of Alcal\'{a}, Departamento de F\'{i}sica y Matem\'{a}ticas,\\
Ap. de Correos 20, E-28871 Alcal\'{a} de Henares (Madrid), Spain,\\
University of Lille 1, Laboratoire Paul Painlev\'e,\\
59655 Villeneuve d'Ascq cedex, France,\\
{\tt alberto.lastra@uah.es}\\
{\tt Stephane.Malek@math.univ-lille1.fr }}
\date{March, 18 2014}
\maketitle
\thispagestyle{empty}
{ \small \begin{center}
{\bf Abstract}
\end{center}
We investigate Gevrey asymptotics for solutions to nonlinear parameter depending Cauchy problems with $2\pi$-periodic coefficients, for
initial data living in a space of quasiperiodic functions. By means of the Borel-Laplace
summation procedure, we construct sectorial holomorphic solutions which are shown to share the same formal power series as asymptotic
expansion in the perturbation parameter. We observe a small divisor phenomenon which emerges from the quasiperiodic nature of the
solutions space and which is the origin of the Gevrey type divergence of this formal series. Our result rests on the classical
Ramis-Sibuya theorem which asks to prove that the difference of any two neighboring constructed solutions satisfies some exponential decay.
This is done by an asymptotic study of a Dirichlet-like series whose exponents are positive real numbers which accumulate to
the origin.\medskip

} \bigskip \bigskip

\section{Introduction}

We consider a family of nonlinear Cauchy problems of the form
\begin{multline}
( \epsilon^{r_{3}}(t^{2}\partial_{t} + t)^{r_2} + (-i\partial_{z}+1)^{r_1}) \partial_{x}^{S}X_{i}(t,z,x,\epsilon) \\
= \sum_{\underline{k}=(s,k_{0},k_{1},k_{2}) \in \mathcal{S}} b_{\underline{k}}(z,x,\epsilon)
t^{s}(\partial_{t}^{k_0}\partial_{z}^{k_1}\partial_{x}^{k_2}X_{i}(t,z,x,\epsilon)\\
+ \sum_{\underline{l}=(l_{0},l_{1}) \in \mathcal{N}} c_{\underline{l}}(z,x,\epsilon) t^{l_{0}+l_{1}-1}(X_{i}(t,z,x,\epsilon))^{l_1}
\label{main_CP_intro}
\end{multline}
for given initial data
\begin{equation}
(\partial_{x}^{j}X_{i})(t,z,0,\epsilon) = \Xi_{i,j}(t,z,\epsilon) \ \ , \ \ 0 \leq i \leq \nu-1, 0 \leq j \leq S-1 \label{main_CP_init_cond_intro}
\end{equation}
where $\epsilon$ is a complex parameter, $S,r_{1},r_{2},r_{3}$ are some positive integers, $\mathcal{S}$ is a finite subset of
$\mathbb{N}^{4}$ and $\mathcal{N}$ is a finite subset of $\mathbb{N}^{2}$ that fulfills the constraints (\ref{shape_main_CP}).
The coefficients $b_{\underline{k}}(z,x,\epsilon)$ of the linear part and $c_{\underline{l}}(z,x,\epsilon)$ of the nonlinear part are 
$2\pi$-periodic Fourier series
$$ b_{\underline{k}}(z,x,\epsilon) = \sum_{\beta \geq 0} b_{\underline{k},\beta}(x,\epsilon) e^{iz \beta} \ \ , \ \
c_{\underline{l}}(z,x,\epsilon) = \sum_{\beta \geq 0} c_{\underline{l},\beta}(x,\epsilon) e^{iz \beta} $$
in the variable $z$ with coefficients $b_{\underline{k},\beta}(x,\epsilon), c_{\underline{l},\beta}(x,\epsilon)$ in
$\mathcal{O}\{ x,\epsilon \}$ (which denotes the Banach space
of bounded holomorphic functions in $(x,\epsilon)$ on some small polydisc $D(0,\rho) \times D(0,\epsilon_{0})$ centered at the origin in
$\mathbb{C}^{2}$ with supremum norm). We assume that all Fourier coefficients
$b_{\underline{k},\beta}(x,\epsilon)$, $c_{\underline{l},\beta}(x,\epsilon)$ have exponential decay in $\beta$ (see
(\ref{b_sk_beta_decay}), (\ref{c_l_beta_decay})). Hence,
$b_{\underline{k}}(z,x,\epsilon)$ and $c_{\underline{l}}(z,x,\epsilon)$ define bounded holomorphic functions on
$H_{\rho'} \times D(0,\rho) \times D(0,\epsilon_{0})$ where $H_{\rho'} = \{ z \in \mathbb{C} / |\mathrm{Im}(z)| < \rho' \}$ is some
strip of width $2\rho'>0$.

The initial data are quasiperiodic in the variable $z$ and are constructed as follows,
\begin{equation}
\Xi_{i,j}(t,z,\epsilon) = \sum_{\beta_{0},\ldots,\beta_{l} \geq 0}
\Xi_{i,\beta_{0},\ldots,\beta_{l},j}(t,\epsilon) \frac{\exp( iz( \sum_{j=0}^{l} \beta_{j} \xi_{j} )) }{\beta_{0}! \cdots \beta_{l}!}
\label{main_CP_init_data_intro}
\end{equation}
where $\xi_{0}=1$ and $\xi_{1},\ldots,\xi_{l}$ are real algebraic numbers (for some integer $l \geq 1$)
such that the family $\{1, \xi_{1}, \ldots, \xi_{l} \}$ is $\mathbb{Z}$-linearly independent and where
the coefficients $\Xi_{i,\beta_{0},\ldots,\beta_{l},j}(t,\epsilon)$ are bounded holomorphic
functions on $\mathcal{T} \times \mathcal{E}_{i}$, where $\mathcal{T}$ is a fixed open bounded sector centered at 0
and $\underline{\mathcal{E}} = \{ \mathcal{E}_{i} \}_{0 \leq i \leq \nu-1}$ is a family of open bounded sectors centered at the origin and
whose union form a covering of $\mathcal{V} \setminus \{ 0 \}$, where $\mathcal{V}$ denotes some bounded neighborhood of 0.
These functions (\ref{main_CP_init_data_intro}) are constructed in such a way that they define holomorphic functions on
$\mathcal{T} \times H_{\rho_{1}'} \times \mathcal{E}_{i}$ for some $0<\rho_{1}'<\rho'$. 

Recall that a function
$f : \mathbb{R} \rightarrow \mathbb{E}$ (where $\mathbb{E}$ denotes some vector space) is said to be \emph{quasiperiodic} with
period $w=(w_{0},\ldots,w_{l}) \in \mathbb{R}^{l+1}$, for some integer $l \geq 1$, if there exists a function
$F : \mathbb{R}^{l+1} \rightarrow \mathbb{E}$ such that for all $0 \leq j \leq l$, the partial function
$x_{j} \mapsto F(x_{0},\ldots,x_{j},\ldots,x_{l})$ is $w_{j}$-periodic on $\mathbb{R}$, for all fixed $x_{k} \in \mathbb{R}$ when $k \neq j$,
which satisfies $f(t) = F(t,\ldots,t)$ (see for instance \cite{na} for a definition and properties of quasiperiodic functions).
In particular, one can check that the functions (\ref{main_CP_init_data_intro}) are quasiperiodic with period
$w=(2\pi,2\pi/\xi_{1},\ldots,2\pi/\xi_{l})$ when seeing as functions of the variable $z$.

Our main purpose is the construction of actual holomorphic functions $X_{i}(t,z,x,\epsilon)$ to the problem
(\ref{main_CP_intro}), (\ref{main_CP_init_cond_intro}) on the domains
$\mathcal{T} \times H_{\rho_{1}'} \times D(0,\rho_{1}) \times \mathcal{E}_{i}$ for some small disc
$D(0,\rho_{1}) \subset \mathbb{C}$ and the
analysis of their asymptotic expansions as $\epsilon$ tends to zero on $\mathcal{E}_{i}$, for all $0 \leq i \leq \nu-1$. More precisely, we can
state our main result as follows.\medskip

\noindent {\bf Main statement} \emph{We take a set of directions $d_{i} \in \mathbb{R}$, $0 \leq i \leq \nu-1$, such that
$d_{i} \neq \pi \frac{2k+1}{r_2}$, for $0 \leq k \leq r_{2}-1$, which are assumed to satisfy moreover
$$ \frac{r_3}{r_2}\mathrm{arg}(\epsilon) + \mathrm{arg}(t) \in (d_{i}-\theta,d_{i}+\theta) $$
for all $\epsilon \in \mathcal{E}_{i}$, all $t \in \mathcal{T}$, all $0 \leq i \leq \nu-1$, for some fixed $\theta > \pi$. We make the hypothesis that
the coefficients $\Xi_{i,\beta_{0},\ldots,\beta_{l},j}(t,\epsilon)$ of the initial data (\ref{main_CP_init_data_intro}) can be expressed as
Laplace transforms
$$ \Xi_{i,\beta_{0},\ldots,\beta_{l},j}(t,\epsilon) = \frac{1}{\epsilon^{r_{3}/r_{2}}t}
\int_{L_{d_i}} V_{i,\beta_{0},\ldots,\beta_{l},j}(\tau,\epsilon) e^{-\tau/(\epsilon^{r_{3}/r_{2}}t)} d\tau $$
on $\mathcal{T} \times \mathcal{E}_{i}$ along the halfline $L_{d_i} = \mathbb{R}_{+}e^{\sqrt{-1}d_{i}}$, where
$V_{i,\beta_{0},\ldots,\beta_{l},j}(\tau,\epsilon)$ is a family of holomorphic functions which share the exponential growth constraints
(\ref{V_beta_j_defin}) with respect to $\tau$, the uniform bound estimates (\ref{norm_varphi_j<I})
and the analytic continuation property (\ref{V_Udi_j_analyt_cont_V_j}).}

\emph{Then, in Proposition 12, we construct a family of holomorphic and bounded functions
\begin{equation}
X_{i}(t,z,x,\epsilon) = \sum_{\beta_{0},\ldots,\beta_{l} \geq 0} X_{i,\beta_{0},\ldots,\beta_{l}}(t,x,\epsilon)
\frac{\exp( iz( \sum_{j=0}^{l} \beta_{j} \xi_{j} )) }{\beta_{0}! \cdots \beta_{l}!} \label{quasiperiodic_X_i_intro}
\end{equation}
which are quasiperiodic with period $w=(2\pi,2\pi/\xi_{1},\ldots,2\pi/\xi_{l})$ in the variable $z$ and which solve the
problem (\ref{main_CP_intro}), (\ref{main_CP_init_cond_intro}) on the products
$\mathcal{T} \times H_{\rho_{1}'} \times D(0,\rho_{1}) \times \mathcal{E}_{i}$, where
$\rho_{1}'>0$ satisfies the inequality (\ref{choice_M0_rho_1_prime}) and for some small radius $0 < \rho_{1}<\rho$. Moreover, the differences
$X_{i+1}(t,z,x,\epsilon) - X_{i}(t,z,x,\epsilon)$ satisfy the exponential decay (\ref{|X_i+1_minus_X_i|<}) whose type depends on the
constants $r_{1},r_{2},r_{3}$ and on the degree $h+1$ any algebraic number field $\mathbb{Q}(\xi)$ containing $\xi_{1},\ldots,\xi_{l}$.}

\emph{In Theorem 1, we show the existence of a formal series
\begin{equation}
 \hat{X}(\epsilon) = \sum_{k \geq 0} H_{k}(t,z,x) \frac{ \epsilon^{k} }{k!}
\end{equation}
whose coefficients $H_{k}(t,z,x)$ belong to the Banach space of bounded holomorphic functions on
$\mathcal{T} \times H_{\rho_{1}'} \times D(0,\rho_{1})$, which formally solves the equation (\ref{main_CP_intro}) and is moreover
the Gevrey asymptotic expansion of order $\frac{hr_{1}+r_{2}}{r_{3}}$ of $X_{i}$ on $\mathcal{E}_{i}$. In other words, there exist two
constants $C,M>0$ such that
$$ \sup_{t \in \mathcal{T}, z \in H_{\rho_{1}'}, x \in D(0,\rho_{1})} |X_{i}(t,z,x,\epsilon) - \sum_{k=0}^{N-1} H_{k}(t,z,x)
\frac{\epsilon^{k}}{k!}| \leq CM^{N}N!^{\frac{hr_{1}+r_{2}}{r_3}}|\epsilon|^{N} $$
for all $N \geq 1$, all $\epsilon \in \mathcal{E}_{i}$.}

Notice that the problem (\ref{main_CP_intro}), (\ref{main_CP_init_cond_intro}) is singularly perturbed with irregular singularity
(in the sense of T. Mandai, \cite{man}) with respect to $t$ at $t=0$ provided that $r_{2} > r_{1}$. It is of Kowalevski type
if $r_{2}<r_{1}$ (meaning that the hypotheses of the classical Cauchy-Kowalevski theorem (see for instance \cite{ho}, p. 346--349) are
fulfilled for the equation (\ref{main_CP_intro})) and of mixed type when $r_{2}$ and $r_{1}$ are equal.

In a recent work \cite{lamasa}, we have considered singularly perturbed nonlinear Cauchy problems of the form
\begin{equation}
\epsilon^{r_3}(z\partial_{z})^{r_1}(t^{2}\partial_{t})^{r_2}\partial_{z}^{S}u_{i}(t,z,\epsilon) =
F(t,z,\epsilon,\partial_{t},\partial_{z})u_{i}(t,z,\epsilon) + P(t,z,\epsilon,u_{i}(t,z,\epsilon)) \label{SPFCP}
\end{equation}
which carry both a irregular singularity with respect to $t$ at $t=0$ and a Fuchsian singularity (see \cite{geta} for a definition) with respect
to $z$ at $z=0$, for given initial data
\begin{equation}
(\partial_{z}^{j}u_{i})(t,0,\epsilon) = \phi_{i,j}(t,\epsilon) \ \ , \ \ 0 \leq i \leq \nu-1, 0 \leq j \leq S-1, \label{SPFCP_init_cond}
\end{equation}
where $F$ is some linear differential operator with polynomial coefficients and $P$ some polynomial. The initial data $\phi_{j,i}(t,\epsilon)$
were assumed to be holomorphic on products $\mathcal{T} \times \mathcal{E}_{i}$.
Under suitable constraints on the shape of the equation (\ref{SPFCP}) and on the initial data (\ref{SPFCP_init_cond}), we have shown
the existence of a formal series $\hat{u}(\epsilon) = \sum_{k \geq 0} h_{k} \epsilon^{k}/k!$ with coefficients $h_{k}$ belonging to
the Banach space $\mathbb{F}$ of bounded holomorphic functions on $\mathcal{T} \times D(0,\delta)$ (for some $\delta>0$)
equipped with the supremum norm, solution of (\ref{SPFCP}), which is the Gevrey asymptotic expansion of order $\frac{r_{1}+r_{2}}{r_3}$
of actual holomorphic solutions $u_{i}$ of (\ref{SPFCP}), (\ref{SPFCP_init_cond}) on $\mathcal{E}_{i}$ as $\mathbb{F}-$valued
functions, for all $0 \leq i \leq \nu-1$.

Compared to this former result \cite{lamasa}, the singularity nature of the equation (\ref{main_CP_intro}) does not come from
the divergence of the formal series.
This divergence rather emerges from the quasiperiodic structure of the solution space which produces a small divisor problem
(as we will see below) and its Gevrey type depends not only on the type of space
of our initial data but also on the shape of the equation (\ref{main_CP_intro}). It is worth noticing that a similar phenomenon has been
observed in the paper \cite{ioru} for the steady Swift-Hohenberg equation
\begin{equation}
(1+\Delta)^{2}U(\mathbf{x},\mu) - \mu U(\mathbf{x},\mu) + U^{3}(\mathbf{x},\mu) = 0 \label{SH}
\end{equation}
where the authors have constructed formal series solutions
\begin{equation}
U(\mathbf{x},\mu) = \sqrt{\mu} \sum_{n \geq 0} U^{(n)}(\mathbf{x}) \mu^{n} \label{formal_sol_U_SH}
\end{equation}
where the coefficients $U^{(n)}(\mathbf{x})$ belong to some weighted Sobolev space $H^{s}(\Gamma)$ (for well chosen real number $s>0$)
of quasiperiodic Fourier expansions in $\mathbf{x} \in \mathbb{R}^{2}$ of the form
$$ U^{(n)}(\mathbf{x}) = \sum_{\mathbf{k} \in \Gamma} U_{\mathbf{k}} e^{i \mathbf{k} \cdot \mathbf{x} } $$
where $\Gamma = \{ \sum_{j=1}^{Q} m_{j} \mathbf{k}_{j} / (m_{1},\ldots,m_{Q}) \in \mathbb{N}^{Q} \}$ with
$\mathbf{k}_{j}=(\cos(2 \pi \frac{j-1}{Q}), \sin(2 \pi \frac{j-1}{Q})) $ is a so-called quasilattice in $\mathbb{R}^{2}$ for some
integer $Q \geq 8$. They have shown that this formal series (\ref{formal_sol_U_SH}) is actually at most of
Gevrey order $4l$ (for a suitable integer $l$ depending on $Q$) as series in the Hilbert space $H^{s}(\Gamma)$. Their main purpose was
actually to use this result in order to construct approximate smooth quasiperiodic
solutions of (\ref{SH}) up to an exponential small order by means of truncated Laplace transforms.

In a more general setting, the Cauchy problem (\ref{main_CP_intro}), (\ref{main_CP_init_cond_intro}) we consider in this work comes within the
framework of asymptotic analysis of solutions to differential equations or to partial differential equations with periodic or
quasiperiodic coefficients which is a domain of intense research these last years.

In the category of differential equations most of the results concern nonlinear equations of the form
$$ \sum_{k=0}^{K} a_{k}(\epsilon) \partial_{t}^{k}u(t,\epsilon) = F(u(t,\epsilon),t,\epsilon) $$
where the forcing term $F$ contains periodic or quasiperiodic coefficients. These statements deal with the construction of formal
solutions $u(t,\epsilon) = \sum_{l=0}^{+\infty} u_{l}(t) \epsilon^{l}$ which are called Lindstedt series in the
literature. For convergence properties of these series, we quote the seminal work \cite{el} and the overview \cite{ge1}, for Borel resummation
procedures applied more recently, we mention \cite{gebade}. For applications in KAM theory for nearly integrable finitely dimensional
Hamiltonian systems, we may refer to \cite{bage}, \cite{cogagegi}.

In the context of partial differential equations, for existence results of quasiperiodic solutions to general families of nonlinear PDE containing
a small real parameter, we indicate \cite{ya2} and for the construction of periodic solutions to abstract second order nonlinear equations,
we notice \cite{ya1}.
Concerning KAM theory results in the context of PDE such as small nonlinear perturbations of wave equations or Schr\"{o}dinger equations
we mention the fundamental works \cite{crwa}, \cite{ku}, \cite{wa}.\medskip

Now, we explain our main result and the principal arguments needed in its proof. The first step consists (as in \cite{lamasa})
of transforming the equation (\ref{main_CP_intro}) by means of the linear map $T \mapsto T/\epsilon^{r_{3}/r_{2}}$ into an auxiliary
regularly perturbed nonlinear equation (\ref{SCP}). The drawback of this transformation is the appearance of poles in the coefficients of
this new equation with respect to $\epsilon$ at 0.

The approach we follow is the same as in our previous works \cite{lamasa}, \cite{ma1} and is based on a Borel resummation procedure applied
to formal expansions of the form
$$ \hat{Y}(T,z,x,\epsilon) = \sum_{\underline{\beta} = (\beta_{0},\ldots,\beta_{l},\beta_{l+1}) \in \mathbb{N}^{l+2}}
\hat{Y}_{\underline{\beta}}(T,\epsilon)
\frac{\exp( iz( \sum_{j=0}^{l} \beta_{j} \xi_{j} )) }{\beta_{0}! \cdots \beta_{l}!}\frac{x^{\beta_{l+1}}}{\beta_{l+1}!} $$
where $\hat{Y}_{\underline{\beta}}(T,\epsilon) = \sum_{m \geq 0} \chi_{m,\underline{\beta}}(\epsilon) T^{m}/m!$
are formal series in $T$, which formally solves the auxiliary equation (\ref{SCP}) for well chosen initial data (\ref{SCP_init_cond_formal}).
It is worth pointing out that this resummation method known as $\kappa-$summability already enjoys a large success in the study
of Gevrey asymptotics for analytic solutions to linear and nonlinear differential equations with irregular singularity, see for instance 
\cite{ba}, \cite{br},  \cite{cos},  \cite{ec}, \cite{mal}, \cite{ra}, \cite{rasi}.
We show that the formal Borel transform of $\hat{Y}(T,z,x,\epsilon)$ with respect to $T$ given by
$$ \hat{V}(\tau,z,x,\epsilon) = \sum_{\underline{\beta} = (\beta_{0},\ldots,\beta_{l},\beta_{l+1}) \in \mathbb{N}^{l+2}}
\hat{V}_{\underline{\beta}}(\tau,\epsilon)
\frac{\exp( iz( \sum_{j=0}^{l} \beta_{j} \xi_{j} )) }{\beta_{0}! \cdots \beta_{l}!}\frac{x^{\beta_{l+1}}}{\beta_{l+1}!} $$
where $\hat{V}_{\underline{\beta}}(\tau,\epsilon) = \sum_{m \geq 0}
\chi_{m,\underline{\beta}}(\epsilon)\tau^{m}/(m!)^2$,
formally solves a nonlinear convolution integro-differen\-tial Cauchy problem with rational coefficients in $\tau$, holomorphic with respect to
$x$ near the origin and with respect to $z$ in some strip, and meromorphic in $\epsilon$ with a pole at 0, see 
 (\ref{Borel_SCP_modif}), (\ref{Borel_SCP_init_cond_modif}).

For appropriate initial data satisfying the conditions (\ref{V_beta_j_defin}), (\ref{V_Udi_j_analyt_cont_V_j}) and (\ref{norm_varphi_j<I}), we show
(in Proposition 9) that the formal series $\hat{V}(\tau,z,x,\epsilon)$ actually defines a holomorphic function
$V_{i}$ on the product $U_{i} \times H_{\rho_{1}'} \times D(0,\rho_{1}) \times D(0,\epsilon_{0}) \setminus \{ 0 \}$,
for some $0 < \rho_{1}' < \rho'$, $0 < \rho_{1} < \rho$ and where $U_{i}$ is some unbounded open sector with small aperture and with
bisecting direction $d_{i}$ (as described above in the main statement). The functions $V_{i}$ have exponential growth
rate with respect to $(\tau,\epsilon)$ meaning that there exist two constants $C,K>0$ such that
\begin{equation}
\sup_{z \in H_{\rho_{1}'},x \in D(0,\rho_{1})} |V_{i}(\tau,z,x,\epsilon)| \leq Ce^{K|\tau|/|\epsilon|} \label{exp_bounds_V_i}
\end{equation}
for all $\tau \in U_{i}$, $\epsilon \in D(0,\epsilon_{0}) \setminus \{ 0 \}$. Moreover, we show that for all
$\underline{\beta} = (\beta_{0},\ldots,\beta_{l},\beta_{l+1}) \in \mathbb{N}^{l+2}$, the formal series
$\hat{V}_{\underline{\beta}}(\tau,\epsilon)$ actually define holomorphic functions $V_{i,\underline{\beta}}(\tau,\epsilon)$ on
domains $(U_{i} \cup D(0,\rho_{\underline{\beta}})) \times D(0,\epsilon_{0}) \setminus \{ 0 \}$ where $\rho_{\underline{\beta}}$ is a
Riemann type sequence of the form $R/(1+ |\underline{\beta}|)^{hr_{1}/r_{2}}$, for some constant $R>0$, which tends to 0 as
$|\underline{\beta}|=\sum_{j=0}^{l+1} \beta_{j}$ tends to infinity and share the same exponential growth rate, namely that
there exist constants $C>0$, $K>0$, $M>0$ with
\begin{equation}
\sup_{z \in H_{\rho_{1}'},x \in D(0,\rho_{1})} |V_{i,\underline{\beta}}(\tau,\epsilon)| \leq C K^{\sum_{j=0}^{l+1}\beta_{j}}
\beta_{0}!\cdots\beta_{l+1}!e^{M|\tau|/|\epsilon|} \label{exp_bounds_V_i_beta_intro}
\end{equation}
for all $\tau \in U_{i}$, $\epsilon \in D(0,\epsilon_{0}) \setminus \{ 0 \}$. We point out that the occurence of a radius of convergence
shrinking to zero for the coefficients $V_{i,\underline{\beta}}$ near the origin of the Borel transform is due to the presence of a
small divisor phenomenon in the convolution Cauchy problem (\ref{Borel_SCP_modif}), (\ref{Borel_SCP_init_cond_modif}) mentioned above.
In our previous study \cite{lamasa}, a similar outcome was caused by a leading term in the main equation (\ref{SPFCP}) containing a
Fuchsian operator $(z\partial_{z})^{r_1}$. In this analysis, the denominators arise from the function space where the solutions are found,
especially from their Fourier exponents $\sum_{j=0}^{l+1} \beta_{j}\xi_{j}$ which may tend to zero but not faster than a Riemann type
sequence as follows from Lemma 5.

In order to get the estimates described above, we use a majorazing technique described
in Propositions 6, 7 and 8 which reduces the investigation for the bounds (\ref{exp_bounds_V_i_beta_intro}) to the study of a
Cauchy-Kowalevski type problem (\ref{Aux_CP_formal}), (\ref{Aux_CP_formal_init_cond}) in several complex variables
for which local analytic solutions are found in Section 2.1, see Proposition 1. On the way we make use of estimates in weighted Banach spaces
introduced in Section 2.2, see Propositions 2, 3, 4 and Corollary 1, which are very much alike those already seen in the work \cite{lamasa}.

In the next step, for given suitable initial data (\ref{SCP_init_cond}) satisfying (\ref{defin_Y_Ud_j}), we construct
actual solutions
\begin{equation}
Y_{i}(T,z,x,\epsilon) = \sum_{\underline{\beta} = (\beta_{0},\ldots,\beta_{l},\beta_{l+1}) \in \mathbb{N}^{l+2}}
Y_{i,\underline{\beta}}(T,\epsilon)
\frac{\exp( iz( \sum_{j=0}^{l} \beta_{j} \xi_{j} )) }{\beta_{0}! \cdots \beta_{l}!}\frac{x^{\beta_{l+1}}}{\beta_{l+1}!}
\end{equation} 
of the equation (\ref{SCP}), where each function $Y_{i,\underline{\beta}}(T,\epsilon)$ can be written as a Laplace transform
of the function $V_{i,\underline{\beta}}(\tau,\epsilon)$ with respect to $\tau$ along a halfline
$L_{\gamma_i} = \mathbb{R}_{+}e^{i \sqrt{-1}\gamma_{i}} \subset U_{i} \cup \{ 0 \}$. For each $\epsilon \in \mathcal{E}_{i}$,
the function $T \mapsto Y_{i,\underline{\beta}}(T,\epsilon)$ is bounded and holomorphic on
a sector $U_{i,\epsilon}$ with aperture larger than $\pi$, with bisecting direction $\gamma_{i}$ and with radius
$h'|\epsilon|^{r_{3}/r_{2}}$ for some constant $h'>0$. In Proposition 11, we show that the function $Y_{i}$ itself turns out to
define a holomorphic function on $U_{i,\epsilon} \times H_{\rho_{1}'} \times D(0,\rho_{1})$ for some $0< \rho_{1} < \rho$ and where
$\rho_{1}'>0$ satisfies (\ref{choice_M0_rho_1_prime}).

We observe that, for all $0 \leq i \leq \nu-1$, the functions $X_{i}$ defined as
$$ X_{i}(t,z,x,\epsilon) = Y_{i}(\epsilon^{r_{3}/r_{2}}t,z,x,\epsilon)$$ 
actually solve our initial Cauchy problem (\ref{main_CP_intro}), (\ref{main_CP_init_cond_intro}) on the products
$\mathcal{T} \times H_{\rho_{1}'} \times D(0,\rho_{1}) \times \mathcal{E}_{i}$ and bear the representation
(\ref{quasiperiodic_X_i_intro}) as a quasiperiodic function whose Fourier coefficients decay exponentially in $\underline{\beta}$. 
It is worthy to mention that spaces of quasiperiodic Fourier series with exponential decay were also recently used in the paper \cite{dago}
in order to find global in time and quasiperiodic in space solutions to the KdV equation. 

In Proposition 12, we show moreover that the difference of two neighboring solutions $X_{i+1}$ and $X_{i}$ has exponentially small bounds
of order $\frac{r_{3}}{hr_{1}+r_{2}}$, uniformly in $(t,z,x)$, as $\epsilon$ tends to 0 on $\mathcal{E}_{i+1} \cap \mathcal{E}_{i}$.
We observe that for each $\underline{\beta} \in \mathbb{N}^{l+2}$ the difference $X_{i+1,\underline{\beta}}-X_{i,\underline{\beta}}$
for the Fourier coefficients has exponential decay of
order $r_{3}/r_{2}$ but its type is proportional to $\rho_{\underline{\beta}}$ and therefore tends to 0 as $\underline{\beta}$ tends to infinity.
This small denominator phenomenon is the reason of the decreasement of the order $r_{3}/r_{2}$ to $\frac{r_{3}}{hr_{1}+r_{2}}$.
As in our previous study \cite{lamasa}, the bulk of the proof rests on a thorough estimation of a Dirichlet like series
of the form $\sum_{k \geq 0} e^{-1/(k+1)^{\alpha}\epsilon^{r}} a^{k}$ for $0 < a < 1$ and $\alpha,r > 0$ with $\epsilon>0$ small.
This kind of series appears in the context of almost periodic functions introduced by H. Bohr, see for instance the textbook \cite{co}. 
These estimates (\ref{|X_i+1_minus_X_i|<}) are crucial in order to apply a cohomological criterion known in the litterature as
the Ramis-Sibuya theorem (Theorem {\bf(RS)}) which leads to the main result of this paper namely the existence of a formal series
$$ \hat{X}(t,z,x,\epsilon) = \sum_{k \geq 0} H_{k} \frac{\epsilon^k}{k!} $$
with coefficients $H_{k}$ in the Banach space of holomorphic and bounded functions on
$\mathcal{T} \times H_{\rho_{1}'} \times D(0,\rho_{1})$, which formally solves the equation (\ref{main_CP_intro}) and which is,
moreover, the Gevrey asymptotic expansion of order $\frac{hr_{1}+r_{2}}{r_{3}}$ of the functions $X_{i}$ on $\mathcal{E}_{i}$, for all
$0 \leq i \leq \nu-1$.\medskip

\noindent The layout of the paper reads as follows.\\
Section 2.1 is dedicated to the study of a version of the Cauchy Kowalevski theorem for nonlinear PDEs in analytic spaces of functions with
precise control on the domain of existence of their solutions in term of norm estimates of the initial data. In Section 2.2 we establish
some continuity properties of several integro-differential and multiplication operators acting on weighted Banach spaces of holomorphic
functions. These results are applied in Section 2.3 when looking for global solutions with growth constraint at
infinity for a parameter depending nonlinear convolution differential Cauchy problem with singular coefficients.

We recall briefly the classical theory concerning the Borel-Laplace transform and we show some commutation formulas with multiplication and
integro-differential operators in Section 3.1, then we center our attention on finding solutions of an auxiliary nonlinear Cauchy problem
obtained by the linear change of variable $T \mapsto T/\epsilon^{r}$ from our main Cauchy problem in Section 3.2.
The link between this Cauchy problem and the one solved in Section 2.3 is performed by means of Borel-Laplace transforms on the
corresponding solutions.

In Section 4.1, we construct actual holomorphic solutions $X_{i}$, $0 \leq i \leq \nu-1$ of our initial problem and we show exponential
decay of the difference of any two of these solutions with respect to $\epsilon$ on the intersection of their domain of definition, uniformly
in the other variables. Finally, in Section 4.2, we conclude with the main result of the work, that is the existence of a formal power series
with coefficients in an appropriate Banach space, which asymptotically represents the functions $X_{i}$ with a precise control on the
Gevrey order on the sectors $\mathcal{E}_{i}$, for all $0 \leq i \leq \nu-1$.
 
\section{A Global Cauchy problem in holomorphic and quasiperiodic function spaces}

\subsection{A Cauchy Kowalevski theorem in several variables}

In this section, we recall the well known Cauchy Kowaleski theorem in some spaces of analytic functions for which the size
of the domain of existence of the solution can be controlled in term of some supremum norm of the initial data.\medskip

The next Banach spaces are natural extensions to the several variables case of the spaces used in \cite{ma2}.

\begin{defin} Let $l \geq 1$ be some integer. Let $\bar{Z}_{0},\ldots,\bar{Z}_{l},\bar{X}>0$ be positive real numbers. We denote
$G(\bar{Z}_{0},\ldots,\bar{Z}_{l},\bar{X})$ the space of formal series
\begin{equation}
U(Z_{0},\ldots,Z_{l},X) = \sum_{\underline{\beta} = (\beta_{0},\ldots,\beta_{l},\beta_{l+1}) \in
\mathbb{N}^{l+2}} u_{\underline{\beta}} \frac{Z_{0}^{\beta_0} \cdots Z_{l}^{\beta_l}
X^{\beta_{l+1}}}{\beta_{0}! \cdots \beta_{l}!\beta_{l+1}!} \label{U_shape}
\end{equation}
that belong to $\mathbb{C}[[Z_{0},\ldots,Z_{l},X]]$ such that
$$ ||U(Z_{0},\ldots,Z_{l},X)||_{(\bar{Z}_{0},\ldots,\bar{Z}_{l},\bar{X})} =
\sum_{\underline{\beta} = (\beta_{0},\ldots,\beta_{l},\beta_{l+1}) \in
\mathbb{N}^{l+2}} |u_{\underline{\beta}}| \frac{\bar{Z}_{0}^{\beta_0} \cdots \bar{Z}_{l}^{\beta_l}
\bar{X}^{\beta_{l+1}}}{(\sum_{j=0}^{l+1} \beta_{j})!} $$
is finite. One can show that $G(\bar{Z}_{0},\ldots,\bar{Z}_{l},\bar{X})$ equipped with the norm
$||.||_{(\bar{Z}_{0},\ldots,\bar{Z}_{l},\bar{X})}$ are Banach spaces.
\end{defin}

In the next two lemmas we show continuity properties for some linear integro-differential operators acting on the
aforementioned Banach spaces.

\begin{lemma} Let $h_{0},\ldots,h_{l},h_{l+1} \in \mathbb{N}$ with
\begin{equation}
h_{l+1} \geq \sum_{j=0}^{l} h_{j} \label{h_l+1_larger_h_l}.
\end{equation}
Then, for any given $\bar{Z}_{0},\ldots,\bar{Z}_{l},\bar{X}>0$, the operator
$\partial_{Z_0}^{h_0} \cdots \partial_{Z_l}^{h_l}\partial_{X}^{-h_{l+1}}$ is a bounded linear map from
$G(\bar{Z}_{0},\ldots,\bar{Z}_{l},\bar{X})$ into itself. Moreover,
\begin{multline}
||\partial_{Z_0}^{h_0} \cdots \partial_{Z_l}^{h_l}
\partial_{X}^{-h_{l+1}}U(Z_{0},\ldots,Z_{l},X)||_{(\bar{Z}_{0},\ldots,\bar{Z}_{l},\bar{X})} \\
\leq \bar{Z}_{0}^{-h_0} \cdots \bar{Z}_{l}^{-h_l}\bar{X}^{h_{l+1}}
||U(Z_{0},\ldots,Z_{l},X)||_{(\bar{Z}_{0},\ldots,\bar{Z}_{l},\bar{X})}
\label{norm_diff_int_U<norm_U}
\end{multline}
for all $U(Z_{0},\ldots,Z_{l},X) \in G(\bar{Z}_{0},\ldots,\bar{Z}_{l},\bar{X})$.
\end{lemma}
\begin{proof} Let $U(Z_{0},\ldots,Z_{l},X) \in G(\bar{Z}_{0},\ldots,\bar{Z}_{l},\bar{X})$ of the form (\ref{U_shape}).
By definition, we can write
\begin{multline}
||\partial_{Z_0}^{h_0} \cdots \partial_{Z_l}^{h_l}
\partial_{X}^{-h_{l+1}}U(Z_{0},\ldots,Z_{l},X)||_{(\bar{Z}_{0},\ldots,\bar{Z}_{l},\bar{X})}\\
= \sum_{\beta_{0},\ldots,\beta_{l},\beta_{l+1} \geq 0}
 \left( \frac{(\sum_{j=0}^{l} \beta_{j} + h_{j} + \beta_{l+1} - h_{l+1})!}{(\sum_{j=0}^{l+1} \beta_{j})!}
\times \bar{Z}_{0}^{-h_0} \cdots \bar{Z}_{l}^{-h_l}\bar{X}^{h_{l+1}} \right)\\
\frac{|U_{\beta_{0}+h_{0},\ldots,\beta_{l}+h_{l},\beta_{l+1}-h_{l+1}}|}{(\sum_{j=0}^{l} \beta_{j}+h_{j} + \beta_{l+1}-h_{l+1})!}
\bar{Z}_{0}^{\beta_{0}+h_{0}} \cdots \bar{Z}_{l}^{\beta_{l}+h_{l}} \bar{X}^{\beta_{l+1}-h_{l+1}}
\end{multline}
Since (\ref{h_l+1_larger_h_l}), we know that
\begin{equation}
\frac{(\sum_{j=0}^{l} \beta_{j} + h_{j} + \beta_{l+1} - h_{l+1})!}{(\sum_{j=0}^{l+1} \beta_{j})!} \leq 1
\end{equation}
for all $\beta_{0},h_{0},\ldots,\beta_{l},h_{l},\beta_{l+1},h_{l+1} \geq 0$. The estimates (\ref{norm_diff_int_U<norm_U}) follows.
\end{proof}

\begin{lemma} Let $h_{0},\ldots,h_{l},h_{l+1} \in \mathbb{N}$. Let, for all $0 \leq j \leq l$,
$0 < \bar{Z}_{j}^{1} < \bar{Z}_{j}^{0}$ and $0 < \bar{X}^{1} < \bar{X}^{0}$ be positive real numbers. Then, there exists
$C_{2}>0$ (depending on $h_{0},\ldots,h_{l+1}$,$\bar{X}^{0},\bar{X}^{1}$,$\bar{Z}_{j}^{0},\bar{Z}_{j}^{1}$
for $0 \leq j \leq l$) such that
\begin{multline}
|| \partial_{Z_0}^{h_0} \cdots \partial_{Z_l}^{h_l} \partial_{X}^{h_{l+1}}
U(Z_{0},\ldots,Z_{l},X) ||_{(\bar{Z}_{0}^{1},\ldots,\bar{Z}_{l}^{1},\bar{X}^{1})}\\
\leq C_{2}(\bar{Z}_{0}^{1})^{-h_0} \cdots (\bar{Z}_{l}^{1})^{-h_l}(\bar{X}^{1})^{-h_{l+1}}||
U(Z_{0},\ldots,Z_{l},X) ||_{(\bar{Z}_{0}^{0},\ldots,\bar{Z}_{l}^{0},\bar{X}^{0})}
\end{multline}
for all $U(Z_{0},\ldots,Z_{l},X) \in G(\bar{Z}_{0}^{0},\ldots,\bar{Z}_{l}^{0},\bar{X}^{0})$.
\end{lemma}
\begin{proof} Let $U(Z_{0},\ldots,Z_{l},X) \in G(\bar{Z}_{0}^{0},\ldots,\bar{Z}_{l}^{0},\bar{X}^{0})$ be of the form (\ref{U_shape}).
By definition, we can write
\begin{multline}
|| \partial_{Z_0}^{h_0} \cdots \partial_{Z_l}^{h_l} \partial_{X}^{h_{l+1}}
U(Z_{0},\ldots,Z_{l},X) ||_{(\bar{Z}_{0}^{1},\ldots,\bar{Z}_{l}^{1},\bar{X}^{1})}\\
= \sum_{\beta_{0},\ldots,\beta_{l},\beta_{l+1} \geq 0}
\frac{ |u_{\beta_{0}+h_{0},\ldots,\beta_{l}+h_{l},\beta_{l+1}+h_{l+1}}|}{(\sum_{j=0}^{l+1} \beta_{j} + h_{j})!}
\times  (\bar{Z}_{0}^{1})^{-h_0} \cdots (\bar{Z}_{l}^{1})^{-h_l} (\bar{X}^{1})^{-h_{l+1}}\\
\times \left(  \frac{ (\sum_{j=0}^{l+1} \beta_{j} + h_{j})! }{(\sum_{j=0}^{l+1} \beta_{j})!}
\times
(\bar{Z}_{0}^{1})^{\beta_{0}+h_{0}} \cdots (\bar{Z}_{l}^{1})^{\beta_{l}+h_{l}}(\bar{X}^{1})^{\beta_{l+1}+h_{l+1}} \right)
\label{norm_diff_U_defin_factorize}
\end{multline}
Now, we take some real number $0 < \delta < 1$ such that $\bar{Z}_{j}^{1} < \delta \bar{Z}_{j}^{0}$ for all $0 \leq j \leq l$ and
$\bar{X}^{1} < \delta \bar{X}^{0}$. We get the estimates
\begin{multline}
\frac{ (\sum_{j=0}^{l+1} \beta_{j} + h_{j})! }{(\sum_{j=0}^{l+1} \beta_{j})!}(\bar{Z}_{0}^{1})^{\beta_{0}+h_{0}} \cdots
(\bar{Z}_{l}^{1})^{\beta_{l}+h_{l}}(\bar{X}^{1})^{\beta_{l+1}+h_{l+1}}\\
\leq (\sum_{j=0}^{l+1} \beta_{j} + \sum_{j=0}^{l+1} h_{j})^{(\sum_{j=0}^{l+1} h_{j})}
\delta^{(\sum_{j=0}^{l+1} \beta_{j} + \sum_{j=0}^{l+1} h_{j})}
(\bar{Z}_{0}^{0})^{\beta_{0}+h_{0}} \cdots (\bar{Z}_{l}^{0})^{\beta_{l}+h_{l}}(\bar{X}^{0})^{\beta_{l+1}+h_{l+1}}
\label{maj_quot_factorial_times_geom}
\end{multline}
for all $\beta_{0},h_{0},\ldots,\beta_{l},h_{l},\beta_{l+1},h_{l+1} \geq 0$. Since $0 < \delta < 1$, we know that for any real
number $a>0$, there exists $K_{\delta,a}>0$ depending on $a,\delta$ such that
$$\sup_{n \geq 0} (n+a)^{a} \delta^{n+a} \leq K_{\delta,a}. $$
Hence, we get a constant $C_{2}>0$ depending on $h_{j}$, $0 \leq j \leq l+1$ and $\delta$ with
\begin{multline}
(\sum_{j=0}^{l+1} \beta_{j} + \sum_{j=0}^{l+1} h_{j})^{(\sum_{j=0}^{l+1} h_{j})}
\delta^{(\sum_{j=0}^{l+1} \beta_{j} + \sum_{j=0}^{l+1} h_{j})}
(\bar{Z}_{0}^{0})^{\beta_{0}+h_{0}} \cdots (\bar{Z}_{l}^{0})^{\beta_{l}+h_{l}}(\bar{X}^{0})^{\beta_{l+1}+h_{l+1}}\\
\leq C_{2} (\bar{Z}_{0}^{0})^{\beta_{0}+h_{0}} \cdots (\bar{Z}_{l}^{0})^{\beta_{l}+h_{l}}(\bar{X}^{0})^{\beta_{l+1}+h_{l+1}}
\label{polyn_geom_comp}
\end{multline}
for all $\beta_{0},h_{0},\ldots,\beta_{l},h_{l},\beta_{l+1},h_{l+1} \geq 0$. As a result,
gathering (\ref{norm_diff_U_defin_factorize}), (\ref{maj_quot_factorial_times_geom}) and (\ref{polyn_geom_comp}), we get the
lemma. 
\end{proof}

\begin{lemma} Let $U_{1}(Z_{0},\ldots,Z_{l},X),U_{2}(Z_{0},\ldots,Z_{l},X) \in G(\bar{Z}_{0},\ldots,\bar{Z}_{l},\bar{X})$. Then,
the product $U_{1}U_{2}$ belongs to $G(\bar{Z}_{0},\ldots,\bar{Z}_{l},\bar{X})$. Moreover, we have
\begin{multline}
|| U_{1}(Z_{0},\ldots,Z_{l},X)U_{2}(Z_{0},\ldots,Z_{l},X) ||_{(\bar{Z}_{0},\ldots,\bar{Z}_{l},\bar{X})} \leq\\
|| U_{1}(Z_{0},\ldots,Z_{l},X) ||_{(\bar{Z}_{0},\ldots,\bar{Z}_{l},\bar{X})}
|| U_{2}(Z_{0},\ldots,Z_{l},X) ||_{(\bar{Z}_{0},\ldots,\bar{Z}_{l},\bar{X})}. \label{norm_U1U2<norm_U1_prod_norm_U2}
\end{multline}
In other words, the space $G(\bar{Z}_{0},\ldots,\bar{Z}_{l},\bar{X})$ is a Banach algebra.
\end{lemma}
\begin{proof} Let
$$ U_{j}(Z_{0},\ldots,Z_{l},X) = \sum_{\beta_{0},\ldots,\beta_{l},\beta_{l+1} \geq 0}
u_{\beta}^{j} \frac{Z_{0}^{\beta_0} \cdots Z_{l}^{\beta_l} X^{\beta_{l+1}}}{\beta_{0}! \cdots \beta_{l}! \beta_{l+1}!} $$
belonging to $G(\bar{Z}_{0},\ldots,\bar{Z}_{l},\bar{X})$ for $j=1,2$. By definition, we can write
\begin{multline}
||U_{1}(Z_{0},\ldots,Z_{l},X)U_{2}(Z_{0},\ldots,Z_{l},X)||_{(\bar{Z}_{0},\ldots,\bar{Z}_{l},\bar{X})}\\
\leq \sum_{\beta_{0},\ldots,\beta_{l},\beta_{l+1} \geq 0} (\sum_{\beta_{j}^{1}+\beta_{j}^{2} = \beta_{j},0 \leq j \leq l+1}
\left( \frac{ \Pi_{j=0}^{l+1} \beta_{j}! }{\Pi_{j=0}^{l+1} \beta_{j}^{1} \Pi_{j=0}^{l+1} \beta_{j}^{2}! }
\frac{1}{(\sum_{j=0}^{l+1} \beta_{j})!} \right)
|u_{(\beta_{0}^{1},\ldots,\beta_{l}^{1},\beta_{l+1}^{1})}^{1} u_{(\beta_{0}^{2},\ldots,\beta_{l}^{2},\beta_{l+1}^{2})}^{2}|)\\
\times \bar{Z}_{0}^{\beta_0} \cdots \bar{Z}_{l}^{\beta_l}\bar{X}^{\beta_{l+1}} \label{maj_defin_norm_U1U2}
\end{multline}
Besides, using the identity $(1+x)^{\sum_{j=0}^{l+1} \beta_{j}} = \Pi_{j=0}^{l+1}(1+x)^{\beta_{j}}$ and the binomial formula, we
get that
\begin{equation}
\frac{ \Pi_{j=0}^{l+1} \beta_{j}! }{\Pi_{j=0}^{l+1} \beta_{j}^{1}! \Pi_{j=0}^{l+1} \beta_{j}^{2}! } \leq
\frac{(\sum_{j=0}^{l+1} \beta_{j})!}{(\sum_{j=0}^{l+1} \beta_{j}^{1})!(\sum_{j=0}^{l+1} \beta_{j}^{2})!} \label{maj_binom}
\end{equation}
for all integers $\beta_{j},\beta_{j}^{1},\beta_{j}^{2} \geq 0$ such that $\beta_{j}=\beta_{j}^{1}+\beta_{j}^{2}$, for $0 \leq j \leq l+1$.
Therefore, the inequality (\ref{norm_U1U2<norm_U1_prod_norm_U2}) follows from (\ref{maj_defin_norm_U1U2}) and (\ref{maj_binom}).
\end{proof}

In the next proposition, we state a version of the Cauchy Kowalevski theorem.

\begin{prop} Let $\mathcal{D}_{1}$ (resp. $\mathcal{D}_{2}$) be a finite subset of $\mathbb{N}^{l+2}$ (resp. $\mathbb{N}^{l+1}$)
and let $S \geq 1$ be an integer such that for all
$(k_{0},\ldots,k_{l+1}) \in \mathcal{D}_{1}$ and all $(p_{0},\ldots,p_{l}) \in \mathcal{D}_{2}$, we have
\begin{equation}
S > k_{l+1} \ \ , \ \ S \geq \sum_{j=0}^{l+1} k_{j} \ \ , \ \ S \geq \sum_{j=0}^{l} p_{j}
\end{equation}
Let $\mathcal{D}_{3}$ be a finite subset of $\mathbb{N} \setminus \{0,1 \}$. Let $M^{0},\bar{X}_{0}>0$ be given real numbers and
let $\bar{Z}_{j}^{0}>M^{0}$, $0 \leq j \leq l$ be real numbers. Let
\begin{multline*}
d_{(k_{0},\ldots,k_{l+1})}(Z_{0},\ldots,Z_{l},X) \in G(\bar{Z}_{0}^{0},\ldots,\bar{Z}_{l}^{0},\bar{X}^{0}) \ \ , \ \ 
f_{(p_{0},\ldots,p_{l})}(Z_{0},\ldots,Z_{l}) \in G(\bar{Z}_{0}^{0},\ldots,\bar{Z}_{l}^{0},\bar{X}^{0}),\\
e_{m}(Z_{0},\ldots,Z_{l},X) \in G(\bar{Z}_{0}^{0},\ldots,\bar{Z}_{l}^{0},\bar{X}^{0})
\end{multline*}
for all $(k_{0},\ldots,k_{l+1}) \in \mathcal{D}_{1}$, all $(p_{0},\ldots,p_{l}) \in \mathcal{D}_{2}$ and
all $m \in \mathcal{D}_{3}$. For all $0 \leq j \leq S-1$, we also choose
$\varphi_{j}(Z_{0},\ldots,Z_{l}) \in G(\bar{Z}_{0}^{0},\ldots,\bar{Z}_{l}^{0},\bar{X}^{0})$.

We consider the following Cauchy problem
\begin{multline}
\partial_{X}^{S}U(Z_{0},\ldots,Z_{l},X) = \sum_{\underline{k} = (k_{0},\ldots,k_{l},k_{l+1}) \in \mathcal{D}_{1}}
d_{\underline{k}}(Z_{0},\ldots,Z_{l},X) \partial_{Z_0}^{k_0} \cdots \partial_{Z_l}^{k_l} \partial_{X}^{k_{l+1}}
U(Z_{0},\ldots,Z_{l},X) \\
+ \sum_{\underline{p} = (p_{0},\ldots,p_{l}) \in \mathcal{D}_{2}}
f_{\underline{p}}(Z_{0},\ldots,Z_{l}) \partial_{Z_0}^{p_0} \cdots \partial_{Z_l}^{p_l}
(\sum_{m \in \mathcal{D}_{3}} e_{m}(Z_{0},\ldots,Z_{l},X)(U(Z_{0},\ldots,Z_{l},X))^{m}) \label{CP_CK}
\end{multline}
for given initial data
\begin{equation}
(\partial_{X}^{j}U)(Z_{0},\ldots,Z_{l},0) = \varphi_{j}(Z_{0},\ldots,Z_{l}) \ \ , \ \ 0 \leq j \leq S-1. \label{CP_CK_init_cond}
\end{equation}
Then, for given real numbers $\bar{Z}_{j}^{1}>0$, $0 \leq j \leq l$, with $M^{0} < \bar{Z}_{j}^{1} < \bar{Z}_{j}^{0}$,
one can choose $0 < \bar{X}^{1} < \bar{X}^{0}$ and $\delta>0$ (which depend on $\bar{Z}_{j}^{1}$ for $0 \leq j \leq l$, on
$||d_{\underline{k}}(Z_{0},\ldots,Z_{l},X)||_{(\bar{Z}_{0}^{0},\ldots,\bar{Z}_{l}^{0},\bar{X}^{0})}$ for
$\underline{k} \in \mathcal{D}_{1}$, on
$||f_{\underline{p}}(Z_{0},\ldots,Z_{l})||_{(\bar{Z}_{0}^{0},\ldots,\bar{Z}_{l}^{0},\bar{X}^{0})}$ for
$\underline{p} \in \mathcal{D}_{2}$ and on
$||e_{m}(Z_{0},\ldots,Z_{l},X)||_{(\bar{Z}_{0}^{0},\ldots,\bar{Z}_{l}^{0},\bar{X}^{0})}$ for
$m \in \mathcal{D}_{3}$) such that
if
\begin{equation}
|| \varphi_{j}(Z_{0},\ldots,Z_{l})||_{(\bar{Z}_{0}^{0},\ldots,\bar{Z}_{l}^{0},\bar{X}^{0})} < \delta \ \ , \ \ 0 \leq j \leq S-1,
\label{init_cond_small}
\end{equation}
then the problem (\ref{CP_CK}), (\ref{CP_CK_init_cond}) has a unique solution
$U(Z_{0},\ldots,Z_{l},X) \in G(\bar{Z}_{0}^{1},\ldots,\bar{Z}_{l}^{1},\bar{X}^{1})$. Moreover, there exists a constant $C_{3}>0$,
depending on $\bar{X}^{0},\bar{X}^{1}$,$\bar{Z}_{j}^{0},\bar{Z}_{j}^{1}$ for $0 \leq j \leq l$, on\\
$||d_{\underline{k}}(Z_{0},\ldots,Z_{l},X)||_{(\bar{Z}_{0}^{0},\ldots,\bar{Z}_{l}^{0},\bar{X}^{0})}$ for
$\underline{k} \in \mathcal{D}_{1}$, on
$||f_{\underline{p}}(Z_{0},\ldots,Z_{l})||_{(\bar{Z}_{0}^{0},\ldots,\bar{Z}_{l}^{0},\bar{X}^{0})}$ for
$\underline{p} \in \mathcal{D}_{2}$ and on
$||e_{m}(Z_{0},\ldots,Z_{l},X)||_{(\bar{Z}_{0}^{0},\ldots,\bar{Z}_{l}^{0},\bar{X}^{0})}$ for
$m \in \mathcal{D}_{3}$, such that
\begin{equation}
||U(Z_{0},\ldots,Z_{l},X)||_{(\bar{Z}_{0}^{1},\ldots,\bar{Z}_{l}^{1},\bar{X}^{1})} \leq \delta C_{3}. \label{estim_sol_CK} 
\end{equation}
\end{prop}
\begin{proof} We put
$$ w(Z_{0},\ldots,Z_{l},X) = \sum_{j=0}^{S-1} \varphi_{j}(Z_{0},\ldots,Z_{l}) \frac{X^j}{j!}$$
and we consider the map $A : \mathbb{C}[[Z_{0},\ldots,Z_{l},X]] \rightarrow \mathbb{C}[[Z_{0},\ldots,Z_{l},X]]$ defined as
\begin{multline*}
A(H(Z_{0},\ldots,Z_{l},X)) = \sum_{\underline{k} = (k_{0},\ldots,k_{l},k_{l+1}) \in \mathcal{D}_{1}}
d_{\underline{k}}(Z_{0},\ldots,Z_{l},X) \partial_{Z_0}^{k_0} \cdots \partial_{Z_l}^{k_l}\partial_{X}^{k_{l+1}-S}
H(Z_{0},\ldots,Z_{l},X) \\
+ \sum_{\underline{k} = (k_{0},\ldots,k_{l},k_{l+1}) \in \mathcal{D}_{1}}
d_{\underline{k}}(Z_{0},\ldots,Z_{l},X) \partial_{Z_0}^{k_0}\cdots \partial_{Z_l}^{k_l}\partial_{X}^{k_{l+1}}w(Z_{0},\ldots,Z_{l},X)\\
+ \sum_{\underline{p} = (p_{0},\ldots,p_{l}) \in \mathcal{D}_{2}} 
f_{\underline{p}}(Z_{0},\ldots,Z_{l}) \partial_{Z_0}^{p_0} \cdots \partial_{Z_l}^{p_l}
(\sum_{m \in \mathcal{D}_{3}} e_{m}(Z_{0},\ldots,Z_{l},X)\\
\times ( \partial_{X}^{-S}H(Z_{0},\ldots,Z_{l},X) + w(Z_{0},\ldots,Z_{l},X) )^{m})
\end{multline*}

\begin{lemma} Let $\bar{Z}_{j}^{1}>0$, $0 \leq j \leq l$, be real numbers such that $M^{0} < \bar{Z}_{j}^{1} < Z_{j}^{0}$. Then,
there exist $\delta>0$, a real number $0 < \bar{X}^{1} < \bar{X}^{0}$ and a constant $K_{4}>0$
(depending on $\bar{Z}_{j}^{1}$ for $0 \leq j \leq l$, on
$||d_{\underline{k}}(Z_{0},\ldots,Z_{l},X)||_{(\bar{Z}_{0}^{0},\ldots,\bar{Z}_{l}^{0},\bar{X}^{0})}$ for
$\underline{k} \in \mathcal{D}_{1}$, on
$||f_{\underline{p}}(Z_{0},\ldots,Z_{l})||_{(\bar{Z}_{0}^{0},\ldots,\bar{Z}_{l}^{0},\bar{X}^{0})}$ for
$\underline{p} \in \mathcal{D}_{2}$ and on
$||e_{m}(Z_{0},\ldots,Z_{l},X)||_{(\bar{Z}_{0}^{0},\ldots,\bar{Z}_{l}^{0},\bar{X}^{0})}$ for
$m \in \mathcal{D}_{3}$), such that if one puts
$R=K_{4}\delta$,\\
{\bf i)} we have
\begin{equation}
A(B_{R}) \subset B_{R} \label{ABR_sub_BR}
\end{equation}
where $B_{R}$ is the closed ball of radius $R$, centered at 0 in
$G(\bar{Z}_{0}^{1},\ldots,\bar{Z}_{l}^{1},\bar{X}^{1})$.\\
{\bf ii)} For all $H_{1},H_{2} \in B_{R}$,
\begin{equation}
||A(H_{1}) - A(H_{2})||_{(\bar{Z}_{0}^{1},\ldots,\bar{Z}_{l}^{1},\bar{X}^{1})} \leq \frac{1}{2}
||H_{1} - H_{2}||_{(\bar{Z}_{0}^{1},\ldots,\bar{Z}_{l}^{1},\bar{X}^{1})}. \label{A_shrink}
\end{equation}
\end{lemma}
\begin{proof} We first show {\bf i)}. We fix $\bar{Z}_{j}^{1}>0$, $0 \leq j \leq l$, be real numbers such that
$M^{0} < \bar{Z}_{j}^{1} < Z_{j}^{0}$. We also consider $\delta>0$ for which (\ref{init_cond_small}) holds. Let $R>0$ be of the form
$R = K_{4}\delta$ for some constant $K_{4}>0$. We take $H \in B_{R}$ where
$B_{R}$ is the closed ball of radius $R$, centered at 0 in
$G(\bar{Z}_{0}^{1},\ldots,\bar{Z}_{l}^{1},\bar{X}^{1})$ for some real number
$0 < \bar{X}^{1} < \bar{X}^{0}$. From Lemma 1 and Lemma 3,
we get
\begin{multline}
|| d_{\underline{k}}(Z_{0},\ldots,Z_{l},X) \partial_{Z_0}^{k_0} \cdots \partial_{Z_l}^{k_l}\partial_{X}^{k_{l+1}-S}
H(Z_{0},\ldots,Z_{l},X) ||_{(\bar{Z}_{0}^{1},\ldots,\bar{Z}_{l}^{1},\bar{X}^{1})}\\
\leq || d_{\underline{k}}(Z_{0},\ldots,Z_{l},X) ||_{(\bar{Z}_{0}^{1},\ldots,\bar{Z}_{l}^{1},\bar{X}^{1})}
(\bar{Z}_{0}^{1})^{-k_0} \cdots (\bar{Z}_{l}^{1})^{-k_l} (\bar{X}^{1})^{S-k_{l+1}}\\
\times ||H(Z_{0},\ldots,Z_{l},X)||_{(\bar{Z}_{0}^{1},\ldots,\bar{Z}_{l}^{1},\bar{X}^{1})} \leq
|| d_{\underline{k}}(Z_{0},\ldots,Z_{l},X) ||_{(\bar{Z}_{0}^{0},\ldots,\bar{Z}_{l}^{0},\bar{X}^{0})}
(\bar{Z}_{0}^{1})^{-k_0} \cdots (\bar{Z}_{l}^{1})^{-k_l}\\
\times (\bar{X}^{1})^{S-k_{l+1}} R. \label{norm_d_int_diff_H<}
\end{multline}
Now, from Lemma 2 and Lemma 3, we get a constant $C_{4}>0$ (depending on $k_{0},\ldots,k_{l}$,
$\bar{X}^{0},\bar{X}^{1}$,$\bar{Z}_{j}^{0},\bar{Z}_{j}^{1}$, for $0 \leq j \leq l$) with
\begin{multline}
|| d_{\underline{k}}(Z_{0},\ldots,Z_{l},X) (\partial_{Z_0}^{k_0}\cdots\partial_{Z_l}^{k_l}\varphi_{j})(Z_{0},\ldots,Z_{l})\frac{X^q}{q!}||
_{(\bar{Z}_{0}^{1},\ldots,\bar{Z}_{l}^{1},\bar{X}^{1})}\\
\leq  || d_{\underline{k}}(Z_{0},\ldots,Z_{l},X) ||_{(\bar{Z}_{0}^{1},\ldots,\bar{Z}_{l}^{1},\bar{X}^{1})}
||\frac{X^q}{q!}||_{(\bar{Z}_{0}^{1},\ldots,\bar{Z}_{l}^{1},\bar{X}^{1})}C_{4}
(\bar{Z}_{0}^{1})^{-k_0} \cdots (\bar{Z}_{l}^{1})^{-k_l}\\
\times ||\varphi_{j}(Z_{0},\ldots,Z_{l})||_{(\bar{Z}_{0}^{0},\ldots,\bar{Z}_{l}^{0},\bar{X}^{0})} \leq
|| d_{\underline{k}}(Z_{0},\ldots,Z_{l},X) ||_{(\bar{Z}_{0}^{0},\ldots,\bar{Z}_{l}^{0},\bar{X}^{0})}
||\frac{X^q}{q!}||_{(\bar{Z}_{0}^{0},\ldots,\bar{Z}_{l}^{0},\bar{X}^{0})}\\
\times C_{4}(\bar{Z}_{0}^{1})^{-k_0} \cdots (\bar{Z}_{l}^{1})^{-k_l}\delta \label{norm_d_diff_varphi<}
\end{multline}
for all $0 \leq j,q \leq S-1$.

Using Leibniz formula, we deduce from Lemma 3, that
\begin{multline}
|| f_{\underline{p}}(Z_{0},\ldots,Z_{p}) \partial_{Z_0}^{p_0} \cdots \partial_{Z_l}^{p_l} \left( e_{m}(Z_{0},\ldots,Z_{l},X) \right. \\
\left. \times 
(\partial_{X}^{-S}H(Z_{0},\ldots,Z_{l},X) + w(Z_{0},\ldots,Z_{l},X) )^{m} \right) ||_{(\bar{Z}_{0}^{1},\ldots,\bar{Z}_{l}^{1},\bar{X}^{1})}\\
 \leq \sum_{\sum_{j=0}^{m} p_{h,j} = p_{h}, 0 \leq h \leq l} \frac{ p_{0}! \cdots p_{l}! }{ \Pi_{0 \leq j \leq m, 0 \leq h \leq l} p_{h,j}! }
|| f_{\underline{p}}(Z_{0},\ldots,Z_{p}) ||_{(\bar{Z}_{0}^{1},\ldots,\bar{Z}_{l}^{1},\bar{X}^{1})}\\
\times ||\partial_{Z_0}^{p_{0,0}} \cdots
\partial_{Z_l}^{p_{l,0}}e_{m}(Z_{0},\ldots,Z_{l},X)||_{(\bar{Z}_{0}^{1},\ldots,\bar{Z}_{l}^{1},\bar{X}^{1})}\\
\times \Pi_{j=1}^{m}||(\partial_{Z_0}^{p_{0,j}} \cdots \partial_{Z_l}^{p_{l,j}}) \left( \partial_{X}^{-S}H(Z_{0},\ldots,Z_{l},X) +
w(Z_{0},\ldots,Z_{l},X) \right)||_{(\bar{Z}_{0}^{1},\ldots,\bar{Z}_{l}^{1},\bar{X}^{1})}
\label{norm_fp_diff_em_H_w}
\end{multline}
By Lemma 2, one also gets a constant $C_{4.1}>0$ (depending on
$p_{0,0},\ldots,p_{l,0}$,$\bar{X}^{0},\bar{X}^{1}$,$\bar{Z}_{j}^{0},\bar{Z}_{j}^{1}$ for $0 \leq j \leq l$) with
\begin{multline}
||\partial_{Z_0}^{p_{0,0}} \cdots
\partial_{Z_l}^{p_{l,0}}e_{m}(Z_{0},\ldots,Z_{l},X)||_{(\bar{Z}_{0}^{1},\ldots,\bar{Z}_{l}^{1},\bar{X}^{1})}\\
 \leq
C_{4.1} (\bar{Z}_{0}^{1})^{-p_{0,0}} \cdots (\bar{Z}_{l}^{1})^{-p_{l,0}}
||e_{m}(Z_{0},\ldots,Z_{l},X)||_{(\bar{Z}_{0}^{0},\ldots,\bar{Z}_{l}^{0},\bar{X}^{0})} \label{norm_diff_em}
\end{multline}
By Lemma 1, one finds
\begin{multline}
||\partial_{Z_0}^{p_{0,j}} \cdots \partial_{Z_l}^{p_{l,j}}\partial_{X}^{-S}
H(Z_{0},\ldots,Z_{l},X)||_{(\bar{Z}_{0}^{1},\ldots,\bar{Z}_{l}^{1},\bar{X}^{1})}\\
\leq (\bar{Z}_{0}^{1})^{-p_{0,j}} \cdots (\bar{Z}_{l}^{1})^{-p_{l,j}}(\bar{X}^{1})^{S}
||H(Z_{0},\ldots,Z_{l},X)||_{(\bar{Z}_{0}^{1},\ldots,\bar{Z}_{l}^{1},\bar{X}^{1})}\\
\leq (\bar{Z}_{0}^{1})^{-p_{0,j}} \cdots (\bar{Z}_{l}^{1})^{-p_{l,j}}(\bar{X}^{1})^{S}R \label{norm_diff_int_H}
\end{multline}
Due to Lemma 2, we get a constant $C_{4.2}>0$ (depending on
$p_{0,j},\ldots,p_{l,j}$,$\bar{X}^{0},\bar{X}^{1}$,$\bar{Z}_{m}^{0},\bar{Z}_{m}^{1}$ for $0 \leq m \leq l$) such that
\begin{multline}
||\partial_{Z_0}^{p_{0,j}} \cdots \partial_{Z_l}^{p_{l,j}}
w(Z_{0},\ldots,Z_{l},X)||_{(\bar{Z}_{0}^{1},\ldots,\bar{Z}_{l}^{1},\bar{X}^{1})}\\
\leq \sum_{q=0}^{S-1} ||\partial_{Z_0}^{p_{0,j}} \cdots \partial_{Z_l}^{p_{l,j}}
\varphi_{q}(Z_{0},\ldots,Z_{l})||_{(\bar{Z}_{0}^{1},\ldots,\bar{Z}_{l}^{1},\bar{X}^{1})}
|| \frac{X^q}{q!} ||_{(\bar{Z}_{0}^{1},\ldots,\bar{Z}_{l}^{1},\bar{X}^{1})}\\
\leq C_{4.2} \sum_{q=0}^{S-1} (\bar{Z}_{0}^{1})^{-p_{0,j}} \cdots (\bar{Z}_{l}^{1})^{-p_{l,j}}
||\varphi_{q}(Z_{0},\ldots,Z_{l})||_{(\bar{Z}_{0}^{0},\ldots,\bar{Z}_{l}^{0},\bar{X}^{0})}
|| \frac{X^q}{q!} ||_{(\bar{Z}_{0}^{0},\ldots,\bar{Z}_{l}^{0},\bar{X}^{0})} \\
\leq C_{4.2} \sum_{q=0}^{S-1} (\bar{Z}_{0}^{1})^{-p_{0,j}} \cdots (\bar{Z}_{l}^{1})^{-p_{l,j}}
\delta || \frac{X^q}{q!} ||_{(\bar{Z}_{0}^{0},\ldots,\bar{Z}_{l}^{0},\bar{X}^{0})} \label{norm_diff_w}
\end{multline}

Now, we can choose $0 < \bar{X}^{1} < \bar{X}^{0}$, $\delta>0$ and the constant $K_{4}>0$ (recall that $R=K_{4}\delta$) in such a way
that
\begin{multline}
\sum_{\underline{k}=(k_{0},\ldots,k_{l+1}) \in \mathcal{D}_{1}}
||d_{\underline{k}}(Z_{0},\ldots,Z_{l},X) ||_{(\bar{Z}_{0}^{0},\ldots,\bar{Z}_{l}^{0},\bar{X}^{0})}
(\bar{Z}_{0}^{1})^{-k_0} \cdots (\bar{Z}_{l}^{1})^{-k_l}\\
\times \left( (\bar{X}^{1})^{S-k_{l+1}}K_{4}\delta +
\sum_{q=0}^{S-1} ||\frac{X^q}{q!}||_{(\bar{Z}_{0}^{0},\ldots,\bar{Z}_{l}^{0},\bar{X}^{0})}C_{4}\delta \right)\\
+ \sum_{\underline{p}=(p_{0},\ldots,p_{l}) \in \mathcal{D}_{2}} \sum_{m \in \mathcal{D}_{3}}
\sum_{\sum_{j=0}^{m} p_{h,j} = p_{h}, 0 \leq h \leq l} \frac{ p_{0}! \cdots p_{l}! }{ \Pi_{0 \leq j \leq m, 0 \leq h \leq l} p_{h,j}! }
|| f_{\underline{p}}(Z_{0},\ldots,Z_{l}) ||_{(\bar{Z}_{0}^{0},\ldots,\bar{Z}_{l}^{0},\bar{X}^{0})}\\
\times C_{4.1} (\bar{Z}_{0}^{1})^{-p_{0,0}} \cdots (\bar{Z}_{l}^{1})^{-p_{l,0}}
||e_{m}(Z_{0},\ldots,Z_{l},X)||_{(\bar{Z}_{0}^{0},\ldots,\bar{Z}_{l}^{0},\bar{X}^{0})}\\
\times \Pi_{j=1}^{m} \left(  (\bar{Z}_{0}^{1})^{-p_{0,j}} \cdots (\bar{Z}_{l}^{1})^{-p_{l,j}}(\bar{X}^{1})^{S}K_{4}\delta \right. \\
+ C_{4.2} \sum_{q=0}^{S-1} (\bar{Z}_{0}^{1})^{-p_{0,j}} \cdots (\bar{Z}_{l}^{1})^{-p_{l,j}}
\left. \delta || \frac{X^q}{q!} ||_{(\bar{Z}_{0}^{0},\ldots,\bar{Z}_{l}^{0},\bar{X}^{0})} \right) \leq K_{4}\delta
\label{stable_delta_K4_X}
\end{multline}

Hence, gathering (\ref{norm_d_int_diff_H<}), (\ref{norm_d_diff_varphi<}), (\ref{norm_fp_diff_em_H_w}), (\ref{norm_diff_em}),
(\ref{norm_diff_int_H}), (\ref{norm_diff_w}) and (\ref{stable_delta_K4_X}) yields the inclusion (\ref{ABR_sub_BR}).\medskip

We turn to the proof of {\bf ii)}. As above, we fix $\bar{Z}_{j}^{1}>0$, $0 \leq j \leq l$, be real numbers such that
$M^{0} < \bar{Z}_{j}^{1} < Z_{j}^{0}$. We also consider $\delta>0$ for which (\ref{init_cond_small}) holds. Let $R>0$ be of the form
$R = K_{4}\delta$ for some constant $K_{4}>0$. We take $H_{1},H_{2} \in B_{R}$ where
$B_{R}$ is the closed ball of radius $R$, centered at 0 in
$G(\bar{Z}_{0}^{1},\ldots,\bar{Z}_{l}^{1},\bar{X}^{1})$ for some real number
$0 < \bar{X}^{1} < \bar{X}^{0}$. From Lemma 1 and Lemma 3, we get that
\begin{multline}
|| d_{\underline{k}}(Z_{0},\ldots,Z_{l},X) \partial_{Z_0}^{k_0} \cdots \partial_{Z_l}^{k_l}\partial_{X}^{k_{l+1}-S}
(H_{2}(Z_{0},\ldots,Z_{l},X) - H_{1}(Z_{0},\ldots,Z_{l},X)) ||_{(\bar{Z}_{0}^{1},\ldots,\bar{Z}_{l}^{1},\bar{X}^{1})}\\
\leq || d_{\underline{k}}(Z_{0},\ldots,Z_{l},X) ||_{(\bar{Z}_{0}^{1},\ldots,\bar{Z}_{l}^{1},\bar{X}^{1})}
(\bar{Z}_{0}^{1})^{-k_0} \cdots (\bar{Z}_{l}^{1})^{-k_l} (\bar{X}^{1})^{S-k_{l+1}}\\
\times ||H_{2}(Z_{0},\ldots,Z_{l},X) - H_{1}(Z_{0},\ldots,Z_{l},X)||_{(\bar{Z}_{0}^{1},\ldots,\bar{Z}_{l}^{1},\bar{X}^{1})}\\
 \leq || d_{\underline{k}}(Z_{0},\ldots,Z_{l},X) ||_{(\bar{Z}_{0}^{0},\ldots,\bar{Z}_{l}^{0},\bar{X}^{0})}
(\bar{Z}_{0}^{1})^{-k_0} \cdots (\bar{Z}_{l}^{1})^{-k_l} (\bar{X}^{1})^{S-k_{l+1}}\\
\times ||H_{2}(Z_{0},\ldots,Z_{l},X) - H_{1}(Z_{0},\ldots,Z_{l},X)||_{(\bar{Z}_{0}^{1},\ldots,\bar{Z}_{l}^{1},\bar{X}^{1})}.
 \label{norm_d_int_diff_H1_minus_H2<}
\end{multline}
Using the identity $b^{m} - a^{m} = (b-a)\sum_{s=0}^{m-1} a^{s}b^{m-1-s}$ for any complex numbers $a,b$ and any integer $m \geq 2$,
we can write
\begin{multline}
(\partial_{X}^{-S}H_{1}(Z_{0},\ldots,Z_{l},X) + w(Z_{0},\ldots,Z_{l},X))^{m} \\
- \partial_{X}^{-S}H_{2}(Z_{0},\ldots,Z_{l},X) + w(Z_{0},\ldots,Z_{l},X))^{m}
= (\partial_{X}^{-S}H_{1} - \partial_{X}^{-S}H_{2}) \\
 \times \sum_{s=0}^{m-1} (\partial_{X}^{-S}H_{2}+w)^{s}
(\partial_{X}^{-S}H_{1}+w)^{m-1-s} \label{factor_int_H1_w_minus_int_H2_w}
\end{multline}
Using Leibniz formula, we deduce from Lemma 3 that
\begin{multline}
||f_{\underline{p}}(Z_{0},\ldots,Z_{l}) \partial_{Z_0}^{p_0} \cdots
\partial_{Z_l}^{p_l}(e_{m}(Z_{0},\ldots,Z_{l},X)(\partial_{X}^{-S}H_{1}-
\partial_{X}^{-S}H_{2})\\
\times (\sum_{s=0}^{m-1}(\partial_{X}^{-S}H_{2}+w)^{s}
(\partial_{X}^{-S}H_{1}+w)^{m-1-s})||_{(\bar{Z}_{0}^{1},\ldots,\bar{Z}_{l}^{1},\bar{X}^{1})}\\
\leq \sum_{\sum_{j=0}^{2} p_{h,j} = p_{h}, 0 \leq h \leq l}  \frac{ p_{0}! \cdots p_{l}! }{ \Pi_{0 \leq j \leq 2, 0 \leq h \leq l} p_{h,j}! }
|| f_{\underline{p}}(Z_{0},\ldots,Z_{l}) ||_{(\bar{Z}_{0}^{1},\ldots,\bar{Z}_{l}^{1},\bar{X}^{1})}\\
\times ||\partial_{Z_0}^{p_{0,0}} \cdots
\partial_{Z_l}^{p_{l,0}}e_{m}(Z_{0},\ldots,Z_{l},X)||_{(\bar{Z}_{0}^{1},\ldots,\bar{Z}_{l}^{1},\bar{X}^{1})}
\times ||\partial_{Z_0}^{p_{0,1}} \cdots \partial_{Z_l}^{p_{l,1}}( \partial_{X}^{-S}H_{1} -
\partial_{X}^{-S}H_{2} )||_{(\bar{Z}_{0}^{1},\ldots,\bar{Z}_{l}^{1},\bar{X}^{1})}\\
\times ||\partial_{Z_0}^{p_{0,2}} \cdots \partial_{Z_l}^{p_{l,2}}( \sum_{s=0}^{m-1} (\partial_{X}^{-S}H_{2}+w)^{s}
(\partial_{X}^{-S}H_{1}+w)^{m-1-s})||_{(\bar{Z}_{0}^{1},\ldots,\bar{Z}_{l}^{1},\bar{X}^{1})}
\label{norm_diff_em_H1_minus_H2}
\end{multline}
By Lemma 2, one also gets a constant $C_{4.3}>0$ (depending on
$p_{0,0},\ldots,p_{l,0}$,$\bar{X}^{0},\bar{X}^{1}$,$\bar{Z}_{j}^{0},\bar{Z}_{j}^{1}$ for $0 \leq j \leq l$) with
\begin{multline}
||\partial_{Z_0}^{p_{0,0}} \cdots
\partial_{Z_l}^{p_{l,0}}e_{m}(Z_{0},\ldots,Z_{l},X)||_{(\bar{Z}_{0}^{1},\ldots,\bar{Z}_{l}^{1},\bar{X}^{1})}\\
 \leq
C_{4.3} (\bar{Z}_{0}^{1})^{-p_{0,0}} \cdots (\bar{Z}_{l}^{1})^{-p_{l,0}}
||e_{m}(Z_{0},\ldots,Z_{l},X)||_{(\bar{Z}_{0}^{0},\ldots,\bar{Z}_{l}^{0},\bar{X}^{0})} \label{norm_diff_em_part2}
\end{multline}
By Lemma 1, one finds
\begin{multline}
||\partial_{Z_0}^{p_{0,1}} \cdots \partial_{Z_l}^{p_{l,1}}(\partial_{X}^{-S}
H_{1} - \partial_{X}^{-S}H_{2})||_{(\bar{Z}_{0}^{1},\ldots,\bar{Z}_{l}^{1},\bar{X}^{1})}\\
\leq (\bar{Z}_{0}^{1})^{-p_{0,1}} \cdots (\bar{Z}_{l}^{1})^{-p_{l,1}}(\bar{X}^{1})^{S}
||H_{1}-H_{2}||_{(\bar{Z}_{0}^{1},\ldots,\bar{Z}_{l}^{1},\bar{X}^{1})} \label{norm_diff_int_H1_minus_H2}
\end{multline}
Using again Leibniz formula and Lemma 3, we can write
\begin{multline}
||\partial_{Z_0}^{p_{0,2}} \cdots \partial_{Z_l}^{p_{l,2}}((\partial_{X}^{-S}H_{2}+w)^{s}
(\partial_{X}^{-S}H_{1}+w)^{m-1-s})||_{(\bar{Z}_{0}^{1},\ldots,\bar{Z}_{l}^{1},\bar{X}^{1})}\\
\leq \sum_{\sum_{j=1}^{m-1} p_{h,2,j} = p_{h,2} , 0 \leq h \leq l}
\frac{ p_{0,2}! \cdots p_{l,2}! }{ \Pi_{1 \leq j \leq m-1,0 \leq h \leq l} p_{h,2,j}! } \\
\times \Pi_{j=1}^{s} ||(\partial_{Z_0}^{p_{0,2,j}} \cdots \partial_{Z_l}^{p_{l,0,j}})(\partial_{X}^{-S}H_{2}
+ w)||_{(\bar{Z}_{0}^{1},\ldots,\bar{Z}_{l}^{1},\bar{X}^{1})}\\
\times \Pi_{j=s+1}^{m-1} ||(\partial_{Z_0}^{p_{0,2,j}} \cdots \partial_{Z_l}^{p_{l,0,j}})(\partial_{X}^{-S}H_{1}
+ w)||_{(\bar{Z}_{0}^{1},\ldots,\bar{Z}_{l}^{1},\bar{X}^{1})}
\label{norm_diff_power_H1_H2_w}
\end{multline}
By Lemma 1 and Lemma 2, one finds a constant $C_{4.4}>0$ (depending on
$p_{0,2,j},\ldots,p_{l,2,j}$,$\bar{X}^{0},\bar{X}^{1}$,\\$\bar{Z}_{j}^{0},\bar{Z}_{j}^{1}$ for $0 \leq j \leq l$) such that
\begin{multline}
||(\partial_{Z_0}^{p_{0,2,j}} \cdots \partial_{Z_l}^{p_{l,2,j}})(\partial_{X}^{-S}H_{r}
+ w)||_{(\bar{Z}_{0}^{1},\ldots,\bar{Z}_{l}^{1},\bar{X}^{1})}
\leq (\bar{Z}_{0}^{1})^{-p_{0,2,j}} \cdots (\bar{Z}_{l}^{1})^{-p_{l,2,j}}(\bar{X}^{1})^{S}R \\
+ C_{4.4} \sum_{q=0}^{S-1} (\bar{Z}_{0}^{1})^{-p_{0,2,j}} \cdots (\bar{Z}_{l}^{1})^{-p_{l,2,j}} \delta
||\frac{X^q}{q!}||_{(\bar{Z}_{0}^{0},\ldots,\bar{Z}_{l}^{0},\bar{X}^{0})}
\label{norm_diff_Hr_w}
\end{multline}
for $r=1,2$.

In the following, we choose $0 < \bar{X}^{1} < \bar{X}^{0}$ in order that
\begin{multline}
\sum_{\underline{k} = (k_{0},\ldots,k_{l+1}) \in \mathcal{D}_{1}}
|| d_{\underline{k}}(Z_{0},\ldots,Z_{l},X) ||_{(\bar{Z}_{0}^{0},\ldots,\bar{Z}_{l}^{0},\bar{X}^{0})}
(\bar{Z}_{0}^{1})^{-k_0} \cdots (\bar{Z}_{l}^{1})^{-k_l} (\bar{X}^{1})^{S-k_{l+1}}\\
+ \sum_{\underline{p}=(p_{0},\ldots,p_{l}) \in \mathcal{D}_{2}} \sum_{m \in \mathcal{D}_{3}}
 \sum_{\sum_{j=0}^{2} p_{h,j} = p_{h}, 0 \leq h \leq l}  \frac{ p_{0}! \cdots p_{l}! }{ \Pi_{0 \leq j \leq 2, 0 \leq h \leq l} p_{h,j}! }
|| f_{\underline{p}}(Z_{0},\ldots,Z_{l}) ||_{(\bar{Z}_{0}^{0},\ldots,\bar{Z}_{l}^{0},\bar{X}^{0})}\\
\times C_{4.3} (\bar{Z}_{0}^{1})^{-p_{0,0}} \cdots (\bar{Z}_{l}^{1})^{-p_{l,0}}
||e_{m}(Z_{0},\ldots,Z_{l},X)||_{(\bar{Z}_{0}^{0},\ldots,\bar{Z}_{l}^{0},\bar{X}^{0})} \times
(\bar{Z}_{0}^{1})^{-p_{0,1}} \cdots (\bar{Z}_{l}^{1})^{-p_{l,1}} (\bar{X}^{1})^{S} \\
\times \left( \sum_{s=0}^{m-1} \sum_{\sum_{j=1}^{m-1} p_{h,2,j} = p_{h,2} , 0 \leq h \leq l} \right.
\frac{ p_{0,2}! \cdots p_{l,2}! }{ \Pi_{1 \leq j \leq m-1,0 \leq h \leq l} p_{h,2,j}! } \\
\times \Pi_{j=1}^{s} ((\bar{Z}_{0}^{1})^{-p_{0,2,j}} \cdots (\bar{Z}_{l}^{1})^{-p_{l,2,j}}(\bar{X}^{1})^{S}K_{4}\delta \\
+ C_{4.4} \sum_{q=0}^{S-1} (\bar{Z}_{0}^{1})^{-p_{0,2,j}} \cdots (\bar{Z}_{l}^{1})^{-p_{l,2,j}} \delta
||\frac{X^q}{q!}||_{(\bar{Z}_{0}^{0},\ldots,\bar{Z}_{l}^{0},\bar{X}^{0})}) \\
\times \Pi_{j=s+1}^{m-1} ((\bar{Z}_{0}^{1})^{-p_{0,2,j}} \cdots (\bar{Z}_{l}^{1})^{-p_{l,2,j}}(\bar{X}^{1})^{S}K_{4}\delta \\
+ C_{4.4} \sum_{q=0}^{S-1} (\bar{Z}_{0}^{1})^{-p_{0,2,j}} \cdots (\bar{Z}_{l}^{1})^{-p_{l,2,j}} \delta
\left. ||\frac{X^q}{q!}||_{(\bar{Z}_{0}^{0},\ldots,\bar{Z}_{l}^{0},\bar{X}^{0})}) \right) \leq \frac{1}{2}
\label{sum_d_e_less_one_half}
\end{multline}
Taking into account all the inequalities (\ref{norm_d_int_diff_H1_minus_H2<}), (\ref{factor_int_H1_w_minus_int_H2_w}),
(\ref{norm_diff_em_H1_minus_H2}), (\ref{norm_diff_em_part2}), (\ref{norm_diff_int_H1_minus_H2}), (\ref{norm_diff_power_H1_H2_w})
and (\ref{norm_diff_Hr_w}) under the constraint (\ref{sum_d_e_less_one_half}), we deduce (\ref{A_shrink}).\medskip

Finally, for fixed real numbers $\bar{Z}_{j}^{1}>0$, $0 \leq j \leq l$, such that
$M^{0} < \bar{Z}_{j}^{1} < Z_{j}^{0}$, we choose $0 < \bar{X}^{1} < \bar{X}^{0}$,
$\delta>0$ and the constant $K_{4}>0$ (recall that $R=K_{4}\delta$) in such a way that both constraints (\ref{stable_delta_K4_X})
and (\ref{sum_d_e_less_one_half}) hold. For these constants, the map $A$ satisfies both (\ref{ABR_sub_BR}) and (\ref{A_shrink}).
\end{proof}
We are in position to give the proof of Proposition 1. Let $\bar{Z}_{j}^{1}>0$, $0 \leq j \leq l$, be real numbers such that
$M^{0} < \bar{Z}_{j}^{1} < Z_{j}^{0}$, we choose $0 < \bar{X}^{1} < \bar{X}^{0}$,
$\delta>0$ and the constant $K_{4}>0$ as in Lemma 4. We have put $R=\delta K_{4}$.

>From the fact that $G(\bar{Z}_{0}^{1},\ldots,\bar{Z}_{l}^{1},\bar{X}^{1})$ equipped with the norm
$||.||_{(\bar{Z}_{0}^{1},\ldots,\bar{Z}_{l}^{1},\bar{X}^{1})}$ is a Banach space, the closed ball $(B_{R},d)$ for the metric
$d(x,y) = ||y-x||_{(\bar{Z}_{0}^{1},\ldots,\bar{Z}_{l}^{1},\bar{X}^{1})}$ is a complete metric space. From Lemma 4, the map
$A$ is a contraction from $(B_{R},d)$ into itself. From the classical fixed point theorem, we deduce that there exists a unique
$H_{f}(Z_{0},\ldots,Z_{l},X) \in B_{R}$ such that $A(H_{f})=H_{f}$.

By construction and taking into account Lemma 1, the formal series
$$ U_{f}(Z_{0},\ldots,Z_{l},X) = \partial_{X}^{-S}H_{f}(Z_{0},\ldots,Z_{l},X) + w(Z_{0},\ldots,Z_{l},X) $$
is a solution of the problem (\ref{CP_CK}), (\ref{CP_CK_init_cond}). Moreover,
$U_{f}(Z_{0},\ldots,Z_{l},X) \in G(\bar{Z}_{0}^{1},\ldots,\bar{Z}_{l}^{1},\bar{X}^{1})$ and
$$
|| U_{f}(Z_{0},\ldots,Z_{l},X) ||_{(\bar{Z}_{0}^{1},\ldots,\bar{Z}_{l}^{1},\bar{X}^{1})} \leq
\delta( (\bar{X}^{1})^{S}K_{4} + \sum_{j=0}^{S-1}
||\frac{X^j}{j!}||_{(\bar{Z}_{0}^{0},\ldots,\bar{Z}_{l}^{0},\bar{X}^{0})} )
$$
which yields (\ref{estim_sol_CK}).
\end{proof}

\subsection{Weighted Banach spaces of holomorphic functions on sectors}

We denote $D(0,r)$ the open disc centered at $0$ with radius $r>0$ in $\mathbb{C}$. Let $\epsilon_{0}>0$ a real number and let $S_{d}$ be
an unbounded sector in direction $d \in \mathbb{R}$ centered at $0$ in $\mathbb{C}$. By convention, these sectors do not contain the origin in
$\mathbb{C}$. For any open set $\mathcal{D} \subset \mathbb{C}$, we denote $\mathcal{O}(\mathcal{D})$ the vector space of
holomorphic functions on
$\mathcal{D}$. Let $l \geq 1$ be an integer. For all tuples $\underline{\beta} = (\beta_{0},\ldots,\beta_{l+1}) \in \mathbb{N}^{l+2}$,
let $\rho_{\underline{\beta}} \geq 0$ be positive real numbers such that
$\rho_{\underline{\beta}'} \leq \rho_{\underline{\beta}}$ if $|\underline{\beta}'| \geq |\underline{\beta}|$ (where
by definition $|\underline{\beta}| = \sum_{j=0}^{l+1} \beta_{j}$). We define
$\Omega_{\underline{\beta}} = D(0,\rho_{\underline{\beta}}) \cup S_{d}$.

\begin{defin} Let $b > 1$ a real number and let $r_{b}(\underline{\beta}) = \sum_{n=0}^{|\underline{\beta}|} 1/(n+1)^b$ for all
 tuples $\underline{\beta} \in \mathbb{N}^{l+2}$. Let $\epsilon \in D(0,\epsilon_{0}) \setminus \{ 0 \}$ and $r,\sigma > 0$ be real numbers.
We denote by
$E_{\underline{\beta},\epsilon,\sigma,r,\Omega_{\underline{\beta}}}$ the vector space of all functions
$v \in \mathcal{O}(\Omega_{\underline{\beta}})$ such that
$$ ||v(\tau)||_{\underline{\beta},\epsilon,\sigma,r,\Omega_{\underline{\beta}}} :=
\sup_{\tau \in \Omega_{\underline{\beta}}} |v(\tau)| (1+\frac{|\tau|^2}{|\epsilon|^{2r}})
\exp\left( -\frac{\sigma}{|\epsilon|^{r}} r_{b}(\underline{\beta}) |\tau| \right) $$
is finite.
\end{defin}
{\bf Remark:} These norms are appropriate modifications of the norms defined in \cite{lamasa}, \cite{ma1}.\medskip

In the next proposition, we study some parameter depending linear operators acting on the spaces
$E_{\underline{\beta},\epsilon,\sigma,r,\Omega_{\underline{\beta}}}$.

\begin{prop} Let $s_{1},k_{0} \geq 0$ be integers. Let $\underline{\beta} = (\beta_{0},\ldots,\beta_{l+1}) \in \mathbb{N}^{l+2}$,
$\tilde{\underline{\beta}} = (\tilde{\beta}_{0},\ldots,\tilde{\beta}_{l+1}) \in \mathbb{N}^{l+2}$. Under the assumption
that $|\tilde{\underline{\beta}}| > |\underline{\beta}|$, the operator $\tau^{s_1}\partial_{\tau}^{-k_{0}}$ defines a bounded linear
map from $E_{\underline{\beta},\epsilon,\sigma,r,\Omega_{\underline{\beta}}}$ into
$E_{\tilde{\underline{\beta}},\epsilon,\sigma,r,\Omega_{\tilde{\underline{\beta}}}}$, for all $\epsilon \in D(0,\epsilon_{0}) \setminus \{ 0 \}$.
Moreover, one has
\begin{multline}
 || \tau^{s_1}\partial_{\tau}^{-k_0}
V_{\underline{\beta}}(\tau,\epsilon)||_{\tilde{\underline{\beta}},\epsilon,\sigma,r,\Omega_{\tilde{\underline{\beta}}}} \leq
|| V_{\underline{\beta}}(\tau,\epsilon)||_{\underline{\beta},\epsilon,\sigma,r,\Omega_{\underline{\beta}}}\\
\times |\epsilon|^{r(s_{1}+k_{0})} \left( (|\tilde{\underline{\beta}}|+1)^{b(s_{1}+k_{0})}(
\frac{ (s_{1}+k_{0})e^{-1} }{\sigma( |\tilde{\underline{\beta}}| - |\underline{\beta}|)} )^{s_{1}+k_{0}} +
(|\tilde{\underline{\beta}}|+1)^{b(s_{1}+k_{0}+2)}(
\frac{ (s_{1}+k_{0}+2)e^{-1} }{\sigma( |\tilde{\underline{\beta}}| - |\underline{\beta}|)} )^{s_{1}+k_{0}+2} \right)
\label{norm_tau_partial_V<norm_V}
\end{multline}
for all $V_{\underline{\beta}}(\tau,\epsilon) \in E_{\underline{\beta},\epsilon,\sigma,r,\Omega_{\underline{\beta}}}$.
 \end{prop}
\begin{proof} The proof follows the same lines of arguments as Lemma 1 from \cite{ma1}. By construction of the operator
$\partial_{\tau}^{-k_0}$, one can write
\begin{multline}
\tau^{s_1}\partial_{\tau}^{-k_0}V_{\underline{\beta}}(\tau,\epsilon) =
\tau^{s_{1}+k_{0}} \int_{0}^{1} \cdots \int_{0}^{1} V_{\underline{\beta}}(h_{k_0} \cdots h_{1} \tau,\epsilon)\\
\times (1 + \frac{|h_{k_0} \cdots h_{1} \tau|^{2} }{ |\epsilon|^{2r} })
\exp( -\frac{\sigma}{|\epsilon|^{r}}r_{b}(\underline{\beta})|h_{k_0} \cdots h_{1} \tau| )
\frac{ \exp( \frac{\sigma}{|\epsilon|^{r}}r_{b}(\underline{\beta})|h_{k_0} \cdots h_{1} \tau| ) }{ 1 +
|h_{k_0} \cdots h_{1} \tau|^{2}/ |\epsilon|^{2r} }\\
\times M_{k_0}(h_{1},\ldots,h_{k_{0}-1}) dh_{k_0} \cdots dh_{1}
\end{multline}
where $M_{k_0}(h_{1},\ldots,h_{k_{0}-1})$  is a monic monomial for $k_{0} \geq 2$, while $M_{1}=1$, for all
$\tau \in \Omega_{\underline{\beta}}$. We deduce that
\begin{multline}
|\tau^{s_1}\partial_{\tau}^{-k_0}V_{\underline{\beta}}(\tau,\epsilon)| (1 + \frac{|\tau|^2}{|\epsilon|^{2r}})
\exp( - \frac{\sigma}{|\epsilon|^{r}} r_{b}(\tilde{\underline{\beta}}) |\tau| )\\
\leq |\tau|^{s_{1}+k_{0}}||V_{\underline{\beta}}(\tau,\epsilon)||_{\underline{\beta},\epsilon,\sigma,r,\Omega_{\underline{\beta}}}
(1 + \frac{|\tau|^2}{|\epsilon|^{2r}}) \exp( - \frac{\sigma}{|\epsilon|^{r}} ( r_{b}(\tilde{\underline{\beta}})  -
 r_{b}(\underline{\beta}) ) |\tau| ) \label{tau_partial_tau_V_exp<}
\end{multline}
for all $\tau \in \Omega_{\tilde{\underline{\beta}}} \subset \Omega_{\underline{\beta}}$. By definition of $r_{b}$, one has
\begin{equation}
r_{b}(\tilde{\underline{\beta}}) - r_{b}(\underline{\beta}) = \sum_{n=|\underline{\beta}|+1}^{|\tilde{\underline{\beta}}|}
\frac{1}{(n+1)^{b}} \geq \frac{|\tilde{\underline{\beta}}| - |\underline{\beta}| }{ ( |\tilde{\underline{\beta}}| + 1)^{b} }
\label{minor_difference_rb}
\end{equation}
Now, we recall the classical estimates: for all integers $m_{1},m_{2} \geq 0$,
\begin{equation}
\sup_{x \geq 0} x^{m_1}e^{-m_{2}x} = (\frac{m_1}{m_2})^{m_1} e^{-m_1} \label{xexpx<}
\end{equation}
holds. Using (\ref{xexpx<}), we deduce that
\begin{multline}
(|\tau|^{s_{1}+k_{0}}+ \frac{|\tau|^{s_{1}+k_{0}+2}}{|\epsilon|^{2r}})
\exp( -\frac{\sigma}{|\epsilon|^{r}}
(\frac{ |\tilde{\underline{\beta}}| - |\underline{\beta}| }{ (|\tilde{\underline{\beta}}|+1)^{b}})|\tau| ) \leq \\
 |\epsilon|^{r(s_{1}+k_{0})} \left( (|\tilde{\underline{\beta}}|+1)^{b(s_{1}+k_{0})}(
\frac{ (s_{1}+k_{0})e^{-1} }{\sigma( |\tilde{\underline{\beta}}| - |\underline{\beta}|)} )^{s_{1}+k_{0}} +
(|\tilde{\underline{\beta}}|+1)^{b(s_{1}+k_{0}+2)}(
\frac{ (s_{1}+k_{0}+2)e^{-1} }{\sigma( |\tilde{\underline{\beta}}| - |\underline{\beta}|)} )^{s_{1}+k_{0}+2} \right)
\label{tau_exp_tau<}
\end{multline}
Taking into account (\ref{tau_partial_tau_V_exp<}), (\ref{minor_difference_rb}) and (\ref{tau_exp_tau<}), we get
(\ref{norm_tau_partial_V<norm_V}).
\end{proof} \medskip

In the next proposition, we study the convolution product of functions in the spaces
$E_{\underline{\beta},\epsilon,\sigma,r,\Omega_{\underline{\beta}}}$.

\begin{prop} Let $\underline{\beta}^{1} = (\beta_{0}^{1},\ldots,\beta_{l+1}^{1}) \in \mathbb{N}^{l+2}$ and
$\underline{\beta}^{2} = (\beta_{0}^{2},\ldots,\beta_{l+1}^{2}) \in \mathbb{N}^{l+2}$. Let
$\underline{\beta} \in \mathbb{N}^{l+2}$ such that $|\underline{\beta}| \geq |\underline{\beta}^{1}| + |\underline{\beta}^{2}|$.
Then, for any $V_{\underline{\beta}^{1}}(\tau,\epsilon) \in E_{\underline{\beta}^{1},\epsilon,\sigma,r,\Omega_{\underline{\beta}^{1}}}$,
$V_{\underline{\beta}^{2}}(\tau,\epsilon) \in E_{\underline{\beta}^{2},\epsilon,\sigma,r,\Omega_{\underline{\beta}^{2}}}$, the
convolution product $V_{\underline{\beta}^{1}}(\tau,\epsilon) * V_{\underline{\beta}^{2}}(\tau,\epsilon)$ belongs to
$E_{\underline{\beta},\epsilon,\sigma,r,\Omega_{\underline{\beta}}}$, for all $\epsilon \in D(0,\epsilon_{0}) \setminus \{ 0 \}$. Moreover,
there exists
a universal constant $C_{5}>0$ such that
\begin{equation}
|| \int_{0}^{\tau} V_{\underline{\beta}^{1}}(\tau-s,\epsilon)
V_{\underline{\beta}^{2}}(s,\epsilon) ds ||_{\underline{\beta},\epsilon,\sigma,r,\Omega_{\underline{\beta}}}\\
\leq
C_{5} |\epsilon|^{r} ||V_{\underline{\beta}^{1}}(\tau,\epsilon)||_{\underline{\beta}^{1},\epsilon,\sigma,r,\Omega_{\underline{\beta}^{1}}}
||V_{\underline{\beta}^{2}}(\tau,\epsilon)||_{\underline{\beta}^{2},\epsilon,\sigma,r,\Omega_{\underline{\beta}^{2}}}
\label{norm_conv_V1_V2<norm_V1_norm_V2}
\end{equation}
for all $V_{\underline{\beta}^{1}}(\tau,\epsilon) \in E_{\underline{\beta}^{1},\epsilon,\sigma,r,\Omega_{\underline{\beta}^{1}}}$,
$V_{\underline{\beta}^{2}}(\tau,\epsilon) \in E_{\underline{\beta}^{2},\epsilon,\sigma,r,\Omega_{\underline{\beta}^{2}}}$.
\end{prop}
\begin{proof} We mimic the proof of Lemma 3 in \cite{ma1}. One can write
\begin{multline*}
\int_{0}^{\tau} V_{\underline{\beta}^{1}}(\tau-s,\epsilon)
V_{\underline{\beta}^{2}}(s,\epsilon) ds = \int_{0}^{\tau}
V_{\underline{\beta}^{1}}(\tau-s,\epsilon)(1+\frac{|\tau-s|^{2}}{|\epsilon|^{2r}})
\exp(-\frac{\sigma}{|\epsilon|^{r}}r_{b}(\underline{\beta}^{1})|\tau-s| ) \\
\times V_{\underline{\beta}^{2}}(s,\epsilon)(1+\frac{|s|^{2}}{|\epsilon|^{2r}})
\exp(-\frac{\sigma}{|\epsilon|^{r}}r_{b}(\underline{\beta}^{2})|s| ) \times
\frac{ \exp( \frac{\sigma}{|\epsilon|^{r}}(r_{b}(\underline{\beta}^{1})|\tau - s| + r_{b}(\underline{\beta}^{2})|s|) ) }{
(1+\frac{|\tau-s|^2}{|\epsilon|^{2r}})(1+\frac{|s|^2}{|\epsilon|^{2r}}) } ds
\end{multline*}
for all $\tau \in \Omega_{\underline{\beta}} \subset \Omega_{\underline{\beta}^{1}} \cap \Omega_{\underline{\beta}^{2}}$. We deduce
that
\begin{multline}
|\int_{0}^{\tau} V_{\underline{\beta}^{1}}(\tau-s,\epsilon)
V_{\underline{\beta}^{2}}(s,\epsilon) ds| \leq ||V_{\underline{\beta}^{1}}(\tau,\epsilon)||_{\underline{\beta}^{1},\epsilon,\sigma,r,\Omega_{\underline{\beta}^{1}}}
||V_{\underline{\beta}^{2}}(\tau,\epsilon)||_{\underline{\beta}^{2},\epsilon,\sigma,r,\Omega_{\underline{\beta}^{2}}}\\
\times \int_{0}^{1} \frac{ |\tau| \exp( \frac{\sigma |\tau|}{|\epsilon|^{r}}(r_{b}(\underline{\beta}^{1})(1-h)+
r_{b}(\underline{\beta}^{2})h) }{ (1 + \frac{|\tau|^2}{|\epsilon|^{2r}}(1-h)^2 )( 1 + \frac{|\tau|^2}{|\epsilon|^{2r}}h^2)} dh
\label{conv_V1_V2<norm_V1_norm_V2}
\end{multline}
for all $\tau \in \Omega_{\underline{\beta}}$. Since $r_{b}$ is increasing, one has
$$ r_{b}(\underline{\beta}^{1})(1-h) + r_{b}(\underline{\beta}^{2})h \leq r_{b}(\underline{\beta}).$$
Therefore,
\begin{multline}
(1 + \frac{|\tau|^2}{|\epsilon|^{2r}}) \exp( -\frac{\sigma}{|\epsilon|^r} r_{b}(\underline{\beta})|\tau| )
\int_{0}^{1} \frac{ |\tau| \exp( \frac{\sigma |\tau|}{|\epsilon|^{r}}(r_{b}(\underline{\beta}^{1})(1-h)+
r_{b}(\underline{\beta}^{2})h) }{ (1 + \frac{|\tau|^2}{|\epsilon|^{2r}}(1-h)^2 )( 1 + \frac{|\tau|^2}{|\epsilon|^{2r}}h^2)} dh\\
\leq \int_{0}^{1} \frac{ (1 + \frac{|\tau|^2}{|\epsilon|^{2r}})|\tau| }{
(1 + \frac{|\tau|^2}{|\epsilon|^{2r}}(1-h)^2 )( 1 + \frac{|\tau|^2}{|\epsilon|^{2r}}h^2)} dh = J(|\tau|,|\epsilon|) \label{defin_J}
\end{multline}
for all $\tau \in \Omega_{\underline{\beta}}$. Now, from \cite{cota}, we know that there exists a universal constant $C_{5}>0$ such that
$$
\frac{J(|\epsilon|^{r}|\tau|,|\epsilon|)}{|\epsilon|^r} \leq C_{5}
$$
for all $\tau \in \mathbb{C}$, $\epsilon \in \mathbb{C}^{\ast}$. We deduce that
\begin{equation}
\sup_{\tau \in \mathbb{C}} \frac{ J( |\tau|, |\epsilon| ) }{|\epsilon|^r} = \sup_{\tau \in \mathbb{C}}
\frac{J(|\epsilon|^{r}|\tau|,|\epsilon|)}{|\epsilon|^r} \leq C_{5} \label{sup_J<}
\end{equation}
for all $\epsilon \in \mathbb{C}^{\ast}$. Finally, gathering (\ref{conv_V1_V2<norm_V1_norm_V2}), (\ref{defin_J}) and (\ref{sup_J<}) yields
the estimates (\ref{norm_conv_V1_V2<norm_V1_norm_V2}).
\end{proof}
\begin{corol} Let $l_{0} \geq 0$ be an integer. The operator $\partial_{\tau}^{-l_0}$ defines a bounded linear map from  
$E_{\underline{\beta},\epsilon,\sigma,r,\Omega_{\underline{\beta}}}$ into itself, for all $\epsilon \in D(0,\epsilon_{0}) \setminus \{ 0 \}$.
Moreover, there
exists a constant $C_{6}>0$ (depending on $l_{0},\sigma$) such that
\begin{equation}
|| \partial_{\tau}^{-l_0}V_{\underline{\beta}}(\tau,\epsilon)||_{\underline{\beta},\epsilon,\sigma,r,\Omega_{\underline{\beta}}} \leq
C_{6}|\epsilon|^{rl_{0}}||V_{\underline{\beta}}(\tau,\epsilon)||_{\underline{\beta},\epsilon,\sigma,r,\Omega_{\underline{\beta}}}
\label{norm_int_tau_l0_v<norm_v}
\end{equation}
for all $V_{\underline{\beta}}(\tau,\epsilon) \in E_{\underline{\beta},\epsilon,\sigma,r,\Omega_{\underline{\beta}}}$.
\end{corol}
\begin{proof} We carry out a similar proof as in Corollary 1 from \cite{ma1}. We denote by $\chi_{\mathbb{C}}$ the function equal to $1$ on
$\mathbb{C}$. By definition,
we put $\chi_{\mathbb{C}}^{\ast 1} = \chi_{\mathbb{C}}$ and $\chi_{\mathbb{C}}^{\ast l}$ means the convolution product of
$\chi_{\mathbb{C}}$, $l-1$ times
for $l \geq 2$. By definition, we can write $\partial_{\tau}^{-l_0}V_{\underline{\beta}}(\tau,\epsilon) =
(\chi_{\mathbb{C}}(\tau))^{\ast l_0} \ast V_{\underline{\beta}}(\tau,\epsilon)$.
>From Proposition 3, there exists a (universal) constant $C_{5}>0$ such that
\begin{equation}
||\partial_{\tau}^{-l_0}V_{\underline{\beta}}(\tau,\epsilon) ||_{(\underline{\beta},\epsilon,\sigma,r,\Omega_{\underline{\beta}})}
 \leq C_{5}^{l_{0}}|\epsilon|^{rl_{0}}
||\chi_{\mathbb{C}}(\tau)||_{(\underline{0},\epsilon,\sigma,r,\Omega_{\underline{0}})}^{l_{0}}
||V_{\underline{\beta}}(\tau,\epsilon)||_{(\underline{\beta},\epsilon,\sigma,r,\Omega_{\underline{\beta}})} \label{norm_int_tau_l0_v<}
\end{equation}
where $\underline{0} = (0,\ldots,0) \in \mathbb{N}^{l+2}$. By Definition 2 and using the formula (\ref{xexpx<}), we have that
\begin{equation}
||\chi_{\mathbb{C}}(\tau)||_{(\underline{0},\epsilon,\sigma,r,\Omega_{\underline{0}})} =
\sup_{\tau \in \Omega_{\underline{0}}} (1+\frac{|\tau|^2}{|\epsilon|^{2r}})
\exp\left( -\frac{\sigma}{|\epsilon|^{r}} |\tau| \right) \leq 1 + (\frac{2 e^{-1}}{\sigma})^{2} \label{norm_chi<}
\end{equation}
>From the estimates (\ref{norm_int_tau_l0_v<}) and (\ref{norm_chi<}), we get the inequality (\ref{norm_int_tau_l0_v<norm_v}).
\end{proof}\medskip

The next proposition involves bound estimates for multiplication operators of bounded holomorphic functions.

\begin{prop} Let $a(\tau,\epsilon)$ be a holomorphic function on $\Omega_{\underline{0} \times D(0,\epsilon_{0}) \setminus \{ 0 \}}$ such
that there exists
a constant $M>0$ with $\sup_{(\tau,\epsilon) \in \Omega_{\underline{0} \times D(0,\epsilon_{0}) \setminus \{ 0 \} }} |a(\tau,\epsilon)| \leq M$.
Then,
the multiplication by $a(\tau,\epsilon)$ is a bounded linear operator from
$E_{\underline{\beta},\epsilon,\sigma,r,\Omega_{\underline{\beta}}}$ into itself, for all $\epsilon \in D(0,\epsilon_{0}) \setminus \{ 0 \}$.
Moreover, the inequality
\begin{equation}
|| a(\tau,\epsilon) V_{\underline{\beta}}(\tau,\epsilon) ||_{\underline{\beta},\epsilon,\sigma,r,\Omega_{\underline{\beta}}}
\leq M ||V_{\underline{\beta}}(\tau,\epsilon)||_{\underline{\beta},\epsilon,\sigma,r,\Omega_{\underline{\beta}}}
\end{equation}
holds for all $V_{\underline{\beta}}(\tau,\epsilon) \in E_{\underline{\beta},\epsilon,\sigma,r,\Omega_{\underline{\beta}}}$.
\end{prop}
\begin{proof} The proof is a direct consequence of the definition 2 for the norm
$||.||_{\underline{\beta},\epsilon,\sigma,r,\Omega_{\underline{\beta}}}$.
\end{proof}

\subsection{A global Cauchy problem}

We keep the same notations as in the previous section. In the following, we introduce some definitions.
Let $\mathcal{A}_{1}$ be a finite subset of $\mathbb{N}^{3}$ and let $\mathcal{A}_{2}$ be a finite subset of $\mathbb{N}^{2}$.
Let $l \geq 1$ be an integer.

\noindent For all $(k_{0},k_{1},k_{2}) \in \mathcal{A}_{1}$, we denote by $I_{(k_{0},k_{1},k_{2})}$ a finite subset of
$\mathbb{N}^{2}$. For all $(k_{0},k_{1},k_{2}) \in \mathcal{A}_{1}$,
all $(s_{1},s_{2}) \in I_{(k_{0},k_{1},k_{2})}$, all integers $\beta_{0},\beta_{l+1} \geq 0$, we denote 
$a_{s_{1},s_{2},k_{0},k_{1},k_{2},\beta_{0},\beta_{l+1}}(\tau,\epsilon)$ some holomorphic function
on $\Omega_{\underline{0}} \times D(0,\epsilon_{0})$ which satisfies the estimates : there exist constants $\rho,\rho'>0$,
$\mathfrak{a} _{s_{1},s_{2},k_{0},k_{1},k_{2}} > 0$ with
\begin{equation}
\sup_{(\tau,\epsilon) \in \Omega_{\underline{0}} \times D(0,\epsilon_{0})}
|a_{s_{1},s_{2},k_{0},k_{1},k_{2},\beta_{0},\beta_{l+1}}(\tau,\epsilon)| \leq \mathfrak{a} _{s_{1},s_{2},k_{0},k_{1},k_{2}}
(\frac{e^{-\rho'}}{2})^{\beta_0}(\frac{1}{2 \rho})^{\beta_{l+1}} \beta_{0}! \beta_{l+1}!
\end{equation}
For all $(k_{0},k_{1},k_{2}) \in \mathcal{A}_{1}$, we consider the series
$$ a_{(k_{0},k_{1},k_{2})}(\tau,z,x,\epsilon) = \sum_{ (s_{1},s_{2}) \in I_{(k_{0},k_{1},k_{2})} }
\sum_{\beta_{0} \geq 0, \beta_{l+1} \geq 0} a_{s_{1},s_{2},k_{0},k_{1},k_{2},\beta_{0},\beta_{l+1}}(\tau,\epsilon)
\tau^{s_1}\epsilon^{-s_2} \frac{e^{i z \beta_{0}}}{\beta_{0}!} \frac{x^{\beta_{l+1}}}{\beta_{l+1}!} $$
which define holomorphic functions on $\Omega_{\underline{0}} \times H_{\rho'} \times D(0,\rho)
\times D(0,\epsilon_{0}) \setminus \{ 0 \}$,
where $H_{\rho'}$ is defined as the following strip in $\mathbb{C}$
$$ H_{\rho'} = \{ z \in \mathbb{C} / |\mathrm{Im}(z)| < \rho' \}. $$

\noindent For all $(l_{0},l_{1}) \in \mathcal{A}_{2}$, we denote by $J_{(l_{0},l_{1})}$ a finite subset of $\mathbb{N}$.
For all $m_{1} \in J_{(l_{0},l_{1})}$, all integers $\beta_{0},\beta_{l+1} \geq 0$, we denote 
$\alpha_{m_{1},l_{0},l_{1},\beta_{0},\beta_{l+1}}(\tau,\epsilon)$ some bounded holomorphic function
on $\Omega_{\underline{0}} \times D(0,\epsilon_{0})$ with the following estimates : there exists a constant
$\mathfrak{a} _{m_{1},l_{0},l_{1}} > 0$ such that
\begin{equation}
\sup_{(\tau,\epsilon) \in \Omega_{\underline{0}} \times D(0,\epsilon_{0})}
|\alpha_{m_{1},l_{0},l_{1},\beta_{0},\beta_{l+1}}(\tau,\epsilon)| \leq \mathfrak{a} _{m_{1},l_{0},l_{1}}
(\frac{e^{-\rho'}}{2})^{\beta_0}(\frac{1}{2 \rho})^{\beta_{l+1}} \beta_{0}! \beta_{l+1}!
\end{equation}
for all $\beta_{0},\beta_{l+1} \geq 0$.

For all $(l_{0},l_{1}) \in \mathcal{A}_{2}$, we consider the series
$$ \alpha_{(l_{0},l_{1})}(\tau,z,x,\epsilon) = \sum_{ m_{1} \in J_{(l_{0},l_{1})} } \sum_{\beta_{0} \geq 0,\beta_{l+1} \geq 0}
\alpha_{m_{1},l_{0},l_{1},\beta_{0},\beta_{l+1}}(\tau,\epsilon) \epsilon^{-m_1} \frac{e^{i z \beta_{0}}}{\beta_{0}!}
\frac{x^{\beta_{l+1}}}{\beta_{l+1}!}$$
which define holomorphic functions on $\Omega_{\underline{0}} \times H_{\rho'} \times D(0,\rho) \times
D(0,\epsilon_{0}) \setminus \{ 0 \}$.

Let $l \geq 1$ be an integer and let $\xi_{0}=1$ and $\xi_{1},\ldots,\xi_{l}$ be real algebraic numbers such that the family
$\{ 1,\xi_{1},\ldots,\xi_{l} \}$ is $\mathbb{Z}$-linearly independent
(this means that each $\xi_{j}$ is a real root of a polynomial $P_{j} \in \mathbb{Z}[X]$ and if there exist integers
$k_{0},\ldots,k_{l} \in \mathbb{Z}$ such that $k_{0}+ \sum_{j=1}^{l} k_{j} \xi_{j} = 0$, then $k_{j}=0$ for $0 \leq j \leq l$). Due to the
classical primitive element theorem of Artin, we consider an algebraic number field $\mathbb{K} = \mathbb{Q}(\xi)$ containing
the numbers $\xi_{j}$, $1 \leq j \leq l$ and we denote
$h+1 \geq 1$ its degree (that is the dimension of the vector space $\mathbb{Q}(\xi)$ over $\mathbb{Q}$). The following lemma is a direct
consequence of Theorem 11 in \cite{sh}.

\begin{lemma} There exists a constant $C_{\xi_{1},\ldots,\xi_{l}}>0$ (depending on $\xi_{1},\ldots,\xi_{l}$) such that for any
$\underline{k}=(k_{0},\ldots,k_{l}) \in  \mathbb{Z}^{l+1} \setminus \{ 0 \}$, the inequality
\begin{equation}
|k_{0}+k_{1} \xi_{1} + \cdots + k_{l}\xi_{l}| \geq \frac{C_{\xi_{1},\ldots,\xi_{l}}}{(\max_{j=0}^{l}|k_{j}|)^{h}} \geq
\frac{C_{\xi_{1},\ldots,\xi_{l}}}{(|k_{0}|+\ldots+|k_{l}|)^{h}} \label{minor_sum_xi_l}
\end{equation}
holds.
\end{lemma}
{\bf Example:} Let $\xi$ be an algebraic number. Assume that the degree of $\mathbb{Q}(\xi)$ is $h+1$. Then, the algebraic numbers
$\{ 1 , \xi, \ldots, \xi^{h} \}$ are $\mathbb{Z}$-linearly independent and the inequality (\ref{minor_sum_xi_l}) above holds for
$\xi_{j}=\xi^{j}$, $0 \leq j \leq h$. In that case, one recovers Lemma 2.1 of \cite{ioru}.\medskip

Let $S,r_{1},r_{2} \geq 1$ be integers. We put
\begin{equation}
\rho_{\underline{\beta}} =
C_{\xi_{1},\ldots,\xi_{l}}^{r_{1}/r_{2}}/(2(\beta_{0}+1+\beta_{1}+\cdots+\beta_{l+1})^{hr_{1}/r_{2}}) \label{defin_rho_beta_gp}
\end{equation}
for all $\underline{\beta}=(\beta_{0},\ldots,\beta_{l+1}) \in \mathbb{N}^{l+2}$. We consider an
unbounded sector $S_{d} \subset \mathbb{C}$ centered at 0 such that
\begin{equation}
\mathrm{arg}(\tau) \neq \frac{2k+1}{r_{2}}\pi \ \ , \ \ 0 \leq k \leq r_{2}-1,
\end{equation}
for all $\tau \in S_{d}$. As in the previous section, we put $\Omega_{\underline{\beta}} = D(0,\rho_{\underline{\beta}}) \cup S_{d}$. For all
$0 \leq j \leq S-1$, we choose a set of functions
$V_{(\beta_{0},\ldots,\beta_{l},j)}(\tau,\epsilon) \in
E_{(\beta_{0},\ldots,\beta_{l},j),\epsilon,\sigma,r,\Omega_{(\beta_{0},\ldots,\beta_{l},j)}}$
for all $\beta_{0},\ldots,\beta_{l} \geq 0$ and we consider the formal series
\begin{equation}
V_{j}(\tau,z,\epsilon) = \sum_{\beta_{0},\ldots,\beta_{l} \geq 0} V_{(\beta_{0},\ldots,\beta_{l},j)}(\tau,\epsilon)
\frac{\exp( iz( \sum_{j=0}^{l} \beta_{j} \xi_{j} )) }{\beta_{0}! \cdots \beta_{l}!}  \label{defin_V_j_init_data}
\end{equation}
for all $0 \leq j \leq S-1$.

We consider the following Cauchy problem
\begin{multline}
(\tau^{r_2} + (-i \partial_{z}+1)^{r_1}) \partial_{x}^{S}V(\tau,z,x,\epsilon) = \sum_{(k_{0},k_{1},k_{2}) \in \mathcal{A}_{1}}
a_{(k_{0},k_{1},k_{2})}(\tau,z,x,\epsilon) \partial_{\tau}^{-k_{0}} \partial_{z}^{k_1} \partial_{x}^{k_2}V(\tau,z,x,\epsilon)\\
+ \sum_{(l_{0},l_{1}) \in \mathcal{A}_{2}, l_{1} \geq 2} \alpha_{(l_{0},l_{1})}(\tau,z,x,\epsilon)
\partial_{\tau}^{-l_{0}}(V(\tau,z,x,\epsilon))^{\ast l_{1}} \label{GCP}
\end{multline}
where $V^{\ast 1} = V$ and $V^{\ast l_{1}}$, $l_{1} \geq 2$, stands for the convolution product of $V$ applied $l_{1}-1$ times with
respect to $\tau$, for given initial conditions
\begin{equation}
 (\partial_{z}^{j}V)(\tau,z,0,\epsilon) = V_{j}(\tau,z,\epsilon) \ \ , \ \ 0 \leq j \leq S-1. \label{init_cond_GCP}
\end{equation}

In the sequel, we will need the next lemma.

\begin{lemma} There exists a constant $C_{7}>0$ (depending on $r_{1},r_{2},C_{\xi_{1},\ldots,\xi_{l}},S_{d}$) such that
\begin{equation}
 \left| \frac{1}{\tau^{r_2} + (1 + \sum_{j=0}^{l} \beta_{j} \xi_{j})^{r_1}} \right| \leq C_{7}(1 + \sum_{j=0}^{l} \beta_{j})^{hr_{1}}
\end{equation}
for all $\tau \in \Omega_{(\beta_{0},\ldots,\beta_{l},\beta_{l+1}+S)}$, for all $\beta_{j} \geq 0$, $0 \leq j \leq l+1$.
 \end{lemma}
\begin{proof} We put $A = 1 + \sum_{j=0}^{l} \beta_{j} \xi_{j}$. The following partial fraction decomposition
\begin{equation}
\frac{1}{ \tau^{r_2} + A^{r_1}} = \sum_{k=0}^{r_{2}-1} \frac{A_{k}}{\tau - |A|^{r_{1}/r_{2}}e^{i \pi \frac{2k+1}{r_2} } }
\label{partial_frac_decomp_A} 
\end{equation}
holds, where
$$ A_{k} =\frac{1}{r_2} \frac{e^{-i \pi \frac{(2k+1)(r_{2}-1)}{r_2} }}{|A|^{r_{1} - \frac{r_1}{r_2}}} $$
for all $0 \leq k \leq r_{2}-1$. Now, there exists some constant $C_{8}>0$ (depending on $S_{d}$) such that
\begin{equation}
 | \tau - |1 + \sum_{j=0}^{l} \beta_{j} \xi_{j}|^{r_{1}/r_{2}}e^{i \pi \frac{2k+1}{r_2}} | \geq
C_{8} |1 + \sum_{j=0}^{l} \beta_{j} \xi_{j}|^{r_{1}/r_{2}} \label{minor_module_tau_minus_sum_xi_l_r1_r2}
\end{equation}
for all $\tau \in \Omega_{(\beta_{0},\ldots,\beta_{l},\beta_{l+1}+S)}$. Indeed, from (\ref{minor_sum_xi_l}), we know that
\begin{equation}
|1 + \sum_{j=0}^{l} \beta_{j} \xi_{j}|^{r_{1}/r_{2}} \geq
\frac{C_{\xi_{1},\ldots,\xi_{l}}^{r_{1}/r_{2}}}{(1+\sum_{j=0}^{l} \beta_{j})^{hr_{1}/r_{2}}} >
\frac{C_{\xi_{1},\ldots,\xi_{l}}^{r_{1}/r_{2}}}{(1 + \sum_{j=0}^{l+1} \beta_{j} + S)^{hr_{1}/r_{2}}} \label{minor_sum_xi_l_r1_r2}
\end{equation}
for all $\beta_{j} \geq 0$, $0 \leq j \leq l+1$. Let $\tau \in D(0, \rho_{(\beta_{0},\ldots,\beta_{l},\beta_{l+1}+S)})$. From
(\ref{minor_sum_xi_l_r1_r2}), we can write
$$ \tau = \frac{h}{2}e^{i\theta}|1 + \sum_{j=0}^{l} \beta_{j} \xi_{j}|^{r_{1}/r_{2}} $$
for some $0 \leq h < 1$ and $\theta \in [0,2\pi)$. Therefore,
\begin{multline}
| \tau - |1 + \sum_{j=0}^{l} \beta_{j} \xi_{j}|^{r_{1}/r_{2}}e^{i \pi \frac{2k+1}{r_2}} | =
|1 + \sum_{j=0}^{l} \beta_{j} \xi_{j}|^{r_{1}/r_{2}}| \frac{h}{2}e^{i \theta} - e^{i \pi \frac{2k+1}{r_2}}|\\
\geq \frac{1}{2} |1 + \sum_{j=0}^{l} \beta_{j} \xi_{j}|^{r_{1}/r_{2}}.
\end{multline}
Now, let $\tau \in S_{d}$. For all $k \in \{ 0, \ldots, r_{2}-1 \}$, we can write
$$ \tau = s |1 + \sum_{j=0}^{l} \beta_{j} \xi_{j}|^{r_{1}/r_{2}}e^{i (\pi \frac{2k+1}{r_2} + s_{k}) } $$
where $s \geq 0$ and $s_{k} \in \mathbb{R}$. By contruction of $S_{d}$, we have that $e^{is_{k}} \neq 1$ for all
$k \in \{ 0,\ldots,r_{2}-1 \}$. As a result, there exists a constant $C_{9}>0$ (depending on $S_{d}$)
with $|se^{is_{k}} - 1| \geq C_{9}$, for all $s \geq 0$. Hence,
\begin{multline}
| \tau - |1 + \sum_{j=0}^{l} \beta_{j} \xi_{j}|^{r_{1}/r_{2}}e^{i \pi \frac{2k+1}{r_2}} | =
|1 + \sum_{j=0}^{l} \beta_{j} \xi_{j}|^{r_{1}/r_{2}} |se^{i (\pi \frac{2k+1}{r_2} + s_{k}) } - e^{i \pi \frac{2k+1}{r_2}}|\\
\geq C_{9}|1 + \sum_{j=0}^{l} \beta_{j} \xi_{j}|^{r_{1}/r_{2}}
\end{multline}
As a consequence, we get that (\ref{minor_module_tau_minus_sum_xi_l_r1_r2}) holds. On the other hand, from
(\ref{minor_sum_xi_l}), we get that
\begin{equation}
\frac{1}{|A|^{r_1}} \leq \frac{(1+\sum_{j=0}^{l} \beta_{j})^{hr_{1}}}{C_{\xi_{1},\ldots,\xi_{l}}^{r_1}} \label{maj_1_over_A_r1}
\end{equation}
for all $\beta_{j} \geq 0$, $0 \leq j \leq l$. Gathering (\ref{partial_frac_decomp_A}), (\ref{minor_module_tau_minus_sum_xi_l_r1_r2})
and (\ref{maj_1_over_A_r1}), we deduce that
\begin{equation}
\frac{|A_{k}|}{|\tau - |A|^{r_{1}/r_{2}}e^{i \pi \frac{2k+1}{r_2} }|} \leq \frac{1}{r_{2}C_{8}|A|^{r_{1}}} \leq
\frac{1}{r_{2}C_{8}C_{\xi_{1},\ldots,\xi_{l}}^{r_1}}
(1 + \sum_{j=0}^{l} \beta_{j})^{hr_{1}}
\end{equation}
for all $0 \leq k \leq r_{2}-1$. The lemma follows.
\end{proof}

In the next proposition, we construct formal series solutions of (\ref{GCP}), (\ref{init_cond_GCP}).

\begin{prop} Under the assumption that
\begin{equation}
S > k_{2}
\end{equation}
for all $(k_{0},k_{1},k_{2}) \in \mathcal{A}_{1}$, there exists a formal series
\begin{equation}
 V(\tau,z,x,\epsilon) = \sum_{\underline{\beta} = (\beta_{0},\ldots,\beta_{l+1}) \in \mathbb{N}^{l+2}}
V_{\underline{\beta}}(\tau,\epsilon)  \frac{\exp( iz( \sum_{j=0}^{l} \beta_{j} \xi_{j} )) }{\beta_{0}! \cdots \beta_{l}!}
\frac{x^{\beta_{l+1}}}{\beta_{l+1}!} \label{defin_V}
\end{equation}
solution of (\ref{GCP}), (\ref{init_cond_GCP}), where the coefficients $\tau \mapsto V_{\underline{\beta}}(\tau,\epsilon)$ belong to
the space $\mathcal{O}(\Omega_{\underline{\beta}})$ for all $\underline{\beta} \in \mathbb{N}^{l+2}$ and satisfy the following
recursion formula
\begin{multline}
(\tau^{r_2} + (1 + \sum_{j=0}^{l} \beta_{j} \xi_{j})^{r_1})
\frac{V_{(\beta_{0},\ldots,\beta_{l},\beta_{l+1}+S)}(\tau,\epsilon)}{\beta_{0}! \cdots \beta_{l}! \beta_{l+1}!} =
\sum_{(k_{0},k_{1},k_{2}) \in \mathcal{A}_{1}} \sum_{(s_{1},s_{2}) \in I_{(k_{0},k_{1},k_{2})}}
\sum_{\substack{\beta_{0}^{1} + \beta_{0}^{2}=\beta_{0} \\ \beta_{l+1}^{1} + \beta_{l+1}^{2} = \beta_{l+1}} } \\
\frac{a_{s_{1},s_{2},k_{0},k_{1},k_{2},\beta_{0}^{1},\beta_{l+1}^{1}}(\tau,\epsilon)}{\beta_{0}^{1}!\beta_{l+1}^{1}!}
\tau^{s_1} \epsilon^{-s_{2}}
\frac{ \partial_{\tau}^{-k_{0}}(V_{(\beta_{0}^{2},\beta_{1},\ldots,\beta_{l},\beta_{l+1}^{2}+k_{2})}(\tau,\epsilon)) }{\beta_{0}^{2}!
\beta_{1}! \cdots \beta_{l}! \beta_{l+1}^{2}! }i^{k_{1}}(\beta_{0}^{2} + \sum_{j=1}^{l} \beta_{j} \xi_{j})^{k_1}\\
+ \sum_{(l_{0},l_{1}) \in \mathcal{A}_{2},l_{1} \geq 2} \sum_{m \in J_{(l_{0},l_{1})} }
\sum_{ \substack{ \beta_{0}^{-1} + \beta_{0}^{0} + \ldots + \beta_{0}^{l_{1}-1} = \beta_{0}  \\
\beta_{j}^{0} + \ldots + \beta_{j}^{l_{1}-1} = \beta_{j}, 1 \leq j \leq l \\
\beta_{l+1}^{-1} + \beta_{l+1}^{0} + \ldots + \beta_{l+1}^{l_{1}-1} = \beta_{l+1} } }
\frac{\alpha_{m_{1},l_{0},l_{1},\beta_{0}^{-1},\beta_{l+1}^{-1}}(\tau,\epsilon)\epsilon^{-m_{1}}}{\beta_{0}^{-1}! \beta_{l+1}^{-1}!}\\
\times \frac{ \partial_{\tau}^{-l_{0}}( V_{(\beta_{0}^{0},\ldots,\beta_{l+1}^{0})}(\tau,\epsilon) * \cdots *
V_{(\beta_{0}^{l_{1}-1},\ldots,\beta_{l+1}^{l_{1}-1})}(\tau,\epsilon) ) }{ \Pi_{m=0}^{l_{1}-1} \Pi_{j=0}^{l+1} \beta_{j}^{m}! }
\label{recursion_V_beta}
\end{multline}
for all $\beta_{0},\ldots,\beta_{l+1} \geq 0$, all $\tau \in \Omega_{(\beta_{0},\ldots,\beta_{l},\beta_{l+1}+S)}$, all
$\epsilon \in D(0,\epsilon_{0}) \setminus \{ 0 \}$.
\end{prop}
\begin{proof} By hypothesis, we know in particular that $V_{(\beta_{0},\ldots,\beta_{l},j)}(\tau,\epsilon)$ belongs to
$\mathcal{O}(\Omega_{(\beta_{0},\ldots,\beta_{l},j)})$ for any $\beta_{0},\ldots,\beta_{l} \geq 0$, all $0 \leq j \leq S-1$, all
$\epsilon \in D(0,\epsilon_{0}) \setminus \{ 0 \}$.
Since $\Omega_{\underline{\beta}'} \subset \Omega_{\underline{\beta}}$ when
$|\underline{\beta}'|>|\underline{\beta}|$ and using Lemma 6, one gets that the functions
$\tau \mapsto V_{\underline{\beta}}(\tau,\epsilon)$ which are defined by the recursion (\ref{recursion_V_beta})
actually belong to $\mathcal{O}(\Omega_{\underline{\beta}})$ for any $\underline{\beta} \in \mathbb{N}^{l+2}$, for all
$\epsilon \in D(0,\epsilon_{0}) \setminus \{ 0 \}$. Direct computation by identification of the powers of $e^{iz},\ldots,e^{iz \xi_{l}}$ and
the powers of $x$
shows that the formal series (\ref{defin_V}) is solution of (\ref{GCP}), (\ref{init_cond_GCP}),
if its coefficients $V_{\underline{\beta}}(\tau,\epsilon)$ satisfy the recursion (\ref{recursion_V_beta}).
\end{proof}

In the next proposition, we state norm inequalities for the sequence $V_{\underline{\beta}}$.

\begin{prop} We consider the sequence of functions $V_{\underline{\beta}}(\tau,\epsilon)$ defined by the recursion
(\ref{recursion_V_beta}) for given
initial data $V_{(\beta_{0},\ldots,\beta_{l},j)}(\tau,\epsilon)$ defined above for all $\beta_{0},\ldots,\beta_{l} \geq 0$, $0 \leq j \leq S-1$.
Then, for all $\underline{\beta} \in \mathbb{N}^{l+2}$, all $\epsilon \in D(0,\epsilon_{0}) \setminus \{ 0 \}$, the function
$\tau \mapsto V_{\underline{\beta}}(\tau,\epsilon)$ belongs to $E_{\underline{\beta},\epsilon,\sigma,r,\Omega_{\underline{\beta}}}$.
We put
$w_{\underline{\beta}}(\epsilon) =
||V_{\underline{\beta}}(\tau,\epsilon)||_{\underline{\beta},\epsilon,\sigma,r,\Omega_{\underline{\beta}}}$, for all
$\underline{\beta} \in \mathbb{N}^{l+2}$, all $\epsilon \in D(0,\epsilon_{0}) \setminus \{ 0 \}$. Then,
the sequence $w_{\underline{\beta}}(\epsilon)$ satisfies
the following estimates. There exist constants $C_{10}>0$ (depending on $r_{1},r_{2},C_{\xi_{1},\ldots,\xi_{l}},S_{d}$) and
$C_{11}>0$ (depending on $\sigma$) such that
\begin{multline}
\frac{w_{(\beta_{0},\ldots,\beta_{l},\beta_{l+1}+S)}(\epsilon)}{\beta_{0}! \cdots \beta_{l}! \beta_{l+1}!} \leq
\sum_{(k_{0},k_{1},k_{2}) \in \mathcal{A}_{1}} \sum_{(s_{1},s_{2}) \in I_{(k_{0},k_{1},k_{2})}}
C_{10}(1+\sum_{j=0}^{l} \beta_{j})^{hr_1}\\
\times \sum_{\substack{\beta_{0}^{1} + \beta_{0}^{2}=\beta_{0} \\ \beta_{l+1}^{1} + \beta_{l+1}^{2} = \beta_{l+1}}}
\frac{A_{s_{1},s_{2},k_{0},k_{1},k_{2},\beta_{0}^{1},\beta_{l+1}^{1}}}{\beta_{0}^{1}! \beta_{l+1}^{1}!}
\times \frac{w_{(\beta_{0}^{2},\beta_{1}, \ldots, \beta_{l}, \beta_{l+1}^{2} + k_{2})}(\epsilon)}{
\beta_{0}^{2}! \beta_{1}! \cdots \beta_{l}! \beta_{l+1}^{2}! } |\epsilon|^{r(s_{1}+k_{0})-s_{2}}\\
\times \left( (\sum_{j=0}^{l+1} \beta_{j} + S)^{b(s_{1}+k_{0})} (\frac{(s_{1}+k_{0})e^{-1}}{\sigma(S-k_{2})})^{s_{1}+k_{0}}+
(\sum_{j=0}^{l+1} \beta_{j} + S)^{b(s_{1}+k_{0}+2)} (\frac{(s_{1}+k_{0}+2)e^{-1}}{\sigma(S-k_{2})})^{s_{1}+k_{0}+2}
\right)\\
\times (\beta_{0} + \sum_{j=1}^{l} \beta_{j}|\xi_{j}|)^{k_1} + \sum_{(l_{0},l_{1}) \in \mathcal{A}_{2},l_{1} \geq 2}
\sum_{m_{1} \in J_{(l_{0},l_{1}})}\\
C_{10}(1+\sum_{j=0}^{l} \beta_{j})^{hr_1} \times
\sum_{ \substack{ \beta_{0}^{-1} + \beta_{0}^{0} + \ldots + \beta_{0}^{l_{1}-1} = \beta_{0}  \\
\beta_{j}^{0} + \ldots + \beta_{j}^{l_{1}-1} = \beta_{j}, 1 \leq j \leq l \\
\beta_{l+1}^{-1} + \beta_{l+1}^{0} + \ldots + \beta_{l+1}^{l_{1}-1} = \beta_{l+1} } }
\frac{B_{m_{1},l_{0},l_{1},\beta_{0}^{-1},\beta_{l+1}^{-1}}}{\beta_{0}^{-1}!\beta_{l+1}^{-1}!}
\times  C_{11}^{l_{1}} |\epsilon|^{r(l_{0}+l_{1}-1)-m_{1}}
\frac{\Pi_{j=0}^{l_{1}-1}w_{(\beta_{0}^{j},\ldots,\beta_{l+1}^{j})}(\epsilon)}{\Pi_{j=0}^{l_{1}-1}
\beta_{0}^{j}! \cdots \beta_{l+1}^{j}!} \label{ineq_rec_w_beta}
\end{multline}
for all $\beta_{j} \geq 0$, $0 \leq j \leq l+1$, all $\epsilon \in D(0,\epsilon_{0}) \setminus \{ 0 \}$, where
\begin{multline*}
 A_{s_{1},s_{2},k_{0},k_{1},k_{2},\beta_{0}^{1},\beta_{l+1}^{1}} =
\sup_{(\tau,\epsilon) \in \Omega_{\underline{0}} \times D(0,\epsilon_{0}) \setminus \{ 0 \}}
|a_{s_{1},s_{2},k_{0},k_{1},k_{2},\beta_{0}^{1},\beta_{l+1}^{1}}(\tau,\epsilon)|,\\
B_{m_{1},l_{0},l_{1},\beta_{0}^{-1},\beta_{l+1}^{-1}} = \sup_{(\tau,\epsilon) \in \Omega_{\underline{0}} \times D(0,\epsilon_{0})
\setminus \{ 0 \} }
|\alpha_{m_{1},l_{0},l_{1},\beta_{0}^{-1},\beta_{l+1}^{-1}}(\tau,\epsilon)|.
\end{multline*}
\end{prop}
\begin{proof} We apply the norm
$||.||_{(\beta_{0},\ldots,\beta_{l},\beta_{l+1}+S),\epsilon,\sigma,r,\Omega_{(\beta_{0},\ldots,\beta_{l},\beta_{l+1}+S)}}$ on the left and
right hand side of the equality (\ref{recursion_V_beta}) and use Propositions 2, 3, 4, Lemma 6  and Corollary 1 in order to majorize
the right hand side. Indeed, using Propositions 2, 3, 4, Corollary 1, Lemma 6 and the estimates
$$ (\beta_{0}^{2} + \sum_{j=1}^{l} \beta_{j} |\xi_{j}|)^{k_1} \leq (\beta_{0} + \sum_{j=1}^{l} \beta_{j} |\xi_{j}|)^{k_1} $$
we get a constant $C_{7}>0$ (depending on $r_{1},r_{2},C_{\xi_{1},\ldots,\xi_{l}},S_{d}$) such that
\begin{multline}
|| \frac{1}{\tau^{r_2} + (1 + \sum_{j=0}^{l} \beta_{j} \xi_{j})^{r_1}}
\frac{a_{s_{1},s_{2},k_{0},k_{1},k_{2},\beta_{0}^{1},\beta_{l+1}^{1}}(\tau,\epsilon)}{\beta_{0}^{1}!\beta_{l+1}^{1}!}
\tau^{s_1} \epsilon^{-s_{2}}\\
\times \frac{ \partial_{\tau}^{-k_{0}}(V_{\beta_{0}^{2},\beta_{1},\ldots,\beta_{l},\beta_{l+1}^{2}+k_{2}})(\tau,\epsilon) }{\beta_{0}^{2}!
\beta_{1}! \cdots \beta_{l}! \beta_{l+1}^{2}! }i^{k_{1}}(\beta_{0}^{2} +
\sum_{j=1}^{l} \beta_{j} \xi_{j})^{k_1} ||_{(\beta_{0},\ldots,\beta_{l},\beta_{l+1}+S)} \leq 
C_{7}(1+\sum_{j=0}^{l} \beta_{j})^{hr_1}\\
\times \frac{A_{s_{1},s_{2},k_{0},k_{1},k_{2},\beta_{0}^{1},\beta_{l+1}^{1}}}{\beta_{0}^{1}! \beta_{l+1}^{1}!}
\times \frac{w_{(\beta_{0}^{2},\beta_{1}, \ldots, \beta_{l}, \beta_{l+1}^{2} + k_{2})}(\epsilon)}{
\beta_{0}^{2}! \beta_{1}! \cdots \beta_{l}! \beta_{l+1}^{2}! } |\epsilon|^{r(s_{1}+k_{0})-s_{2}}\\
\times \left( (\sum_{j=0}^{l+1} \beta_{j} + S)^{b(s_{1}+k_{0})} (\frac{(s_{1}+k_{0})e^{-1}}{\sigma(S-k_{2})})^{s_{1}+k_{0}}+
(\sum_{j=0}^{l+1} \beta_{j} + S)^{b(s_{1}+k_{0}+2)} (\frac{(s_{1}+k_{0}+2)e^{-1}}{\sigma(S-k_{2})})^{s_{1}+k_{0}+2}
\right)\\
\times (\beta_{0} + \sum_{j=1}^{l} \beta_{j}|\xi_{j}|)^{k_1}
\end{multline}
and using the propositions 3, 4 and corollary 1, we obtain a universal constant $C_{5}>0$ and some constants
$C_{6}>0$ (depending on $l_{0},\sigma$) and $C_{7}>0$ (depending on $r_{1},r_{2},C_{\xi_{1},\ldots,\xi_{l}},S_{d}$)
such that
\begin{multline}
||  \frac{1}{\tau^{r_2} + (1 + \sum_{j=0}^{l} \beta_{j} \xi_{j})^{r_1}}
\frac{\alpha_{m_{1},l_{0},l_{1},\beta_{0}^{-1},\beta_{l+1}^{-1}}(\tau,\epsilon)\epsilon^{-m_{1}}}{\beta_{0}^{-1}! \beta_{l+1}^{-1}!}\\
\times \frac{ \partial_{\tau}^{-l_{0}}( V_{\beta_{0}^{0},\ldots,\beta_{l+1}^{0}}(\tau,\epsilon) * \cdots *
V_{\beta_{0}^{l_{1}-1},\ldots,\beta_{l+1}^{l_{1}-1}}(\tau,\epsilon) ) }{ \Pi_{m=0}^{l_{1}-1} \Pi_{j=0}^{l+1} \beta_{j}^{m}! }
||_{(\beta_{0},\ldots,\beta_{l},\beta_{l+1}+S)} \\
\leq C_{7}(1+\sum_{j=0}^{l} \beta_{j})^{hr_1}
\frac{B_{m_{1},l_{0},l_{1},\beta_{0}^{-1},\beta_{l+1}^{-1}}}{\beta_{0}^{-1}!\beta_{l+1}^{-1}!}
\times C_{6} C_{5}^{l_{1}-1} |\epsilon|^{r(l_{0}+l_{1}-1)-m_{1}}
\frac{\Pi_{j=0}^{l_{1}-1}w_{(\beta_{0}^{j},\ldots,\beta_{l+1}^{j})}(\epsilon)}{\Pi_{j=0}^{l_{1}-1}
\beta_{0}^{j}! \cdots \beta_{l+1}^{j}!}
\end{multline}
\end{proof}

We define the following formal series
\begin{multline*}
 \mathbb{A}_{s_{1},s_{2},k_{0},k_{1},k_{2}}(Z_{0},X) =
\sum_{\beta_{0},\beta_{l+1} \geq 0} A_{s_{1},s_{2},k_{0},k_{1},k_{2},\beta_{0},\beta_{l+1}}
\frac{Z_{0}^{\beta_0}}{\beta_{0}!} \frac{X^{\beta_{l+1}}}{\beta_{l+1}!},\\
\mathbb{B}_{m_{1},l_{0},l_{1}}(Z_{0},X)=\sum_{\beta_{0},\beta_{l+1} \geq 0}
B_{m_{1},l_{0},l_{1},\beta_{0},\beta_{l+1}} \frac{Z_{0}^{\beta_0}}{\beta_{0}!} \frac{X^{\beta_{l+1}}}{\beta_{l+1}!}
\end{multline*}
We consider the following Cauchy problem
\begin{multline}
\partial_{X}^{S}U(Z_{0},\ldots,Z_{l},X,\epsilon) = \sum_{(k_{0},k_{1},k_{2}) \in \mathcal{A}_{1}}
\sum_{(s_{1},s_{2}) \in I_{(k_{0},k_{1},k_{2})}} C_{10}(1 + \sum_{j=0}^{l} Z_{j}\partial_{Z_j})^{hr_{1}}\\
\times \left( (\frac{(s_{1}+k_{0})e^{-1}}{\sigma (S-k_{2})})^{s_{1}+k_{0}} (\sum_{j=0}^{l} Z_{j}\partial_{Z_j} +
X\partial_{X} + S)^{b(s_{1}+k_{0})} \right. \\
\left. +  (\frac{(s_{1}+k_{0}+2)e^{-1}}{\sigma (S-k_{2})})^{s_{1}+k_{0}+2}
(\sum_{j=0}^{l} Z_{j}\partial_{Z_j} + X\partial_{X} + S)^{b(s_{1}+k_{0}+2)} \right)\\
\times ( \sum_{j=0}^{l} |\xi_{j}| Z_{j}\partial_{Z_j} )^{k_1}\left( \epsilon_{0}^{r(s_{1}+k_{0}) - s_{2}}
\mathbb{A}_{s_{1},s_{2},k_{0},k_{1},k_{2}}(Z_{0},X)
(\partial_{X}^{k_2}U)(Z_{0},\ldots,Z_{l},X,\epsilon) \right)\\
+ \sum_{(l_{0},l_{1}) \in \mathcal{A}_{2},l_{1} \geq 2} \sum_{m_{1} \in J_{(l_{0},l_{1})}}
C_{10}(1 + \sum_{j=0}^{l} Z_{j}\partial_{Z_j})^{hr_{1}} \\
\times \epsilon_{0}^{r(l_{0}+l_{1}-1)-m_{1}}C_{11}^{l_{1}}
\mathbb{B}_{m_{1},l_{0},l_{1}}(Z_{0},X) (U(Z_{0},\ldots,Z_{l},X,\epsilon))^{l_1} \label{Aux_CP_formal}
\end{multline}
for given initial data
\begin{equation}
(\partial_{X}^{j}U)(Z_{0},\ldots,Z_{l},0,\epsilon) = \sum_{\beta_{0} \geq 0,\ldots,\beta_{l} \geq 0}
w_{\beta_{0},\ldots,\beta_{l},j}(\epsilon) \frac{Z_{0}^{\beta_0} \cdots Z_{l}^{\beta_l} }{ \beta_{0}! \cdots \beta_{l}! } \ \ , \ \
0 \leq j \leq S-1
\label{Aux_CP_formal_init_cond}
\end{equation}
for all $\epsilon \in D(0,\epsilon_{0}) \setminus \{ 0 \}$.

\begin{prop} Under the assumption that
\begin{equation}
S > k_{2} + b(s_{1}+k_{0}+2) \label{assum_A}
\end{equation}
for all $(k_{0},k_{1},k_{2}) \in \mathcal{A}_{1}$, all $(s_{1},s_{2}) \in I_{(k_{0},k_{1},k_{2})}$, there exists a formal series
\begin{equation}
U(Z_{0},\ldots,Z_{l},X,\epsilon) = \sum_{\underline{\beta} = (\beta_{0},\ldots,\beta_{l+1}) \in \mathbb{N}^{l+2}}
U_{\underline{\beta}}(\epsilon) \frac{Z_{0}^{\beta_0} \cdots Z_{l}^{\beta_l} X^{\beta_{l+1}} }{\beta_{0}! \cdots \beta_{l}!
\beta_{l+1}!}
\end{equation}
solution of (\ref{Aux_CP_formal}), (\ref{Aux_CP_formal_init_cond}), where the coefficients
$U_{\underline{\beta}}(\epsilon)$ satisfy the following recursion
\begin{multline}
\frac{U_{(\beta_{0},\ldots,\beta_{l},\beta_{l+1}+S)}(\epsilon)}{\beta_{0}! \cdots \beta_{l}! \beta_{l+1}!} =
\sum_{(k_{0},k_{1},k_{2}) \in \mathcal{A}_{1}} \sum_{(s_{1},s_{2}) \in I_{(k_{0},k_{1},k_{2})}}
C_{10}(1+\sum_{j=0}^{l} \beta_{j})^{hr_1}\\
\times \sum_{\substack{\beta_{0}^{1} + \beta_{0}^{2}=\beta_{0} \\ \beta_{l+1}^{1} + \beta_{l+1}^{2} = \beta_{l+1}}}
\frac{A_{s_{1},s_{2},k_{0},k_{1},k_{2},\beta_{0}^{1},\beta_{l+1}^{1}}}{\beta_{0}^{1}! \beta_{l+1}^{1}!}
\times \frac{U_{(\beta_{0}^{2},\beta_{1}, \ldots, \beta_{l}, \beta_{l+1}^{2} + k_{2})}(\epsilon)}{
\beta_{0}^{2}! \beta_{1}! \cdots \beta_{l}! \beta_{l+1}^{2}! }\epsilon_{0}^{r(s_{1}+k_{0})-s_{2}}\\
\times \left( (\sum_{j=0}^{l+1} \beta_{j} + S)^{b(s_{1}+k_{0})} (\frac{(s_{1}+k_{0})e^{-1}}{\sigma(S-k_{2})})^{s_{1}+k_{0}}+
(\sum_{j=0}^{l+1} \beta_{j} + S)^{b(s_{1}+k_{0}+2)} (\frac{(s_{1}+k_{0}+2)e^{-1}}{\sigma(S-k_{2})})^{s_{1}+k_{0}+2}
\right)\\
\times (\beta_{0} + \sum_{j=1}^{l} \beta_{j}|\xi_{j}|)^{k_1} + \sum_{(l_{0},l_{1}) \in \mathcal{A}_{2},l_{1} \geq 2}
\sum_{m_{1} \in J_{l_{0},l_{1}}}\\
C_{10}(1+\sum_{j=0}^{l} \beta_{j})^{hr_1} \times
\sum_{ \substack{ \beta_{0}^{-1} + \beta_{0}^{0} + \ldots + \beta_{0}^{l_{1}-1} = \beta_{0}  \\
\beta_{j}^{0} + \ldots + \beta_{j}^{l_{1}-1} = \beta_{j}, 1 \leq j \leq l \\
\beta_{l+1}^{-1} + \beta_{l+1}^{0} + \ldots + \beta_{l+1}^{l_{1}-1} = \beta_{l+1} } }
\frac{B_{m_{1},l_{0},l_{1},\beta_{0}^{-1},\beta_{l+1}^{-1}}}{\beta_{0}^{-1}!\beta_{l+1}^{-1}!}\\
\times C_{11}^{l_{1}}\epsilon_{0}^{r(l_{0}+l_{1}-1)-m_{1}}
\frac{\Pi_{j=0}^{l_{1}-1}U_{(\beta_{0}^{j},\ldots,\beta_{l+1}^{j})}(\epsilon)}{\Pi_{j=0}^{l_{1}-1}
\beta_{0}^{j}! \cdots \beta_{l+1}^{j}!} \label{rec_U_beta}
\end{multline}
for all $\beta_{j} \geq 0$, $0 \leq j \leq l+1$, all $\epsilon \in D(0,\epsilon_{0}) \setminus \{ 0 \}$.
\end{prop}

\begin{prop} Under the assumption (\ref{assum_A}) with the additional condition that
\begin{equation}
r(s_{1}+k_{0}) \geq s_{2} \ \ , \ \ r(l_{0}+l_{1}-1) \geq m_{1} \ \ , \ \ l_{1} \geq 2, \label{assum_B}
\end{equation}
for all $(k_{0},k_{1},k_{2}) \in \mathcal{A}_{1}$, all $(s_{1},s_{2}) \in I_{(k_{0},k_{1},k_{2})}$, all
$(l_{0},l_{1}) \in \mathcal{A}_{2},l_{1} \geq 2$ and $m_{1} \in J_{l_{0},l_{1}}$, the following inequalities
\begin{equation}
w_{\underline{\beta}}(\epsilon) \leq U_{\underline{\beta}}(\epsilon) \label{w_beta<U_beta}
\end{equation}
hold for all $\underline{\beta} \in \mathbb{N}^{l+2}$, all $\epsilon \in D(0,\epsilon_{0}) \setminus \{ 0 \}$.
\end{prop}
\begin{proof} By the assumption (\ref{Aux_CP_formal_init_cond}), we know that
$$ U_{(\beta_{0},\ldots,\beta_{l},j)}(\epsilon) = w_{(\beta_{0},\ldots,\beta_{l},j)}(\epsilon) $$
for all $0 \leq j \leq S-1$, all $(\beta_{0},\ldots,\beta_{l}) \in \mathbb{N}^{l+1}$ and all $\epsilon \in D(0,\epsilon_{0}) \setminus \{ 0 \}$.
Therefore,
we get our result by using induction from the inequalities (\ref{ineq_rec_w_beta}) and the equalities (\ref{rec_U_beta}).
\end{proof}

In the next proposition we give sufficient conditions for the formal series solutions of the Cauchy problem
(\ref{GCP}), (\ref{init_cond_GCP}) in order to define actual holomorphic functions with exponential bound estimates.

\begin{prop} We make the assumption that (\ref{assum_B}) holds. We also assume that
\begin{equation}
S > b(s_{1}+k_{0}+2) + k_{2} \ \ , \ \ S \geq hr_{1} + b(s_{1}+k_{0}+2) + k_{1} + k_{2} \label{assum_C}
\end{equation}
for all $(k_{0},k_{1},k_{2}) \in \mathcal{A}_{1}$, all $(s_{1},s_{2}) \in I_{(k_{0},k_{1},k_{2})}$.

We choose two real numbers $\rho_{1}',M^{0}>0$ such that
\begin{equation}
M^{0} > 2(l+2)\exp( \rho_{1}' \max_{j=0}^{l} |\xi_{j}| ) \label{M0_larger_than_l_xij}
\end{equation}
and we take $\bar{X}^{0}>0$ and $\bar{Z}_{j}^{0} > M^{0}$, for $0 \leq j \leq l$. We
assume that the formal series
$$ \varphi_{j}(Z_{0},\ldots,Z_{l},\epsilon) = \sum_{\beta_{0},\ldots,\beta_{l} \geq 0}
w_{(\beta_{0},\ldots,\beta_{l},j)}(\epsilon) \frac{Z_{0}^{\beta_0} \cdots Z_{l}^{\beta_l}}{\beta_{0}! \cdots \beta_{l}!} \ \ , \ \ 
0 \leq j \leq S-1, $$
belong to $G(\bar{Z}_{0}^{0},\ldots,\bar{Z}_{l}^{0},\bar{X}^{0})$ for all $\epsilon \in D(0,\epsilon_{0}) \setminus \{ 0 \}$ and
that there exists a constant $C_{\varphi_{j}}$ such that
$\sup_{\epsilon \in D(0,\epsilon_{0}) \setminus \{ 0 \}} ||\varphi_{j}||_{(\bar{Z}_{0}^{0},\ldots,\bar{Z}_{l}^{0},\bar{X}^{0})}
\leq C_{\varphi_{j}}$. As a consequence,
the formal series (\ref{defin_V_j_init_data}) define holomorphic functions $V_{j}(\tau,z,\epsilon)$ on a the product
$S_{d} \times H_{\rho_{1}'} \times D(0,\epsilon_{0}) \setminus \{ 0 \}$ and satisfy the following estimates :
there exists a constant $C_{12}>0$ (depending on $C_{\varphi_{j}},l$)
such that
\begin{equation}
|V_{j}(\tau,z,\epsilon)| \leq C_{12} (1 + \frac{|\tau|^2}{|\epsilon|^{2r}})^{-1} \exp( \frac{\sigma}{|\epsilon|^{r}} \zeta(b) |\tau| )
\end{equation}
for all $(\tau,z,\epsilon) \in S_{d} \times H_{\rho_{1}'} \times D(0,\epsilon_{0}) \setminus \{ 0 \}$, all $0 \leq j \leq S-1$.

Then, there exists $\delta>0$ (depending on $M^{0}$,$\bar{Z}_{j}^{0},|\xi_{j}|$ for $0 \leq j \leq l$,
$\bar{X}^{0},r_{1},r_{2},C_{\xi_{1},\ldots,\xi_{l}},S_{d}$,$h,\sigma$,\\$b,S$,$\mathbb{A}_{s_{1},s_{2},k_{0},k_{1},k_{2}}(Z_{0},X)$
for $(k_{0},k_{1},k_{2}) \in \mathcal{A}_{1}$ and $(s_{1},s_{2}) \in I_{(k_{0},k_{1},k_{2})}$ and
$\mathbb{B}_{m_{1},l_{0},l_{1}}(Z_{0},X)$ for $(l_{0},l_{1}) \in \mathcal{A}_{2}$, $m_{1} \in J_{(l_{0},l_{1})}$) such that
if one assumes moreover that
\begin{equation}
||\varphi_{j}(Z_{0},\ldots,Z_{l},\epsilon)||_{(\bar{Z}_{0}^{0},\ldots,\bar{Z}_{l}^{0},\bar{X}^{0})} < \delta \ \ , \ \ 0 \leq j \leq S-1,
\label{small_norm_init_cond_GCP}
\end{equation}
for all $\epsilon \in D(0,\epsilon_{0}) \setminus \{ 0 \}$, the formal series (\ref{defin_V}), solution of (\ref{GCP}), (\ref{init_cond_GCP}),
defines a holomorphic function $V(\tau,z,x,\epsilon)$ on the
product $S_{d} \times H_{\rho_{1}'} \times D(0,\rho_{1}) \times D(0,\epsilon_{0}) \setminus \{ 0 \}$, for some $\rho_{1}>0$ and carries the next
bound estimates : there exists a constant $C_{13}>0$ (depending on the same as for $\delta$ given above) with
\begin{equation}
|V(\tau,z,x,\epsilon)| \leq C_{13} (1 + \frac{|\tau|^2}{|\epsilon|^{2r}})^{-1} \exp( \frac{\sigma}{|\epsilon|^{r}} \zeta(b) |\tau| )
\end{equation}
for all $(\tau,z,x,\epsilon) \in S_{d} \times H_{\rho_{1}'} \times D(0, \rho_{1}) \times D(0,\epsilon_{0}) \setminus \{ 0 \}$.
\end{prop}
\begin{proof} Since $\varphi_{j}(Z_{0},\ldots,Z_{l},\epsilon)$ belongs to $G(\bar{Z}_{0}^{0},\ldots,\bar{Z}_{l}^{0},\bar{X}^{0})$, we
get a constant $C_{14}>0$ (depending on $C_{\varphi_{j}}$)
such that
\begin{equation}
||V_{(\beta_{0},\ldots,\beta_{l},j)}(\tau,\epsilon)||_{(\beta_{0},\ldots,\beta_{l},j),\epsilon,\sigma,r,\Omega_{(\beta_{0},\ldots,\beta_{l},j)}}
\leq C_{14} (\frac{1}{\bar{Z}_{0}^{0}})^{\beta_0} \cdots (\frac{1}{\bar{Z}_{l}^{0}})^{\beta_l}
(\sum_{q=0}^{l} \beta_{q})!  \label{norm_V_j_beta<}
\end{equation}
for all $\beta_{0},\ldots,\beta_{l} \geq 0$, all $0 \leq j \leq S-1$. From the multinomial formula, we know that
\begin{equation}
 (\sum_{q=0}^{l} \beta_{q})! \leq (l+1)^{\sum_{q=0}^{l} \beta_{q}} \beta_{0}! \cdots \beta_{l}! \label{multin_ineq}
\end{equation}
for all $\beta_{0},\ldots,\beta_{l} \geq 0$. Therefore, from (\ref{norm_V_j_beta<}), we deduce that
\begin{equation}
|V_{(\beta_{0},\ldots,\beta_{l},j)}(\tau,\epsilon)| \leq C_{14} (1 + \frac{|\tau|^2}{|\epsilon|^{2r}})^{-1}
\exp( \frac{\sigma}{|\epsilon|^{r}}r_{b}(\sum_{q=0}^{l} \beta_{q} + j)|\tau| )(\frac{l+1}{M^{0}})^{\sum_{q=0}^{l} \beta_{q}}
\beta_{0}! \cdots \beta_{l}!
\end{equation}
for all $\beta_{0},\ldots,\beta_{l} \geq 0$, $0 \leq j \leq S-1$, all $\tau \in \Omega_{(\beta_{0},\ldots,\beta_{l},j)}$, and
all $\epsilon \in D(0,\epsilon_{0}) \setminus \{ 0 \}$. From the assumption (\ref{M0_larger_than_l_xij}), we deduce that
the formal series (\ref{defin_V_j_init_data}) defines a holomorphic function $V_{j}(\tau,z,\epsilon)$ on the product
$S_{d} \times H_{\rho_{1}'} \times D(0,\epsilon_{0}) \setminus \{ 0 \}$ and satisfies
\begin{multline}
|V_{j}(\tau,z,\epsilon)| \leq C_{14} (1 + \frac{|\tau|^2}{|\epsilon|^{2r}})^{-1} \exp( \frac{\sigma}{|\epsilon|^{r}} \zeta(b) |\tau| )\\
\times \sum_{\beta_{0},\ldots,\beta_{l} \geq 0}
(\frac{(l+1)\exp( \rho_{1}' \max_{q=0}^{l} |\xi_{q}| )}{M^{0}})^{\sum_{q=0}^{l} \beta_{q}}\\
\leq 2^{l+1}  C_{14} (1 + \frac{|\tau|^2}{|\epsilon|^{2r}})^{-1} \exp( \frac{\sigma}{|\epsilon|^{r}} \zeta(b) |\tau| )
\end{multline}
for all $(\tau,z,\epsilon) \in S_{d} \times H_{\rho_{1}'} \times D(0,\epsilon_{0}) \setminus \{ 0 \}$, all $0 \leq j \leq S-1$.

Under the assumptions (\ref{assum_B}), (\ref{assum_C}) together with (\ref{small_norm_init_cond_GCP}), we see that the
hypotheses of Proposition 1 are fulfilled for the Cauchy problem (\ref{Aux_CP_formal}), (\ref{Aux_CP_formal_init_cond}). Therefore,
we deduce that the formal solution $U(Z_{0},\ldots,Z_{l},X,\epsilon)$ of (\ref{Aux_CP_formal}), (\ref{Aux_CP_formal_init_cond})
constructed in Proposition 7 belongs to the space $G(\bar{Z}_{0}^{1},\ldots,\bar{Z}_{l}^{1},\bar{X}^{1})$ for some
$0 < \bar{X}^{1} < \bar{X}^{0}$ and for some $\bar{Z}_{j}^{1} > M^{0}$. Moreover, we get a constant $C_{15}>0$
(depending on $\bar{Z}_{j}^{0},\bar{Z}_{j}^{1}$,$|\xi_{j}|$ for $0 \leq j \leq l$,
$\bar{X}^{0},r_{1},r_{2},C_{\xi_{1},\ldots,\xi_{l}},S_{d}$,$h,\sigma,b,S$,\\
$\mathbb{A}_{s_{1},s_{2},k_{0},k_{1},k_{2}}(Z_{0},X)$
for $(k_{0},k_{1},k_{2}) \in \mathcal{A}_{1}$ and $(s_{1},s_{2}) \in I_{(k_{0},k_{1},k_{2})}$ and
$\mathbb{B}_{m_{1},l_{0},l_{1}}(Z_{0},X)$ for $(l_{0},l_{1}) \in \mathcal{A}_{2}$, $m_{1} \in J_{(l_{0},l_{1})}$)
such that
$$ ||U(Z_{0},\ldots,Z_{l},X,\epsilon)||_{(\bar{Z}_{0}^{1},\ldots,\bar{Z}_{l}^{1})} \leq \delta C_{15} $$
for all $\epsilon \in D(0,\epsilon_{0}) \setminus \{ 0 \}$. In particular, we deduce that
\begin{equation}
|U_{(\beta_{0},\ldots,\beta_{l+1})}(\epsilon)| \leq \delta C_{15}
(\frac{1}{\bar{Z}_{0}^{1}})^{\beta_0} \cdots (\frac{1}{\bar{Z}_{l}^{1}})^{\beta_l}(\frac{1}{\bar{X}^1})^{\beta_{l+1}}
(\sum_{j=0}^{l+1} \beta_{j})! \label{abs_U_beta<}
\end{equation}
for all $\beta_{0},\ldots,\beta_{l+1} \geq 0$. Gathering (\ref{w_beta<U_beta}) and (\ref{abs_U_beta<}) yields
\begin{equation}
||V_{\underline{\beta}}(\tau,\epsilon)||_{\underline{\beta},\epsilon,\sigma,r,\Omega_{\underline{\beta}}} \leq
\delta C_{15} (\frac{1}{\bar{Z}_{0}^{1}})^{\beta_0} \cdots (\frac{1}{\bar{Z}_{l}^{1}})^{\beta_l}(\frac{1}{\bar{X}^1})^{\beta_{l+1}}
(\sum_{j=0}^{l+1} \beta_{j})! \label{norm_V_beta<}
\end{equation}
for all $\underline{\beta} = (\beta_{0},\ldots,\beta_{l+1}) \in \mathbb{N}^{l+2}$. Again by the multinomial formula, we have that
$$ (\sum_{j=0}^{l+1} \beta_{j})! \leq (l+2)^{\sum_{j=0}^{l+1} \beta_{j}} \beta_{0}! \cdots \beta_{l+1}! $$
for all $\beta_{0},\ldots,\beta_{l+1} \geq 0$. Hence, from (\ref{norm_V_beta<}), we get that
\begin{equation}
|V_{\underline{\beta}}(\tau,\epsilon)| \leq \delta C_{15} (1 + \frac{|\tau|^2}{|\epsilon|^{2r}})^{-1}
\exp( \frac{\sigma}{|\epsilon|^{r}}r_{b}(\underline{\beta})|\tau| )(\frac{l+2}{M^{0}})^{\sum_{j=0}^{l} \beta_{j}}
(\frac{l+2}{\bar{X}^{1}})^{\beta_{l+1}} \beta_{0}! \cdots \beta_{l+1}! \label{V_beta<}
\end{equation}
for all $\beta_{0},\ldots,\beta_{l+1} \geq 0$, all $\tau \in \Omega_{\underline{\beta}}$, all $\epsilon \in D(0,\epsilon_{0}) \setminus \{ 0 \}$.
We deduce that
the formal series (\ref{defin_V}) defines a holomorphic function $V(\tau,z,x,\epsilon)$ on the product
$S_{d} \times H_{\rho_{1}'} \times D(0, \frac{\bar{X}^{1}}{2(l+2)}) \times D(0,\epsilon_{0}) \setminus \{ 0 \}$ and satisfies
\begin{multline}
|V(\tau,z,x,\epsilon)| \leq \delta C_{15} (1 + \frac{|\tau|^2}{|\epsilon|^{2r}})^{-1} \exp( \frac{\sigma}{|\epsilon|^{r}} \zeta(b) |\tau| )\\
\times \sum_{\beta_{0},\ldots,\beta_{l+1} \geq 0}
(\frac{(l+2)\exp( \rho_{1}' \max_{j=0}^{l} |\xi_{j}| )}{M^{0}})^{\sum_{j=0}^{l} \beta_{j}}
(\frac{(l+2)|x|}{\bar{X}^{1}})^{\beta_{l+1}}\\
\leq 2^{l+2} \delta C_{15} (1 + \frac{|\tau|^2}{|\epsilon|^{2r}})^{-1} \exp( \frac{\sigma}{|\epsilon|^{r}} \zeta(b) |\tau| )
\end{multline}
for all $(\tau,z,x,\epsilon) \in S_{d} \times H_{\rho_{1}'} \times D(0, \frac{\bar{X}^{1}}{2(l+2)}) \times D(0,\epsilon_{0}) \setminus \{ 0 \}$.
\end{proof}

\section{Analytic solutions in a complex parameter of a singular Cauchy problem}

\subsection{Laplace transform and asymptotic expansions}

We recall the definition of Borel summability of formal series with coefficients in a Banach space,
see \cite{ba}.

\begin{defin} A formal series
$$\hat{X}(t) = \sum_{j=0}^{\infty}  \frac{ a_{j} }{ j! } t^{j} \in \mathbb{E}[[t]]$$
with coefficients in a Banach space $( \mathbb{E}, ||.||_{\mathbb{E}} )$ is said to be $1-$summable
with respect to $t$ in the direction $d \in [0,2\pi)$ if \medskip

{\bf i)} there exists $\rho \in \mathbb{R}_{+}$ such that the following formal series, called formal Borel transform of $\hat{X}$ of order 1 
$$ \mathcal{B}(\hat{X})(\tau) = \sum_{j=0}^{\infty} \frac{ a_{j} \tau^{j}  }{ (j!)^{2} } \in \mathbb{E}[[\tau]],$$
is absolutely convergent for $|\tau| < \rho$, \medskip

{\bf ii)} there exists $\delta > 0$ such that the series $\mathcal{B}(\hat{X})(\tau)$ can be analytically continued with
respect to $\tau$ in a sector
$S_{d,\delta} = \{ \tau \in \mathbb{C}^{\ast} : |d - \mathrm{arg}(\tau) | < \delta \} $. Moreover, there exist $C >0$, and $K >0$ such that
$$ ||\mathcal{B}(\hat{X})(\tau)||_{\mathbb{E}}
\leq C e^{ K|\tau| } $$
for all $\tau \in S_{d, \delta}$.
\end{defin}
If this is so, the vector valued Laplace transform of order $1$ of $\mathcal{B}(\hat{X})(\tau)$ in the direction $d$ is defined by
$$ \mathcal{L}^{d}(\mathcal{B}(\hat{X}))(t) = t^{-1} \int_{L_{\gamma}}
\mathcal{B}(\hat{X})( \tau ) e^{ - ( \tau/t ) } d \tau,$$
along a half-line $L_{\gamma} = \mathbb{R}_{+}e^{i\gamma} \subset S_{d,\delta} \cup \{ 0 \}$, where $\gamma$ depends on $t$
and is chosen in such a way that $\cos(\gamma - \mathrm{arg}(t)) \geq \delta_{1} > 0$, for some fixed $\delta_{1}$, for all $t$ in
any sector
$$ S_{d,\theta,R} = \{ t \in \mathbb{C}^{\ast} : |t| < R \ \ , \ \ |d - \mathrm{arg}(t) | < \theta/2 \},$$
where $\pi < \theta < \pi + 2\delta$ and $0 < R < \delta_{1}/K$. The function $\mathcal{L}^{d} (\mathcal{B}(\hat{X}))(t)$
is called the $1-$sum of the formal series $\hat{X}(t)$ in the direction $d$ and defines a holomorphic bounded function on
the sector $S_{d,\theta,R}$. Moreover, it has the formal series $\hat{X}(t)$ as Gevrey asymptotic
expansion of order 1 with respect to $t$ on $S_{d,\theta,R}$. This means that for all $\theta_{1} < \theta$, there exist $C,M > 0$ such
that
$$ ||\mathcal{L}^{d}(\mathcal{B}(\hat{X}))(t) - \sum_{p=0}^{n-1} \frac{a_p}{p!} t^{p}||_{\mathbb{E}} \leq CM^{n}n!|t|^{n} $$
for all $n \geq 1$, all $t \in S_{d,\theta_{1},R}$.\medskip

In the next proposition, we give some well known identities for the Borel transform that will be useful in the sequel.

\begin{prop} Let $\hat{X}(t) = \sum_{n \geq 0} a_{n}t^{n}/n!$ and $\hat{G}(t) = \sum_{n \geq 0} b_{n}t^{n}/n!$ be formal series in $\mathbb{E}[[t]]$. We have the following equalities as formal series in $\mathbb{E}[[\tau]]$:
\begin{multline*}
({\tau} \partial_{\tau}^{2} + \partial_{\tau})(\mathcal{B}(\hat{X})(\tau)) = \mathcal{B}(\partial_{t}\hat{X}(t))(\tau),
\partial_{\tau}^{-1}(\mathcal{B}(\hat{X}))(\tau) = \mathcal{B}(t\hat{X}(t))(\tau),\\
\tau \mathcal{B}(\hat{X})(\tau) = \mathcal{B}((t^{2}\partial_{t} + t)\hat{X}(t))(\tau),
\int_{0}^{\tau}(\mathcal{B}\hat{X})(\tau - s)(\mathcal{B}\hat{G})(s) ds = \mathcal{B}(t\hat{X}(t)\hat{G}(t))(\tau).
\end{multline*}
\end{prop}
\begin{proof} By a direct computation, we have the following expansions from which the proposition 10 follows.
\begin{multline*}
\partial_{t}\hat{X}(t) = \sum_{n \geq 0} a_{n+1}\frac{t^n}{n!}, ({\tau} \partial_{\tau}^{2} + \partial_{\tau})(\mathcal{B}(\hat{X})(\tau)) = 
\sum_{n \geq 0} a_{n+1} \frac{{\tau}^n}{(n!)^{2}},
t\hat{X}(t) = \sum_{n \geq 1} na_{n-1} \frac{t^{n}}{n!},\\ \partial_{\tau}^{-1}(\mathcal{B}(\hat{X}))(\tau) =
\sum_{n \geq 1} na_{n-1} \frac{{\tau}^{n}}{(n!)^2},
(t^{2}\partial_{t} + t)\hat{X}(t) = \sum_{n \geq 1} n^{2}a_{n-1}\frac{t^n}{n!}, \tau \mathcal{B}(\hat{X})(\tau) = \sum_{n \geq 1} n^{2}a_{n-1}\frac{{\tau}^n}{(n!)^2},\\
t\hat{X}(t)\hat{G}(t) = \sum_{n \geq 1} (\sum_{l+m=n-1} \frac{ n! }{l!m!}a_{l}b_{m})\frac{t^n}{n!},
 \int_{0}^{\tau}
(\mathcal{B}\hat{X})(\tau - s)(\mathcal{B}\hat{G})(s) ds = \sum_{n \geq 1} (\sum_{l+m=n-1} \frac{ n! }{l!m!}a_{l}b_{m})\frac{\tau^n}{(n!)^{2}}
\end{multline*}
\end{proof}

\subsection{Analytic solutions of some singular Cauchy problem}
Let $S,r_{1},r_{2} \geq 1$ be integers. Let $\mathcal{S}$ be a finite subset of
$\mathbb{N}^{4}$, $\mathcal{N}$ be a finite subset of $\mathbb{N}^{2}$. For all $(s,k_{0},k_{1},k_{2}) \in \mathcal{S}$,
all integers $\beta_{0},\beta_{l+1} \geq 0$, we denote
$b_{s,k_{0},k_{1},k_{2},\beta_{0},\beta_{l+1}}(\epsilon)$ some holomorphic function on $D(0,\epsilon_{0})$ which satisfies the next
estimates : there exist constants $\rho,\rho'>0$, $\mathfrak{b} _{s,k_{0},k_{1},k_{2}} > 0$ with
\begin{equation}
\sup_{\epsilon \in D(0,\epsilon_{0})}
|b_{s,k_{0},k_{1},k_{2},\beta_{0},\beta_{l+1}}(\epsilon)| \leq \mathfrak{b} _{s,k_{0},k_{1},k_{2}}
(\frac{e^{-\rho'}}{2})^{\beta_0}(\frac{1}{2 \rho})^{\beta_{l+1}} \beta_{0}! \beta_{l+1}! 
\end{equation}
for all $\beta_{0},\beta_{l+1} \geq 0$. Likewise, for all $(l_{0},l_{1}) \in \mathcal{N}$, all integers $\beta_{0},\beta_{l+1} \geq 0$, we denote
$c_{l_{0},l_{1},\beta_{0},\beta_{l+1}}(\epsilon)$ some holomorphic function on $D(0,\epsilon_{0})$ with the following estimates :
there exist a constant $\mathfrak{c} _{l_{0},l_{1}} > 0$ with
\begin{equation}
\sup_{\epsilon \in D(0,\epsilon_{0})}
|c_{l_{0},l_{1},\beta_{0},\beta_{l+1}}(\epsilon)| \leq \mathfrak{c} _{l_{0},l_{1}}
(\frac{e^{-\rho'}}{2})^{\beta_0}(\frac{1}{2 \rho})^{\beta_{l+1}} \beta_{0}! \beta_{l+1}!
\end{equation}
for all $\beta_{0},\beta_{l+1} \geq 0$.

For all $(s,k_{0},k_{1},k_{2}) \in \mathcal{S}$ and all $(l_{0},l_{1}) \in \mathcal{N}$, we consider the series
\begin{multline*}
 b_{s,k_{0},k_{1},k_{2}}(z,x,\epsilon) = \sum_{\beta_{0},\beta_{l+1} \geq 0} b_{s,k_{0},k_{1},k_{2},\beta_{0},\beta_{l+1}}(\epsilon)
\frac{e^{iz \beta_{0}}}{\beta_{0}!} \frac{x^{\beta_{l+1}}}{\beta_{l+1}!},\\
c_{l_{0},l_{1}}(z,x,\epsilon) = \sum_{\beta_{0},\beta_{l+1} \geq 0} c_{l_{0},l_{1},\beta_{0},\beta_{l+1}}(\epsilon)
\frac{e^{iz \beta_{0}}}{\beta_{0}!} \frac{x^{\beta_{l+1}}}{\beta_{l+1}!}
\end{multline*}
which define bounded holomorphic functions on $H_{\rho'} \times D(0,\rho) \times D(0,\epsilon_{0})$.\medskip

We consider the following singular Cauchy problem
\begin{multline}
( (T^{2}\partial_{T} + T)^{r_2} + (-i\partial_{z}+1)^{r_1}) \partial_{x}^{S} Y_{U_{d},D(0,\epsilon_{0}) \setminus \{ 0 \}}(T,z,x,\epsilon) \\
= \sum_{(s,k_{0},k_{1},k_{2}) \in \mathcal{S}} b_{s,k_{0},k_{1},k_{2}}(z,x,\epsilon) \epsilon^{r(k_{0}-s)}
T^{s}(\partial_{T}^{k_0}\partial_{z}^{k_1}\partial_{x}^{k_2}Y_{U_{d},D(0,\epsilon_{0}) \setminus \{ 0 \}})(T,z,x,\epsilon)\\
+ \sum_{(l_{0},l_{1}) \in \mathcal{N}} c_{l_{0},l_{1}}(z,x,\epsilon) \epsilon^{-r(l_{0}+l_{1}-1)}
T^{l_{0}+l_{1}-1}(Y_{U_{d},D(0,\epsilon_{0}) \setminus \{ 0 \}}(T,z,x,\epsilon))^{l_{1}} \label{SCP}
\end{multline}
for given initial data
\begin{equation}
(\partial_{x}^{j}Y_{U_{d},D(0,\epsilon_{0}) \setminus \{ 0 \}})(T,z,0,\epsilon) = Y_{U_{d},D(0,\epsilon_{0}) \setminus \{ 0 \},j}(T,z,\epsilon)
\ \ , \ \ 0 \leq j \leq S-1. \label{SCP_init_cond}
\end{equation}
The initial conditions are constructed as follows: Let $U_{d}$ be an unbounded sector such that
$$ \mathrm{arg}(\tau) \neq \frac{2k+1}{r_{2}}\pi $$
for all $0 \leq k \leq r_{2}-1$, all $\tau \in U_{d}$. For all $0 \leq j \leq S-1$, all $(\beta_{0},\ldots,\beta_{l}) \in \mathbb{N}^{l+1}$,
let $V_{U_{d},(\beta_{0},\ldots,\beta_{l},j)}(\tau,\epsilon)$ be a function such that
\begin{equation}
V_{U_{d},(\beta_{0},\ldots,\beta_{l},j)}(\tau,\epsilon) \in
E_{(\beta_{0},\ldots,\beta_{l},j),\epsilon,\sigma,r,D(0,\rho_{(\beta_{0},\ldots,\beta_{l},j)}) \cup U_{d}} \label{V_beta_j_defin}
\end{equation}
for all $\epsilon \in D(0,\epsilon_{0}) \setminus \{ 0 \}$.

We choose two real numbers $\rho_{1}',M^{0}>0$ such that
\begin{equation}
 M^{0} > 2(l+2) \exp( \rho_{1}' \max_{j=0}^{l}|\xi_{j}| ) \label{choice_M0_rho_1_prime}
\end{equation}
and we take $\bar{X}^{0}>0$ and $\bar{Z}_{j}^{0}>M^{0}$, for all $0 \leq j \leq l$. We make the assumption that
the formal series
\begin{equation}
\varphi_{d,j}(Z_{0},\ldots,Z_{l},\epsilon) = \sum_{\beta_{0},\ldots,\beta_{l} \geq 0}
||V_{U_{d},(\beta_{0},\ldots,\beta_{l},j)}(\tau,\epsilon)||_{(\beta_{0},\ldots,\beta_{l},j),\epsilon,\sigma,r,D(0,\rho_{(\beta_{0},\ldots,\beta_{l},j)}) \cup
U_{d}} \frac{ Z_{0}^{\beta_{0}} \cdots Z_{l}^{\beta_{l}} }{ \beta_{0}! \cdots \beta_{l}! } \label{varphi_j_with_norm_V_beta_j}
\end{equation}
belongs to the Banach space $G(\bar{Z}_{0}^{0},\ldots,\bar{Z}_{l}^{0},\bar{X}^{0})$ for all $\epsilon \in D(0,\epsilon_{0}) \setminus \{ 0 \}$,
all $0 \leq j \leq S-1$. Moreover, we assume that there exists a constant $C_{\varphi_{d,j}}>0$ such that
$\sup_{\epsilon \in D(0,\epsilon_{0}) \setminus \{ 0 \}}
||\varphi_{d,j}||_{(\bar{Z}_{0}^{0},\ldots,\bar{Z}_{l}^{0},\bar{X}^{0})} \leq C_{\varphi_{d,j}}$.

Let
$$ V_{U_{d},(\beta_{0},\ldots,\beta_{l},j)}(\tau,\epsilon) = \sum_{m \geq 0}
\frac{\chi_{m,(\beta_{0},\ldots,\beta_{l},j)}(\epsilon)}{(m!)^{2}} \tau^{m} $$
be its Taylor expansion with respect to $\tau$ on $D(0,\rho_{(\beta_{0},\ldots,\beta_{l},j)})$, for all
$\epsilon \in D(0,\epsilon_{0}) \setminus \{ 0 \}$. We consider the formal series
\begin{equation}
\hat{Y}_{(\beta_{0},\ldots,\beta_{l},j)}(T,\epsilon) = \sum_{m \geq 0}
\frac{\chi_{m,(\beta_{0},\ldots,\beta_{l},j)}(\epsilon)}{m!} T^{m} \label{hat_Y_beta_j_definition}
\end{equation}
for all $\epsilon \in D(0,\epsilon_{0}) \setminus \{ 0 \}$. For all $(\beta_{0},\ldots,\beta_{l},j)$ we define
$Y_{U_{d},D(0,\epsilon_{0}) \setminus \{ 0 \},(\beta_{0},\ldots,\beta_{l},j)}(T,\epsilon)$ as the $1-$sum of
$\hat{Y}_{(\beta_{0},\ldots,\beta_{l},j)}(T,\epsilon)$
in the direction $d$. From the fact that $\tau \mapsto V_{U_{d},(\beta_{0},\ldots,\beta_{l},j)}(\tau,\epsilon)$ belongs to 
$E_{(\beta_{0},\ldots,\beta_{l},j),\epsilon,\sigma,r,D(0,\rho_{(\beta_{0},\ldots,\beta_{l},j)}) \cup U_{d}}$, for all
$\epsilon \in D(0,\epsilon_{0}) \setminus \{ 0 \}$, we get that
$Y_{U_{d},D(0,\epsilon_{0}) \setminus \{ 0 \},(\beta_{0},\ldots,\beta_{l},j)}(T,\epsilon)$ defines a holomorphic function
for all $T \in U_{d,\theta,h|\epsilon|^{r}}$,
all $\epsilon \in D(0,\epsilon_{0}) \setminus \{ 0 \}$, where
$$ U_{d,\theta,h|\epsilon|^{r}} = \{ T \in \mathbb{C} : |T| < h|\epsilon|^{r} \ \ , \ \ |d - \mathrm{arg}(T) | < \theta/2 \},$$
for some $\theta > \pi$ and some constant $h>0$ (independent of $\epsilon$ and $\beta_{0},\ldots,\beta_{l}$), for all
$0 \leq j \leq S-1$. The initial data are defined as the formal series
\begin{multline}
Y_{U_{d},D(0,\epsilon_{0}) \setminus \{ 0 \},j}(T,z,\epsilon) \\
= \sum_{\beta_{0},\ldots,\beta_{l} \geq 0}
Y_{U_{d},D(0,\epsilon_{0}) \setminus \{ 0 \},(\beta_{0},\ldots,\beta_{l},j)}(T,\epsilon)
\frac{ \exp( iz (\sum_{j=0}^{l} \beta_{j} \xi_{j}) ) }{\beta_{0}! \cdots \beta_{l}!} \ \ , \ \ 0 \leq j \leq S-1, \label{defin_Y_Ud_j}
\end{multline}
which actually define holomorphic functions on the domain $U_{d,\theta,h|\epsilon|^{r}} \times H_{\rho_{1}'} \times
D(0,\epsilon_{0}) \setminus \{ 0 \}$.
Indeed from the hypotheses (\ref{choice_M0_rho_1_prime}), (\ref{varphi_j_with_norm_V_beta_j}) and the multinomial formula
(\ref{multin_ineq}), following the first part of the proof of Proposition 9, we get that there exists a constant $C_{16}>0$ such that
\begin{multline}
|\sum_{\beta_{0},\ldots,\beta_{l} \geq 0} V_{U_{d},(\beta_{0},\ldots,\beta_{l},j)}(\tau,\epsilon)
\frac{ \exp( iz (\sum_{j=0}^{l} \beta_{j} \xi_{j}) ) }{\beta_{0}! \cdots \beta_{l}!}| \\
\leq C_{16} (1 + \frac{|\tau|^2}{|\epsilon|^{2r}})^{-1} \exp( \frac{\sigma}{|\epsilon|^{r}} \zeta(b) |\tau| )\\
\times \sum_{\beta_{0},\ldots,\beta_{l} \geq 0}
(\frac{(l+1)\exp( \rho_{1}' \max_{j=0}^{l} |\xi_{j}| )}{M^{0}})^{\sum_{j=0}^{l} \beta_{j}}\\
\leq 2^{l+1}  C_{16} (1 + \frac{|\tau|^2}{|\epsilon|^{2r}})^{-1} \exp( \frac{\sigma}{|\epsilon|^{r}} \zeta(b) |\tau| )
\end{multline}
for all $\tau \in U_{d}$, $z \in H_{\rho_{1}'}$ and $\epsilon \in D(0,\epsilon_{0}) \setminus \{ 0 \}$. \medskip

\noindent We get the following result.

\begin{prop} Let the initial data constructed as above. We make the following assumptions.
For all $(s,k_{0},k_{1},k_{2}) \in \mathcal{S}$, all $(l_{0},l_{1}) \in \mathcal{N}$, we have that
\begin{equation}
s \geq 2k_{0} \ \ , \ \ S > k_{2} \ \ , \ \ S > b(s-k_{0}+2) + k_{2} \ \ , \ \ S \geq hr_{1} + b(s-k_{0}+2) + k_{1} + k_{2} \ \ , \ \
l_{1} \geq 2.  \label{shape_SCP}
\end{equation}
Then, there exists a constant $I>0$ (independent of $\epsilon$) such that if one assumes that
\begin{equation}
|| \varphi_{d,j}(Z_{0},\ldots,Z_{l},\epsilon)||_{(\bar{Z}_{0}^{0},\ldots,\bar{Z}_{l}^{0},\bar{X}^{0})} < I \label{norm_varphi_j<I}
\end{equation}
where $\varphi_{d,j}$ is defined in (\ref{varphi_j_with_norm_V_beta_j}), for all $\epsilon \in D(0,\epsilon_{0}) \setminus \{ 0 \}$, the problem
(\ref{SCP}), (\ref{SCP_init_cond}) has a solution
\begin{multline}
Y_{U_{d},D(0,\epsilon_{0}) \setminus \{ 0 \}}(T,z,x,\epsilon) \\
= \sum_{\beta_{0},\ldots,\beta_{l},\beta_{l+1} \geq 0}
Y_{U_{d},D(0,\epsilon_{0}) \setminus \{ 0 \},(\beta_{0},\ldots,\beta_{l},\beta_{l+1})}(T,\epsilon)
\frac{ \exp( iz (\sum_{j=0}^{l} \beta_{j} \xi_{j}) ) }{\beta_{0}! \cdots \beta_{l}!}\frac{x^{\beta_{l+1}}}{\beta_{l+1}!}
 \end{multline}
which defines a bounded holomorphic function on $U_{d,\theta,h'|\epsilon|^{r}} \times H_{\rho_{1}'} \times D(0,\rho_{1})$ for some $h'>0$
and $\rho_{1}>0$, for all $\epsilon \in D(0,\epsilon_{0}) \setminus \{ 0 \}$. Moreover, each function
$Y_{U_{d},D(0,\epsilon_{0}) \setminus \{ 0 \},(\beta_{0},\ldots,\beta_{l},\beta_{l+1})}(T,\epsilon)$ can be written as a Laplace transform
of order 1 in the direction $d$
of a function $\tau \mapsto V_{U_{d},(\beta_{0},\ldots,\beta_{l},\beta_{l+1})}(\tau,\epsilon)$ which is holomorphic on
$D(0,\rho_{(\beta_{0},\ldots,\beta_{l},\beta_{l+1})}) \cup U_{d}$ and satisfies the estimates: there exist two constants
$C_{17}>0$ , $K_{17}>0$ (both independent of $\epsilon$) such that
\begin{equation}
|V_{U_{d},(\beta_{0},\ldots,\beta_{l},\beta_{l+1})}(\tau,\epsilon)| \leq C_{17} \exp( \frac{\sigma}{|\epsilon|^r} \zeta(b) |\tau| )
(\frac{l+2}{M^{0}})^{\sum_{j=0}^{l} \beta_{j}} K_{17}^{\beta_{l+1}} \beta_{0}! \cdots \beta_{l+1}!
\end{equation}
for all $\tau \in D(0,\rho_{(\beta_{0},\ldots,\beta_{l},\beta_{l+1})}) \cup U_{d}$ and $\epsilon \in D(0,\epsilon_{0}) \setminus \{ 0 \}$.
\end{prop}
\begin{proof} One considers a formal series
$$
\hat{Y}(T,z,x,\epsilon) = \sum_{\beta_{0},\ldots,\beta_{l},\beta_{l+1} \geq 0} \hat{Y}_{(\beta_{0},\ldots,\beta_{l},\beta_{l+1})}(T,\epsilon)
\frac{ \exp( iz (\sum_{j=0}^{l} \beta_{j} \xi_{j}) ) }{\beta_{0}! \cdots \beta_{l}!}\frac{x^{\beta_{l+1}}}{\beta_{l+1}!}
$$
solution of the equation (\ref{SCP}), with initial data
\begin{multline}
(\partial_{x}^{j}\hat{Y})(T,z,0,\epsilon) =  \hat{Y}_{j}(T,z,\epsilon) \\
= \sum_{\beta_{0},\ldots,\beta_{l} \geq 0}
\hat{Y}_{(\beta_{0},\ldots,\beta_{l},j)}(T,\epsilon)
\frac{ \exp( iz (\sum_{j=0}^{l} \beta_{j} \xi_{j}) ) }{\beta_{0}! \cdots \beta_{l}!} \ \ , \ \ 0 \leq j \leq S-1, \label{SCP_init_cond_formal}
\end{multline}
for all $\epsilon \in D(0,\epsilon_{0}) \setminus \{ 0 \}$ where $\hat{Y}_{(\beta_{0},\ldots,\beta_{l},j)}(T,\epsilon)$ are defined in
(\ref{hat_Y_beta_j_definition}). We consider the formal Borel transform of $\hat{Y}(T,z,x,\epsilon)$ of order 1 with
respect to $T$ denoted by
$$ \hat{V}(\tau,z,x,\epsilon) = \sum_{\beta_{0},\ldots,\beta_{l},\beta_{l+1} \geq 0}
\hat{V}_{(\beta_{0},\ldots,\beta_{l},\beta_{l+1})}(\tau,\epsilon) \frac{ \exp( iz (\sum_{j=0}^{l}
\beta_{j} \xi_{j}) ) }{\beta_{0}! \cdots \beta_{l}!}\frac{x^{\beta_{l+1}}}{\beta_{l+1}!} $$
where, by construction, $\hat{V}_{(\beta_{0},\ldots,\beta_{l},\beta_{l+1})}(\tau,\epsilon)$ is the Borel transform of order 1 with respect
to $T$ of the formal series $\hat{Y}_{(\beta_{0},\ldots,\beta_{l},\beta_{l+1})}(T,\epsilon)$, for all
$\epsilon \in D(0,\epsilon_{0}) \setminus \{ 0 \}$.

>From the identities of Proposition 10, we get that $\hat{V}(\tau,z,x,\epsilon)$ satisfies the following singular Cauchy problem
\begin{multline}
(\tau^{r_2} + (-i\partial_{z}+1)^{r_1})\partial_{x}^{S}\hat{V}(\tau,z,x,\epsilon) =\\
\sum_{(s,k_{0},k_{1},k_{2}) \in \mathcal{S}} b_{s,k_{0},k_{1},k_{2}}(z,x,\epsilon) \epsilon^{r(k_{0}-s)}
\partial_{\tau}^{-s}(\tau \partial_{\tau}^{2} + \partial_{\tau})^{k_0}\partial_{z}^{k_1}
\partial_{x}^{k_2}\hat{V}(\tau,z,x,\epsilon) \\
+ \sum_{(l_{0},l_{1}) \in \mathcal{N}} c_{l_{0},l_{1}}(z,x,\epsilon) \epsilon^{-r(l_{0}+l_{1}-1)}
\partial_{\tau}^{-l_0}(\hat{V}(\tau,z,x,\epsilon))^{\ast l_{1}} \label{Borel_SCP}
\end{multline}
with initial data
\begin{equation}
(\partial_{x}^{j}\hat{V})(\tau,z,0,\epsilon) = \sum_{\beta_{0},\ldots,\beta_{l} \geq 0}
V_{(\beta_{0},\ldots,\beta_{l},j)}(\tau,\epsilon)
\frac{ \exp( iz (\sum_{j=0}^{l} \beta_{j} \xi_{j}) ) }{\beta_{0}! \cdots \beta_{l}!} \ \ , \ \ 0 \leq j \leq S-1, \label{Borel_SCP_init_cond}
\end{equation}
where $V_{(\beta_{0},\ldots,\beta_{l},j)}(\tau,\epsilon)$ are defined in (\ref{V_beta_j_defin}), for all
$\epsilon \in D(0,\epsilon_{0}) \setminus \{ 0 \}$. In the following, we rewrite the equation (\ref{Borel_SCP}) using the two following technical
lemma. Their proofs can be found in \cite{ma1}, Lemma 5 and Lemma 6. Therefore we omit them.

\begin{lemma} For all $k_{0} \geq 1$, there exist constants $a_{k,k_{0}} \in \mathbb{N}$, $k_{0} \leq k \leq 2k_{0}$, such that
\begin{equation}
(\tau \partial_{\tau}^{2} + \partial_{\tau})^{k_0}u(\tau) = \sum_{k=k_{0}}^{2k_{0}} a_{k,k_{0}} \tau^{k-k_{0}} \partial_{\tau}^{k}u(\tau) 
\end{equation}
for all holomorphic functions $u : \Omega \rightarrow \mathbb{C}$ on an open set $\Omega \subset \mathbb{C}$.
 \end{lemma}

\begin{lemma} Let $a,b,c \geq 0$ be positive integers such that $a \geq b$ and $a \geq c$. We put $\delta = a+b-c$. Then, for all
holomorphic functions
$u : \Omega \rightarrow \mathbb{C}$, the function $\partial_{\tau}^{-a}(\tau^{b}\partial_{\tau}^{c}u(\tau))$ can be written in the form
$$ \partial_{\tau}^{-a}(\tau^{b}\partial_{\tau}^{c}u(\tau)) = \sum_{(b',c') \in \mathcal{O}_{\delta}} \alpha_{b',c'} \tau^{b'}
\partial_{\tau}^{c'}u(\tau) $$
where $\mathcal{O}_{\delta}$ is a finite subset of $\mathbb{Z}^{2}$ such that for all $(b',c') \in \mathcal{O}_{\delta}$, $b' - c' = \delta$, 
$b' \geq 0$, $c' \leq 0$, and $\alpha_{b',c'} \in \mathbb{Z}$.
\end{lemma}

Using the latter Lemma 7 and Lemma 8 together with the assumption (\ref{shape_SCP}), we can rewrite the equation
(\ref{Borel_SCP}) in the form
\begin{multline}
(\tau^{r_2} + (-i\partial_{z}+1)^{r_1})\partial_{x}^{S}\hat{V}(\tau,z,x,\epsilon) =\\
\sum_{(s,k_{0},k_{1},k_{2}) \in \mathcal{S}} b_{s,k_{0},k_{1},k_{2}}(z,x,\epsilon) \epsilon^{r(k_{0}-s)}
(\sum_{(r',p') \in \mathcal{O}_{s-k_{0}}} \alpha_{r',p'} \tau^{r'}\partial_{\tau}^{-p'}\partial_{z}^{k_1}
\partial_{x}^{k_2}\hat{V}(\tau,z,x,\epsilon)) \\
+ \sum_{(l_{0},l_{1}) \in \mathcal{N}} c_{l_{0},l_{1}}(z,x,\epsilon) \epsilon^{-r(l_{0}+l_{1}-1)}
\partial_{\tau}^{-l_0}(\hat{V}(\tau,z,x,\epsilon))^{\ast l_{1}} \label{Borel_SCP_modif}
\end{multline}
where $\mathcal{O}_{s-k_{0}}$ is a finite subset of $\mathbb{N}^{2}$ such that for all $(r',p') \in \mathcal{O}_{s-k_{0}}$, we have
$r'+p'=s-k_{0}$ and $\alpha_{r',p'} \in \mathbb{Z}$, for the given initial data
\begin{equation}
(\partial_{x}^{j}\hat{V})(\tau,z,0,\epsilon) = \sum_{\beta_{0},\ldots,\beta_{l} \geq 0}
V_{(\beta_{0},\ldots,\beta_{l},j)}(\tau,\epsilon)
\frac{ \exp( iz (\sum_{j=0}^{l} \beta_{j} \xi_{j}) ) }{\beta_{0}! \cdots \beta_{l}!} \ \ , \ \ 0 \leq j \leq S-1, \label{Borel_SCP_init_cond_modif}
\end{equation}

>From the assumption (\ref{shape_SCP}), we deduce that the assumptions (\ref{assum_B}) and (\ref{assum_C}) of Proposition 9 are fulfilled 
for the equation (\ref{Borel_SCP_modif}). Hence, from Proposition 9, we deduce the existence of a constant $I>0$
(independent of $\epsilon$) such that
if the inequality (\ref{norm_varphi_j<I}) holds, then the formal series $\hat{V}(\tau,z,x,\epsilon)$ solution of
(\ref{Borel_SCP_modif}), (\ref{Borel_SCP_init_cond_modif}) defines a holomorphic function $V_{U_d}(\tau,z,x,\epsilon)$ on the product
$U_{d} \times H_{\rho_{1}'} \times D(0,\rho_{1}) \times D(0,\epsilon_{0}) \setminus \{ 0 \}$ which satisfies the next bound
estimates : there
exists a constant $C_{18}>0$ such that
\begin{equation}
|V_{U_d}(\tau,z,x,\epsilon)| \leq C_{18}(1 + \frac{|\tau|^2}{|\epsilon|^{2r}})^{-1}\exp(\frac{\sigma}{|\epsilon|^{r}}\zeta(b)|\tau|) 
\end{equation}
for all $(\tau,z,x,\epsilon) \in U_{d} \times H_{\rho_{1}'} \times D(0,\rho_{1}) \times D(0,\epsilon_{0}) \setminus \{ 0 \}$. Moreover,
from the proof of Proposition 9 (especially the formula (\ref{V_beta<})), we also get that each formal series
$\hat{V}_{(\beta_{0},\ldots,\beta_{l},\beta_{l+1})}(\tau,\epsilon)$ defines a holomorphic function
$V_{U_{d},(\beta_{0},\ldots,\beta_{l},\beta_{l+1})}(\tau,\epsilon)$ on
$(U_{d} \cup D(0,\rho_{(\beta_{0},\ldots,\beta_{l+1})})) \times D(0,\epsilon_{0}) \setminus \{ 0 \}$ with the following estimates : there
exist two constants $C_{19}>0$, $K_{19}>0$ such that
\begin{multline}
|V_{U_{d},(\beta_{0},\ldots,\beta_{l},\beta_{l+1})}(\tau,\epsilon)| \\
\leq IC_{19} \exp( \frac{\sigma}{|\epsilon|^r}
r_{b}((\beta_{0},\ldots,\beta_{l+1})) |\tau| )
(\frac{l+2}{M^{0}})^{\sum_{j=0}^{l} \beta_{j}} K_{19}^{\beta_{l+1}} \beta_{0}! \cdots \beta_{l+1}! \label{V_Ud_beta<}
\end{multline}
for all $\tau \in D(0,\rho_{(\beta_{0},\ldots,\beta_{l+1})}) \cup U_{d}$ and $\epsilon \in D(0,\epsilon_{0}) \setminus \{ 0 \}$. From
(\ref{V_Ud_beta<}), we get that each formal series $\hat{Y}_{(\beta_{0},\ldots,\beta_{l},\beta_{l+1})}(T,\epsilon)$ is
$1$-summable with respect to $T$ is the direction $d$ and its $1-$sum denoted by
$Y_{U_{d},D(0,\epsilon_{0}) \setminus \{ 0 \},(\beta_{0},\ldots,\beta_{l},\beta_{l+1})}(T,\epsilon)$ can be written as Laplace transform
of order 1 in the direction $d$ of the function $\tau \mapsto V_{U_{d},(\beta_{0},\ldots,\beta_{l},\beta_{l+1})}(\tau,\epsilon)$. Moreover,
again by (\ref{V_Ud_beta<}), we deduce that the series
\begin{multline}
Y_{U_{d},D(0,\epsilon_{0}) \setminus \{ 0 \}}(T,z,x,\epsilon) \\
= \sum_{\beta_{0},\ldots,\beta_{l},\beta_{l+1} \geq 0}
Y_{U_{d},D(0,\epsilon_{0}) \setminus \{ 0 \},(\beta_{0},\ldots,\beta_{l},\beta_{l+1})}(T,\epsilon)
\frac{ \exp( iz (\sum_{j=0}^{l} \beta_{j} \xi_{j}) ) }{\beta_{0}! \cdots \beta_{l}!}\frac{x^{\beta_{l+1}}}{\beta_{l+1}!}
 \end{multline}
defines a bounded holomorphic function on $U_{d,\theta,h'|\epsilon|^{r}} \times H_{\rho_{1}'} \times D(0,\rho_{1})$ for some $h'>0$
and $\rho_{1}>0$, for all $\epsilon \in D(0,\epsilon_{0}) \setminus \{ 0 \}$. Finally, from the algebraic properties of the $\kappa-$summation
procedure, see \cite{ba} Section 6.3, since $\hat{Y}(T,z,x,\epsilon)$ formally solves (\ref{SCP}),
we deduce that $Y_{U_{d},D(0,\epsilon_{0}) \setminus \{ 0 \}}(T,z,x,\epsilon)$ is an actual solution of
the Cauchy problem (\ref{SCP}), (\ref{SCP_init_cond}).
\end{proof}

\section{Formal series solutions and Gevrey asymptotic expansion in a complex parameter for the main Cauchy problem}

\subsection{Analytic solutions in a complex parameter for the main Cauchy problem}

We recall the definition of a good covering.

\begin{defin} Let $\nu \geq 2$ be an integer. For all $0 \leq i \leq \nu-1$, we consider an open sector $\mathcal{E}_{i}$ with vertex at 0
and with radius $\epsilon_{0}$. We assume that these sectors are three by three disjoint and that
$\mathcal{E}_{i+1} \cap \mathcal{E}_{i} \neq \emptyset$, for all $0 \leq i \leq \nu-1$, where by convention we define
$\mathcal{E}_{\nu}=\mathcal{E}_{0}$. Moreover,
we assume that $\cup_{0 \leq i \leq \nu-1} \mathcal{E}_{i} = \mathcal{U} \setminus \{ 0 \}$ where $\mathcal{U}$ is some neighborhood
of 0 in $\mathbb{C}$. Such a set of sectors $\{ \mathcal{E}_{i} \}_{0 \leq i \leq \nu-1}$ is called a good covering in $\mathbb{C}^{\ast}$.
\end{defin}

\begin{defin} Let $\{ \mathcal{E}_{i} \}_{0 \leq i \leq \nu-1}$ be a good covering in $\mathbb{C}^{\ast}$. Let $r>0$ be a positive real number
and $r_{2} \geq 1$ be some integer. Let $\mathcal{T}$ be an open sector with vertex at 0 with radius $r_{\mathcal{T}}>0$. We consider
the following family of open sectors
$$ U_{d_{i},\theta,\epsilon_{0}^{r}r_{\mathcal{T}}} = \{ t \in \mathbb{C}^{\ast} : |t|<\epsilon_{0}^{r}r_{\mathcal{T}} \ \ , \ \
|d_{i} - \mathrm{arg}(t)| < \theta/2 \} $$
where $d_{i} \in \mathbb{R}$, $0 \leq i \leq \nu-1$ and $\theta > \pi$, which satisfy the following properties:\medskip

\noindent 1) For all $0 \leq i \leq \nu-1$, $d_{i} \neq \pi \frac{2k+1}{r_{2}}$ for all $0 \leq k \leq r_{2}-1$.\\
\noindent 2) For all $0 \leq i \leq \nu-1$, all $t \in \mathcal{T}$, all $\epsilon \in \mathcal{E}_{i}$, we have that
$\epsilon^{r} t \in U_{d_{i},\theta,\epsilon_{0}^{r}r_{\mathcal{T}}}$.\medskip

Under the above settings, we say that the family $\{ \{U_{d_{i},\theta,\epsilon_{0}^{r}r_{\mathcal{T}}} \}_{0 \leq i \leq \nu-1}, \mathcal{T} \}$
is associated to the good covering $\{ \mathcal{E}_{i} \}_{0 \leq i \leq \nu-1}$.
\end{defin}

Let $S \geq 1$ be an integer. Let $\mathcal{S}$ be a finite subset of $\mathbb{N}^{4}$, $\mathcal{N}$ be a finite
subset of $\mathbb{N}^{2}$. 

As in the previous section, for all $(s,k_{0},k_{1},k_{2}) \in \mathcal{S}$, all integers $\beta_{0},\beta_{l+1} \geq 0$, we denote
$b_{s,k_{0},k_{1},k_{2},\beta_{0},\beta_{l+1}}(\epsilon)$ some holomorphic function on $D(0,\epsilon_{0})$ which satisfies the next
estimates : there exist constants $\rho,\rho'>0$, $\mathfrak{b} _{s,k_{0},k_{1},k_{2}} > 0$ with
\begin{equation}
\sup_{\epsilon \in D(0,\epsilon_{0})}
|b_{s,k_{0},k_{1},k_{2},\beta_{0},\beta_{l+1}}(\epsilon)| \leq \mathfrak{b} _{s,k_{0},k_{1},k_{2}}
(\frac{e^{-\rho'}}{2})^{\beta_0}(\frac{1}{2 \rho})^{\beta_{l+1}} \beta_{0}! \beta_{l+1}! \label{b_sk_beta_decay}
\end{equation}
for all $\beta_{0},\beta_{l+1} \geq 0$. Likewise, for all $(l_{0},l_{1}) \in \mathcal{N}$, all integers $\beta_{0},\beta_{l+1} \geq 0$, we denote
$c_{l_{0},l_{1},\beta_{0},\beta_{l+1}}(\epsilon)$ some holomorphic function on $D(0,\epsilon_{0})$ with the following estimates :
there exist a constant $\mathfrak{c} _{l_{0},l_{1}} > 0$ with
\begin{equation}
\sup_{\epsilon \in D(0,\epsilon_{0})}
|c_{l_{0},l_{1},\beta_{0},\beta_{l+1}}(\epsilon)| \leq \mathfrak{c} _{l_{0},l_{1}}
(\frac{e^{-\rho'}}{2})^{\beta_0}(\frac{1}{2 \rho})^{\beta_{l+1}} \beta_{0}! \beta_{l+1}! \label{c_l_beta_decay}
\end{equation}
for all $\beta_{0},\beta_{l+1} \geq 0$.

For all $(s,k_{0},k_{1},k_{2}) \in \mathcal{S}$ and all $(l_{0},l_{1}) \in \mathcal{N}$, we consider the series
\begin{multline*}
 b_{s,k_{0},k_{1},k_{2}}(z,x,\epsilon) = \sum_{\beta_{0},\beta_{l+1} \geq 0} b_{s,k_{0},k_{1},k_{2},\beta_{0},\beta_{l+1}}(\epsilon)
\frac{e^{iz \beta_{0}}}{\beta_{0}!} \frac{x^{\beta_{l+1}}}{\beta_{l+1}!},\\
c_{l_{0},l_{1}}(z,x,\epsilon) = \sum_{\beta_{0},\beta_{l+1} \geq 0} c_{l_{0},l_{1},\beta_{0},\beta_{l+1}}(\epsilon)
\frac{e^{iz \beta_{0}}}{\beta_{0}!} \frac{x^{\beta_{l+1}}}{\beta_{l+1}!}
\end{multline*}
which define bounded holomorphic functions on $H_{\rho'} \times D(0,\rho) \times D(0,\epsilon_{0})$.
Let $\{ \mathcal{E}_{i} \}_{0 \leq i \leq \nu-1}$ be a good covering in $\mathbb{C}^{\ast}$ and let
$r_{1},r_{2},r_{3} \geq 1$ be three integers. We put $r=r_{3}/r_{2}$.

For all $0 \leq i \leq \nu-1$, we consider the following Cauchy problem
\begin{multline}
( \epsilon^{r_{3}}(t^{2}\partial_{t} + t)^{r_2} + (-i\partial_{z}+1)^{r_1}) \partial_{x}^{S}X_{i}(t,z,x,\epsilon) \\
= \sum_{(s,k_{0},k_{1},k_{2}) \in \mathcal{S}} b_{s,k_{0},k_{1},k_{2}}(z,x,\epsilon)
t^{s}(\partial_{t}^{k_0}\partial_{z}^{k_1}\partial_{x}^{k_2}X_{i})(t,z,x,\epsilon)\\
+ \sum_{(l_{0},l_{1}) \in \mathcal{N}} c_{l_{0},l_{1}}(z,x,\epsilon)
t^{l_{0}+l_{1}-1}(X_{i}(t,z,x,\epsilon))^{l_{1}} \label{main_CP}
\end{multline}
for given initial data
\begin{equation}
(\partial_{x}^{j}X_{i})(t,z,0,\epsilon) = \Xi_{i,j}(t,z,\epsilon) \ \ , \ \ 0 \leq j \leq S-1 \label{main_CP_init_cond}
\end{equation}
where the functions $\Xi_{i,j}$ are constructed as follows. We consider a family of sectors\\
$\{ \{U_{d_{i},\theta,\epsilon_{0}^{r}r_{\mathcal{T}}} \}_{0 \leq i \leq \nu-1}, \mathcal{T} \}$ associated to the good covering
$\{ \mathcal{E}_{i} \}_{0 \leq i \leq \nu-1}$. For all $0 \leq i \leq \nu-1$, let $U_{d_i}$ be an unbounded open sector centered at 0, with bisecting
direction $d_{i}$ and with aperture $n_{i} > \theta - \pi$. We choose $\theta$ and $n_{i}$ in such a way that
$$ \mathrm{arg}(\tau) \neq \pi \frac{2k+1}{r_{2}} $$
for all $\tau \in U_{d_i}$, all $0 \leq i \leq \nu-1$, all $0 \leq k \leq r_{2}-1$. For all $0 \leq i \leq \nu-1$, all $0 \leq j \leq S-1$, we define
$$ \Xi_{i,j}(t,z,\epsilon) = Y_{U_{d_i},D(0,\epsilon_{0})\setminus \{ 0 \},j}(\epsilon^{r}t,z,\epsilon) $$
where $Y_{U_{d_i},D(0,\epsilon_{0})\setminus \{ 0 \},j}(T,z,\epsilon)$ is given by the expression (\ref{defin_Y_Ud_j})
and constructed as in the beginning of Section 3.2 with the help of a family of functions
$V_{U_{d_i},(\beta_{0},\ldots,\beta_{l},j)}(\tau,\epsilon)$, for $(\beta_{0},\ldots,\beta_{l}) \in \mathbb{N}^{l+1}$, $0 \leq j \leq S-1$,
satisfying (\ref{V_beta_j_defin}).

We make the additional assumption that for all $(\beta_{0},\ldots,\beta_{l}) \in \mathbb{N}^{l+1}$, $0 \leq j \leq S-1$,
there exists a holomorphic function 
$\tau \mapsto V_{(\beta_{0},\ldots,\beta_{l},j)}(\tau,\epsilon)$ on $D(0,\rho_{(\beta_{0},\ldots,\beta_{l+1})})$ for all
$\epsilon \in D(0,\epsilon_{0}) \setminus \{ 0 \}$ such that
\begin{equation}
V_{U_{d_i},(\beta_{0},\ldots,\beta_{l},j)}(\tau,\epsilon) = V_{(\beta_{0},\ldots,\beta_{l},j)}(\tau,\epsilon) \label{V_Udi_j_analyt_cont_V_j}
\end{equation}
for all $0 \leq i \leq \nu-1$, all $0 \leq j \leq S-1$, all $\tau \in D(0,\rho_{(\beta_{0},\ldots,\beta_{l+1})})$, all
$\epsilon \in D(0,\epsilon_{0}) \setminus \{ 0 \}$.

Moreover, we assume that the series $\varphi_{d_{i},j}(Z_{0},\ldots,Z_{l},\epsilon)$ defined in
(\ref{varphi_j_with_norm_V_beta_j}) belong to the Banach space $G(\bar{Z}_{0}^{0},\ldots,\bar{Z}_{l}^{0},\bar{X}^{0})$
for all $\epsilon \in D(0,\epsilon_{0}) \setminus \{ 0 \}$ where $\bar{X}^{0}>0$ and $\bar{Z}_{j}^{0}$, $0 \leq j \leq l$, are chosen in
such a way that $\bar{Z}_{j}^{0}>M^{0}$ for $M^{0}>0$ that fulfills the inequality (\ref{choice_M0_rho_1_prime}) for some
real number $\rho_{1}'>0$.

By construction, $\Xi_{i,j}(t,z,\epsilon)$ defines a holomorphic function on $\mathcal{T} \times H_{\rho_{1}'} \times \mathcal{E}_{i}$, for
all $0 \leq i \leq \nu-1$, all $0 \leq j \leq S-1$, for well chosen radius $r_{\mathcal{T}}>0$ and aperture $\theta$.

\begin{prop} Let the initial data (\ref{main_CP_init_cond}) constructed as above. We make the following assumptions:
for all $(s,k_{0},k_{1},k_{2}) \in \mathcal{S}$, all $(l_{0},l_{1}) \in \mathcal{N}$, we have that
\begin{equation}
s \geq 2k_{0} \ \ , \ \ S > k_{2} \ \ , \ \ S > b(s-k_{0}+2) + k_{2} \ \ , \ \ S \geq hr_{1} + b(s-k_{0}+2) + k_{1} + k_{2} \ \ , \ \
l_{1} \geq 2.  \label{shape_main_CP}
\end{equation}
Then, there exists a constant $I>0$ (independent of $\epsilon$) such that if one assumes that
\begin{equation}
|| \varphi_{d_{i},j}(Z_{0},\ldots,Z_{l},\epsilon)||_{(\bar{Z}_{0}^{0},\ldots,\bar{Z}_{l}^{0},\bar{X}^{0})} < I
\label{norm_varphi_j_d_i<I}
\end{equation}
for all $0 \leq i \leq \nu-1$, for all $\epsilon \in D(0,\epsilon_{0}) \setminus \{ 0 \}$, the problem
(\ref{main_CP}), (\ref{main_CP_init_cond}) has a solution $X_{i}(t,z,x,\epsilon)$ which is holomorphic and bounded on
$(\mathcal{T} \cap D(0,h')) \times H_{\rho_{1}'} \times D(0,\rho_{1}) \times \mathcal{E}_{i}$, for some $\rho_{1}>0$.
Moreover, there exist constants $K_{23},M_{23}>0$ and $0<h''<h'$ such that
\begin{equation}
\sup_{t \in \mathcal{T} \cap D(0,h''),z \in H_{\rho_{1}'},x \in D(0,\rho_{1})} |X_{i+1}(t,z,x,\epsilon) - X_{i}(t,z,x,\epsilon)| \leq K_{23}
\exp( - \frac{M_{23}}{|\epsilon|^{r_{3}/(hr_{1}+r_{2})}} )
\label{|X_i+1_minus_X_i|<}
\end{equation}
for all $\epsilon \in \mathcal{E}_{i+1} \cap \mathcal{E}_{i}$, all $0 \leq i \leq \nu-1$, (where by convention $X_{\nu} = X_{0}$),
provided that $\epsilon_{0}>0$ is small enough.
\end{prop}
\begin{proof} For all $0 \leq i \leq \nu-1$, we consider the singular Cauchy problem (\ref{SCP}) with initial data
\begin{equation}
 (\partial_{x}^{j}Y_{U_{d_i},D(0,\epsilon_{0}) \setminus \{ 0 \}})(T,z,0,\epsilon) =
Y_{U_{d_i},D(0,\epsilon_{0}) \setminus \{ 0 \},j}(T,z,\epsilon) \ \ , \ \ 0 \leq j \leq S-1. \label{SCP_init_cond_with_i}
\end{equation}
Bearing in mind the hypothesis (\ref{shape_main_CP}) and the assumption (\ref{norm_varphi_j_d_i<I}), we see that the assumptions of the
proposition 11 are all fulfilled for the problem (\ref{SCP}), (\ref{SCP_init_cond_with_i}), which therefore, possesses a solution
$(T,z,x) \mapsto Y_{U_{d_i},D(0,\epsilon_{0}) \setminus \{ 0 \}}(T,z,x,\epsilon)$, holomorphic and bounded on 
$U_{d_{i},\theta,h'|\epsilon|^{r}} \times H_{\rho_{1}'} \times D(0,\rho_{1})$ for some $h'>0$ and $\rho_{1}>0$, for all
$\epsilon \in D(0,\epsilon_{0}) \setminus \{ 0 \}$. Now, we put
$$ X_{i}(t,z,x,\epsilon) = Y_{U_{d_i},D(0,\epsilon_{0}) \setminus \{ 0 \}}(\epsilon^{r}t,z,x,\epsilon) $$
which defines a holomorphic and bounded function on $(\mathcal{T} \cap D(0,h') ) \times H_{\rho_{1}'} \times
D(0,\rho_{1}) \times \mathcal{E}_{i}$, for all $0 \leq i \leq \nu-1$, by construction of $\mathcal{T}$ and $\mathcal{E}_{i}$ in Definition 5.
Since $Y_{U_{d_i},D(0,\epsilon_{0}) \setminus \{ 0 \}}(T,z,x,\epsilon)$ solves the problem (\ref{SCP}), (\ref{SCP_init_cond_with_i}), one can check
that $X_{i}(t,z,x,\epsilon)$ solves the problem (\ref{main_CP}), (\ref{main_CP_init_cond}) on
$(\mathcal{T} \cap D(0,h') ) \times H_{\rho_{1}'} \times D(0,\rho_{1}) \times \mathcal{E}_{i}$, for all $0 \leq i \leq \nu-1$.\medskip

In the next step of the proof, we show the estimates (\ref{|X_i+1_minus_X_i|<}). Let $0 \leq i \leq \nu-1$. Using Proposition 11, we can write
the function $X_{i}(t,z,x,\epsilon)$ as follows:
$$
X_{i}(t,z,x,\epsilon) = \sum_{\beta_{0},\ldots,\beta_{l},\beta_{l+1} \geq 0}
X_{i,(\beta_{0},\ldots,\beta_{l+1})}(t,\epsilon) \frac{ \exp( iz( \sum_{j=0}^{l} \beta_{j} \xi_{j}) ) }{\beta_{0}! \cdots \beta_{l}!}
\frac{x^{\beta_{l+1}}}{\beta_{l+1}!}
$$
where
$$ X_{i,(\beta_{0},\ldots,\beta_{l+1})}(t,\epsilon) = \frac{1}{\epsilon^{r}t} \int_{L_{\gamma_{i}}}
V_{U_{d_{i},(\beta_{0},\ldots,\beta_{l+1})}}(\tau,\epsilon) e^{-\frac{\tau}{\epsilon^{r}t}} d\tau $$
with integration path $L_{\gamma_{i}}=\mathbb{R}_{+}e^{\sqrt{-1}\gamma_{i}} \subset U_{d_i}$ and such that
$\tau \mapsto V_{U_{d_{i},(\beta_{0},\ldots,\beta_{l+1})}}(\tau,\epsilon)$ are holomorphic functions on
$D(0,\rho_{(\beta_{0},\ldots,\beta_{l+1})}) \cup U_{d_i}$, for all $\epsilon \in D(0,\epsilon_{0}) \setminus \{ 0 \}$ and satisfy the estimates:
there exist constants $C_{20},K_{20}>0$, which satisfies $\rho_{1}K_{20}<1$, with
\begin{equation}
|V_{U_{d_i},(\beta_{0},\ldots,\beta_{l},\beta_{l+1})}(\tau,\epsilon)| \leq C_{20} \exp( \frac{\sigma}{|\epsilon|^r} \zeta(b) |\tau| )
(\frac{l+2}{M^{0}})^{\sum_{j=0}^{l} \beta_{j}} K_{20}^{\beta_{l+1}} \beta_{0}! \cdots \beta_{l+1}! \label{|V_U_di|<}
\end{equation}
for all $\tau \in D(0,\rho_{(\beta_{0},\ldots,\beta_{l},\beta_{l+1})}) \cup U_{d_i}$ and $\epsilon \in D(0,\epsilon_{0}) \setminus \{ 0 \}$.
Moreover, from the assumption (\ref{V_Udi_j_analyt_cont_V_j}), we deduce with the help of the recusion (\ref{recursion_V_beta}),
which is satisfied for the coefficients of the formal solution $\hat{V}(\tau,z,x,\epsilon)$ of the problem
(\ref{Borel_SCP_modif}), (\ref{Borel_SCP_init_cond_modif}), that for all $(\beta_{0},\ldots,\beta_{l},\beta_{l+1}) \in \mathbb{N}^{l+2}$,
there exists a holomorphic function
$\tau \mapsto V_{(\beta_{0},\ldots,\beta_{l},\beta_{l+1})}(\tau,\epsilon)$ on $D(0,\rho_{(\beta_{0},\ldots,\beta_{l},\beta_{l+1})})$
for all $\epsilon \in D(0,\epsilon_{0}) \setminus \{ 0 \}$ such that
\begin{equation}
V_{U_{d_i},(\beta_{0},\ldots,\beta_{l},\beta_{l+1})}(\tau,\epsilon) = V_{(\beta_{0},\ldots,\beta_{l},\beta_{l+1})}(\tau,\epsilon)
\label{V_Udi_beta_analyt_cont_V_beta}
\end{equation}
for all $\tau \in D(0,\rho_{(\beta_{0},\ldots,\beta_{l},\beta_{l+1})})$, all
$\epsilon \in D(0,\epsilon_{0}) \setminus \{ 0 \}$, all $0 \leq i \leq \nu-1$.\medskip

We show the following

\begin{lemma} There exist constants $0 < h'' < h'$, $C_{21},K_{21},M_{21}>0$ (independent of $\epsilon$),
which satisfies $\rho_{1}K_{21}<1$, such that
\begin{multline}
\sup_{t \in \mathcal{T} \cap D(0,h'')} |X_{i+1,(\beta_{0},\ldots,\beta_{l+1})}(t,\epsilon) - X_{i,(\beta_{0},\ldots,\beta_{l+1})}(t,\epsilon)|\\
\leq C_{21}(\frac{l+2}{M^{0}})^{\sum_{j=0}^{l} \beta_{j}}
K_{21}^{\beta_{l+1}} \beta_{0}! \cdots \beta_{l+1}!
e^{-M_{21}\frac{\rho_{(\beta_{0},\ldots,\beta_{l+1})}}{|\epsilon|^{r}}} \label{|X_i+1_beta_minus_X_i_beta|<}
\end{multline}
for all $\epsilon \in \mathcal{E}_{i+1} \cap \mathcal{E}_{i}$, all $0 \leq i \leq \nu-1$, all
$(\beta_{0},\ldots,\beta_{l+1}) \in \mathbb{N}^{l+2}$ (where by convention $X_{\nu,(\beta_{0},\ldots,\beta_{l+1})} =
X_{0,(\beta_{0},\ldots,\beta_{l+1})}$).
\end{lemma}
\begin{proof} From the fact that the function $\tau \mapsto V_{(\beta_{0},\ldots,\beta_{l},\beta_{l+1})}(\tau,\epsilon)
e^{-\frac{\tau}{\epsilon^{r}t}}$ is
holomorphic on \\ $D(0,\rho_{(\beta_{0},\ldots,\beta_{l},\beta_{l+1})})$, for all $\epsilon \in D(0,\epsilon_{0})$, we deduce that
its integral along the union of a segment starting from 0 to $(\rho_{(\beta_{0},\ldots,\beta_{l+1})}/2)e^{\sqrt{-1}\gamma_{i+1}}$,
an arc of circle with radius $\rho_{(\beta_{0},\ldots,\beta_{l+1})}/2$ connecting
$(\rho_{(\beta_{0},\ldots,\beta_{l+1})}/2)e^{\sqrt{-1}\gamma_{i+1}}$ and
$(\rho_{(\beta_{0},\ldots,\beta_{l+1})}/2)e^{\sqrt{-1}\gamma_{i}}$ and a segment starting from\\
$(\rho_{(\beta_{0},\ldots,\beta_{l+1})}/2)e^{\sqrt{-1}\gamma_{i}}$ to the origin, is vanishing. Therefore, using the property
(\ref{V_Udi_beta_analyt_cont_V_beta}), we can write, for all
$(\beta_{0},\ldots,\beta_{l},\beta_{l+1}) \in \mathbb{N}^{l+2}$,
\begin{multline}
X_{i+1,(\beta_{0},\ldots,\beta_{l+1})}(t,\epsilon) - X_{i,(\beta_{0},\ldots,\beta_{l+1})}(t,\epsilon) \\
= \frac{1}{\epsilon^{r}t} \left(
\int_{L_{\rho_{(\beta_{0},\ldots,\beta_{l+1})}/2},\gamma_{i+1}} V_{U_{d_{i+1}},(\beta_{0},\ldots,\beta_{l+1})}(\tau,\epsilon)
e^{-\frac{\tau}{\epsilon^{r}t}} d\tau \right. \\
\left. - \int_{L_{\rho_{(\beta_{0},\ldots,\beta_{l+1})}/2},\gamma_{i}}
V_{U_{d_{i}},(\beta_{0},\ldots,\beta_{l+1})}(\tau,\epsilon) e^{-\frac{\tau}{\epsilon^{r}t}} d\tau \right. \\
\left. + \int_{C(\rho_{(\beta_{0},\ldots,\beta_{l+1})}/2,\gamma_{i},\gamma_{i+1})} V_{(\beta_{0},\ldots,\beta_{l+1})}(\tau,\epsilon)
e^{-\frac{\tau}{\epsilon^{r}t}} d\tau \right) \label{decomp_X_i+1_beta_minus_X_i_beta}
\end{multline}
where 
\begin{multline*}
L_{\rho_{(\beta_{0},\ldots,\beta_{l+1})}/2},\gamma_{i+1} = [\rho_{(\beta_{0},\ldots,\beta_{l+1})}/2,+\infty)
e^{\sqrt{-1}\gamma_{i+1}},\\
L_{\rho_{(\beta_{0},\ldots,\beta_{l+1})}/2},\gamma_{i} = [\rho_{(\beta_{0},\ldots,\beta_{l+1})}/2,+\infty)
e^{\sqrt{-1}\gamma_{i}}
\end{multline*}
and $C(\rho_{(\beta_{0},\ldots,\beta_{l+1})}/2,\gamma_{i},\gamma_{i+1})$ is an arc of circle centered at 0 with
radius $\rho_{(\beta_{0},\ldots,\beta_{l+1})}/2$ connecting
$(\rho_{(\beta_{0},\ldots,\beta_{l+1})}/2)e^{\sqrt{-1}\gamma_{i}}$ and
$(\rho_{(\beta_{0},\ldots,\beta_{l+1})}/2)e^{\sqrt{-1}\gamma_{i+1}}$ for a well chosen orientation.\medskip

First, we give estimates for
$$ I_{1} = | \frac{1}{\epsilon^{r}t}\int_{L_{\rho_{(\beta_{0},\ldots,\beta_{l+1})}/2},\gamma_{i+1}}
V_{U_{d_{i+1}},(\beta_{0},\ldots,\beta_{l+1})}(\tau,\epsilon)
e^{-\frac{\tau}{\epsilon^{r}t}} d\tau |. $$
By construction, the direction $\gamma_{i+1}$ (which depends on $\epsilon^{r}t$) is chosen in such a way that
$\cos(\gamma_{i+1} - \mathrm{arg}(\epsilon^{r}t)) \geq \delta_{1}$ for all $\epsilon \in \mathcal{E}_{i+1} \cap \mathcal{E}_{i}$,
all $t \in \mathcal{T} \cap D(0,h')$, for some fixed $\delta_{1}>0$. From the estimates (\ref{|V_U_di|<}), we deduce that
\begin{multline}
I_{1} \leq \frac{1}{|\epsilon^{r}t|} \int_{\rho_{(\beta_{0},\ldots,\beta_{l+1})}/2}^{+\infty}
C_{20}(\frac{l+2}{M^{0}})^{\sum_{j=0}^{l} \beta_{j}} K_{20}^{\beta_{l+1}} \beta_{0}! \cdots \beta_{l+1}!
e^{\frac{\sigma \zeta(b)}{|\epsilon|^{r}}h} e^{-\frac{h}{|\epsilon^{r}t|}\cos( \gamma_{i+1} - \mathrm{arg}(\epsilon^{r}t))} dh\\
\leq \frac{1}{|\epsilon^{r}t|}C_{20}(\frac{l+2}{M^{0}})^{\sum_{j=0}^{l} \beta_{j}}
K_{20}^{\beta_{l+1}} \beta_{0}! \cdots \beta_{l+1}! \int_{\rho_{(\beta_{0},\ldots,\beta_{l+1})}/2}^{+\infty}
e^{(\sigma \zeta(b) - \frac{\delta_{1}}{|t|})\frac{h}{|\epsilon|^{r}}} dh\\
= C_{20}(\frac{l+2}{M^{0}})^{\sum_{j=0}^{l} \beta_{j}}
K_{20}^{\beta_{l+1}} \beta_{0}! \cdots \beta_{l+1}!\frac{1}{\delta_{1} - \sigma \zeta(b)|t|}
e^{-(\frac{\delta_1}{|t|} - \sigma \zeta(b)) \frac{\rho_{(\beta_{0},\ldots,\beta_{l+1})}}{2|\epsilon|^{r}}} \\
\leq \frac{C_{20}}{\delta_{2}} (\frac{l+2}{M^{0}})^{\sum_{j=0}^{l} \beta_{j}}
K_{20}^{\beta_{l+1}} \beta_{0}! \cdots \beta_{l+1}!
e^{-\delta_{2}\frac{\rho_{(\beta_{0},\ldots,\beta_{l+1})}}{2|\epsilon|^{r}h'}} \label{I_1<}
\end{multline}
for all $t \in \mathcal{T} \cap D(0,h')$ with $|t|<(\delta_{1} - \delta_{2})/\sigma \zeta(b)$, for some $0 < \delta_{2} < \delta_{1}$ and
for all $\epsilon \in \mathcal{E}_{i+1} \cap \mathcal{E}_{i}$.\medskip

Now, we provide estimates for
$$ I_{2} = | \frac{1}{\epsilon^{r}t}\int_{L_{\rho_{(\beta_{0},\ldots,\beta_{l+1})}/2},\gamma_{i}}
V_{U_{d_{i}},(\beta_{0},\ldots,\beta_{l+1})}(\tau,\epsilon)
e^{-\frac{\tau}{\epsilon^{r}t}} d\tau |. $$
By construction, the direction $\gamma_{i}$ (wich depends on $\epsilon^{r}t$) is chosen in such a way that
$\cos(\gamma_{i} - \mathrm{arg}(\epsilon^{r}t)) \geq \delta_{1}$ for all $\epsilon \in \mathcal{E}_{i+1} \cap \mathcal{E}_{i}$,
all $t \in \mathcal{T} \cap D(0,h')$, for some fixed $\delta_{1}>0$. Again, from the estimates (\ref{|V_U_di|<}), we deduce as above that
\begin{equation}
I_{2} \leq \frac{C_{20}}{\delta_{2}} (\frac{l+2}{M^{0}})^{\sum_{j=0}^{l} \beta_{j}}
K_{20}^{\beta_{l+1}} \beta_{0}! \cdots \beta_{l+1}!
e^{-\delta_{2}\frac{\rho_{(\beta_{0},\ldots,\beta_{l+1})}}{2|\epsilon|^{r}h'}} \label{I_2<}
\end{equation}
for all $t \in \mathcal{T} \cap D(0,h')$ with $|t|<(\delta_{1} - \delta_{2})/\sigma \zeta(b)$, for some $0 < \delta_{2} < \delta_{1}$ and
for all $\epsilon \in \mathcal{E}_{i+1} \cap \mathcal{E}_{i}$.\medskip

Finally, we give upper bounds for
$$ I_{3} = | \int_{C(\rho_{(\beta_{0},\ldots,\beta_{l+1})}/2,\gamma_{i},\gamma_{i+1})} V_{(\beta_{0},\ldots,\beta_{l+1})}(\tau,\epsilon)
e^{-\frac{\tau}{\epsilon^{r}t}} d\tau|.$$
Due to (\ref{|V_U_di|<}) and bearing in mind the property (\ref{V_Udi_beta_analyt_cont_V_beta}), we deduce that
\begin{multline}
I_{3} \leq \frac{1}{|\epsilon^{r}t|}| \int_{\gamma_{i}}^{\gamma_{i+1}}
C_{20}(\frac{l+2}{M^{0}})^{\sum_{j=0}^{l} \beta_{j}} K_{20}^{\beta_{l+1}} \beta_{0}! \cdots \beta_{l+1}!\\
\times e^{\frac{\sigma \zeta(b)}{|\epsilon|^{r}} \frac{\rho_{(\beta_{0},\ldots,\beta_{l+1})}}{2}}
e^{-\frac{\rho_{(\beta_{0},\ldots,\beta_{l+1})}}{2} \frac{\cos( \theta - \mathrm{arg}(\epsilon^{r}t) ) }{|\epsilon^{r}t|} }
\frac{\rho_{(\beta_{0},\ldots,\beta_{l+1})}}{2} d\theta | \label{I_3<_1}
\end{multline}
By construction, the circle $C(\rho_{(\beta_{0},\ldots,\beta_{l+1})}/2,\gamma_{i},\gamma_{i+1})$ is chosen in order that
$\cos( \theta - \mathrm{arg}(\epsilon^{r}t) ) \geq \delta_{1}$ for all $\theta \in [\gamma_{i},\gamma_{i+1}]$ (if
$\gamma_{i} < \gamma_{i+1}$), for all $\theta \in [\gamma_{i+1},\gamma_{i}]$ (if $\gamma_{i+1} < \gamma_{i}$), for every
$t \in \mathcal{T} \cap D(0,h')$, $\epsilon \in \mathcal{E}_{i+1} \cap \mathcal{E}_{i}$. From (\ref{I_3<_1}), we get that
\begin{multline}
I_{3} \leq |\gamma_{i+1} - \gamma_{i}|
C_{20}(\frac{l+2}{M^{0}})^{\sum_{j=0}^{l} \beta_{j}} K_{20}^{\beta_{l+1}} \beta_{0}! \cdots \beta_{l+1}!
\frac{\rho_{(\beta_{0},\ldots,\beta_{l+1})}}{2}\frac{1}{|\epsilon^{r}t|}e^{-(\frac{\delta_{1}}{|t|} - \sigma \zeta(b))
\frac{\rho_{(\beta_{0},\ldots,\beta_{l+1})}}{2|\epsilon|^{r}}}\\
\leq |\gamma_{i+1} - \gamma_{i}|
C_{20}(\frac{l+2}{M^{0}})^{\sum_{j=0}^{l} \beta_{j}} K_{20}^{\beta_{l+1}} \beta_{0}! \cdots \beta_{l+1}!
\frac{\rho_{(\beta_{0},\ldots,\beta_{l+1})}}{2}\frac{1}{|\epsilon^{r}t|}
e^{-\frac{\delta_2}{4} \frac{\rho_{(\beta_{0},\ldots,\beta_{l+1})}}{|\epsilon|^{r}|t|}}
e^{-\frac{\delta_2}{4} \frac{\rho_{(\beta_{0},\ldots,\beta_{l+1})}}{|\epsilon|^{r}h'}} \label{I_3<_2}
\end{multline}
for all $t \in \mathcal{T} \cap D(0,h')$ with $|t|<(\delta_{1} - \delta_{2})/\sigma \zeta(b)$, for some $0 < \delta_{2} < \delta_{1}$ and
for all $\epsilon \in \mathcal{E}_{i+1} \cap \mathcal{E}_{i}$. Regarding the inequality (\ref{xexpx<}), we deduce from (\ref{I_3<_2}) that
\begin{equation}
I_{3} \leq \frac{2e^{-1}|\gamma_{i+1} - \gamma_{i}|}{\delta_{2}}
C_{20}(\frac{l+2}{M^{0}})^{\sum_{j=0}^{l} \beta_{j}} K_{20}^{\beta_{l+1}} \beta_{0}! \cdots \beta_{l+1}!
e^{-\frac{\delta_2}{4} \frac{\rho_{(\beta_{0},\ldots,\beta_{l+1})}}{|\epsilon|^{r}h'}} \label{I_3<_3}
\end{equation}
for all $t \in \mathcal{T} \cap D(0,h')$ with $|t|<(\delta_{1} - \delta_{2})/\sigma \zeta(b)$, for some $0 < \delta_{2} < \delta_{1}$ and
for all $\epsilon \in \mathcal{E}_{i+1} \cap \mathcal{E}_{i}$.\medskip

Finally, gathering the decomposition (\ref{decomp_X_i+1_beta_minus_X_i_beta}) and the estimates
(\ref{I_1<}), (\ref{I_2<}), (\ref{I_3<_3}), we get the inequality (\ref{|X_i+1_beta_minus_X_i_beta|<}).
\end{proof}
>From Lemma 9 and taking into account the assumption (\ref{choice_M0_rho_1_prime}), we can write
\begin{multline}
\sup_{t \in \mathcal{T} \cap D(0,h''),z \in H_{\rho_{1}'},x \in D(0,\rho_{1})}
|X_{i+1}(t,z,x,\epsilon) - X_{i}(t,z,x,\epsilon)|\\
\leq \sum_{\beta_{0},\ldots,\beta_{l},\beta_{l+1} \geq 0} C_{21} e^{-M_{21}\frac{\rho_{(\beta_{0},\ldots,\beta_{l+1})}}{|\epsilon|^r}}
( \frac{(l+2)e^{\rho_{1}'(\max_{j=0}^{l}|\xi_{j}|)}}{M^{0}} )^{\sum_{j=0}^{l} \beta_{j}} (\rho_{1}K_{21})^{\beta_{l+1}}\\
\leq C_{21} \sum_{\beta_{0},\ldots,\beta_{l},\beta_{l+1} \geq 0} e^{-M_{21}\frac{\rho_{(\beta_{0},\ldots,\beta_{l+1})}}{|\epsilon|^r}}
( \max \{ \rho_{1}K_{21},\frac{1}{2} \} )^{\sum_{j=0}^{l+1} \beta_{j}} \label{|X_i+1_minus_X_i|<_1}
\end{multline}
for all $\epsilon \in \mathcal{E}_{i+1} \cap \mathcal{E}_{i}$, all $0 \leq i \leq \nu-1$. Moreover, we recall that
for all integers $l \geq 1$, $\kappa \geq 0$,
\begin{equation}
\mathrm{Card} \{ (\beta_{0},\ldots,\beta_{l+1}) \in \mathbb{N}^{l+2}/ \sum_{j=0}^{l+1} \beta_{j} = \kappa \} =
\frac{(\kappa + l +1)!}{(l+1)!\kappa!} = \frac{\mathcal{P}_{l}(\kappa)}{(l+1)!} \label{card_sum_l_kappa}
\end{equation}
where $\mathcal{P}_{l}(\kappa)=(\kappa+l+1)(\kappa+l)\cdots(\kappa+1)$ is a polynomial of degree $l+1$ in $\kappa$. From
the definition (\ref{defin_rho_beta_gp}) and with the help of (\ref{card_sum_l_kappa}), we get two constants $C_{22}>0$ and
$0<K_{22}<1$ such that
\begin{multline}
C_{21} \sum_{\beta_{0},\ldots,\beta_{l},\beta_{l+1} \geq 0} e^{-M_{21}\frac{\rho_{(\beta_{0},\ldots,\beta_{l+1})}}{|\epsilon|^r}}
( \max \{ \rho_{1}K_{21},\frac{1}{2} \} )^{\sum_{j=0}^{l+1} \beta_{j}}\\
\leq C_{21} \sum_{\kappa \geq 0} \frac{\mathcal{P}_{l}(\kappa)}{(l+1)!}
\exp( - M_{21} \frac{ C_{\xi_{1},\ldots,\xi_{l}}^{r_{1}/r_{2}} }{2(1+\kappa)^{hr_{1}/r_{2}}|\epsilon|^{r}} )(
\max\{ \rho_{1}K_{21},\frac{1}{2} \})^{\kappa}\\
\leq C_{22} \sum_{\kappa \geq 0} \exp( - M_{21} \frac{ C_{\xi_{1},\ldots,\xi_{l}}^{r_{1}/r_{2}} }{2(1+\kappa)^{hr_{1}/r_{2}}|\epsilon|^{r}} )
(K_{22})^{\kappa} \label{majorize_Dirichlet_series}
\end{multline}
for all $\epsilon \in \mathcal{E}_{i+1} \cap \mathcal{E}_{i}$, all $0 \leq i \leq \nu-1$. Now, we recall the following lemma from \cite{lamasa}.
\begin{lemma} Let $0 < a < 1$ and $\alpha > 0$. There exist $K,M>0$ and $\delta>0$ such that
$$ \sum_{\kappa \geq 0} e^{-\frac{1}{(\kappa +1)^{\alpha}}\frac{1}{\epsilon}} a^{\kappa} \leq K
\exp \left(  -M \epsilon^{-\frac{1}{\alpha + 1}} \right) $$
for all $\epsilon \in (0,\delta ]$.
\end{lemma}
>From Lemma 10 applied to the inequality (\ref{majorize_Dirichlet_series}) and from (\ref{|X_i+1_minus_X_i|<_1}), we deduce that
the estimates (\ref{|X_i+1_minus_X_i|<}) hold, if $\epsilon_{0}$ is chosen small enough.
\end{proof}

\subsection{Existence of formal power series solutions in the complex parameter for the main Cauchy problem}
 
This subsection is devoted to explain the main result of this work. Namely, we will establish the existence of a formal power
series
$$ \hat{X}(t,z,x,\epsilon) = \sum_{k \geq 0} H_{k}(t,z,x) \frac{\epsilon^{k}}{k!} \in \mathcal{O}((\mathcal{T} \cap D(0,h'')) \times
H_{\rho_{1}'} \times D(0,\rho_{1}))[[\epsilon]] $$
which solves formally the equation (\ref{main_CP}) and is constructed in such a way that the actual solutions
$X_{i}(t,z,x,\epsilon)$ of the problem (\ref{main_CP}), (\ref{main_CP_init_cond}) all have $\hat{X}$ as asymptotic expansion of
Gevrey order $\frac{hr_{1}+r_{2}}{r_3}$ on $\mathcal{E}_{i}$ (as formal series and functions with coefficients and values in
the Banach space $\mathcal{O}((\mathcal{T} \cap D(0,h'')) \times
H_{\rho_{1}'} \times D(0,\rho_{1}))$ equipped with the supremum norm) (see Definition 6 below), for all $0 \leq i \leq \nu-1$.

The proof makes use a Banach valued version of cohomological criterion for Gevrey asymptotic expansion of sectorial holomorphic functions
known in the litterature as the Ramis-Sibuya theorem. For a reference, we refer to \cite{ba}, Section 7.4 Proposition 18 and
\cite{hssi}, Lemma XI-2-6.\medskip

\begin{defin} Let $(\mathbb{E},\left\|\cdot\right\|_{\mathbb{E}})$ be a complex Banach space over $\mathbb{C}$. One considers a formal series
$\hat{G}(\epsilon)=\sum_{n\ge0}G_{n}\epsilon^{n}$ where the coefficients $G_{n}$ belong to $\mathbb{E}$ and a holomorphic function
$G : \mathcal{E} \rightarrow \mathbb{E}$ on an open bounded sector $\mathcal{E}$ with vertex at 0. Let $s>0$ be a positive real number.
One says that $G$ admits $\hat{G}$ as its asymptotic expansion of Gevrey order $s$ on $\mathcal{E}$ if 
for every proper and bounded subsector $T$ of $\mathcal{E}$, there exist two constants $K,M>0$ such that for all $N \geq 1$, one has
$$||G(\epsilon)-\sum_{n=0}^{N-1}G_{n}\epsilon^n||_{\mathbb{E}}\le KM^{N}N!^{s}|\epsilon|^{N}$$
for all $\epsilon \in T$.
\end{defin}

\noindent \textbf{Theorem (RS)}
\emph{Let $(\mathbb{E},\left\|\cdot\right\|_{\mathbb{E}})$ be a complex Banach space over $\mathbb{C}$. Let $\{\mathcal{E}_{i}\}_{0\le i\le \nu-1}$
be a good covering in $\mathbb{C}^{\star}$. For every $0\le i\le \nu-1$, let $G_{i}$
be a holomorphic function from $\mathcal{E}_{i}$ into $\mathbb{E}$, and let the
cocycle $\Delta_{i}(\epsilon):=G_{i+1}(\epsilon)-G_i(\epsilon)$ be a holomorphic function from
$Z_{i}:=\mathcal{E}_{i}\cap\mathcal{E}_{i+1}$ into $\mathbb{E}$ (with the convention that $\mathcal{E}_{\nu}=\mathcal{E}_{0}$ and
$G_{\nu}=G_0$). We assume that:
\begin{enumerate}
\item $G_{i}(\epsilon)$ is bounded as $\epsilon\in\mathcal{E}_{i}$ tends to 0, for every $0\le i\le\nu-1$,
\item $\Delta_{i}$ has an exponential decreasing of order $s>0$ on $Z_i$, for every $0\le i\le \nu-1$, meaning there exist
$C_i,A_i>0$ such that
\begin{equation}
\left\|\Delta_{i}(\epsilon)\right\|_{\mathbb{E}}\le C_{i}e^{-\frac{A_i}{|\epsilon|^{1/s}}}, \label{s_exp_flat}
\end{equation}
for every $\epsilon\in Z_i$ and $0\le i\le \nu-1$.
\end{enumerate}
Then, there exists a formal power series $\hat{G}(\epsilon)\in\mathbb{E}[[\epsilon]]$ such that
$G_i(\epsilon)$ admits $\hat{G}(\epsilon)$ as its asymptotic expansion of Gevrey order $s$ on $\mathcal{E}_{i}$, for every $0\le i\le \nu-1$.}\medskip

\noindent We now state the main result of our paper.\medskip

\begin{theo} Let us assume that the conditions (\ref{shape_main_CP}) hold. For all $0 \leq i \leq \nu-1$, $0 \leq j \leq S-1$, we also
assume that for the functions $\Xi_{i,j}(t,z,\epsilon)$ constructed in (\ref{main_CP_init_cond}) the inequalities (\ref{norm_varphi_j_d_i<I}) are fulfilled
for some constant $I>0$ given in Proposition 12. Let
$$ \mathbb{E} = \mathcal{O}((\mathcal{T} \cap D(0,h'')) \times
H_{\rho_{1}'} \times D(0,\rho_{1})) $$
be the Banach space of holomorphic and bounded functions on $(\mathcal{T} \cap D(0,h'')) \times
H_{\rho_{1}'} \times D(0,\rho_{1})$ equipped with the supremum norm, where $h''$,$\rho_{1}'$,$\rho_{1}$ are the constants
appearing in Proposition 12.

Then, there exists a formal series
$$ \hat{X}(t,z,x,\epsilon) = \sum_{k \geq 0} H_{k}(t,z,x) \frac{\epsilon^{k}}{k!} \in \mathbb{E}[[\epsilon]] $$
which formally solves the equation
\begin{multline}
( \epsilon^{r_{3}}(t^{2}\partial_{t} + t)^{r_2} + (-i\partial_{z}+1)^{r_1}) \partial_{x}^{S}\hat{X}(t,z,x,\epsilon) \\
= \sum_{(s,k_{0},k_{1},k_{2}) \in \mathcal{S}} b_{s,k_{0},k_{1},k_{2}}(z,x,\epsilon)
t^{s}(\partial_{t}^{k_0}\partial_{z}^{k_1}\partial_{x}^{k_2}\hat{X}(t,z,x,\epsilon)\\
+ \sum_{(l_{0},l_{1}) \in \mathcal{N}} c_{l_{0},l_{1}}(z,x,\epsilon)
t^{l_{0}+l_{1}-1}(\hat{X}(t,z,x,\epsilon))^{l_{1}} \label{main_CP_formal}
\end{multline}
and is the Gevrey asymptotic expansion of order $\frac{hr_{1}+r_{2}}{r_{3}}$ of the $\mathbb{E}-$valued function
$\epsilon \in \mathcal{E}_{i} \mapsto X_{i}(t,z,x,\epsilon)$, solution of the problem (\ref{main_CP}), (\ref{main_CP_init_cond})
constructed in Proposition 12, for all $0 \leq i \leq \nu-1$.
\end{theo}
\begin{proof} We consider the functions $X_{i}(t,z,x,\epsilon)$, $0 \leq i \leq \nu-1$ constructed in Proposition 12. For all
$0 \leq i \leq \nu-1$, we define $G_{i}(\epsilon) := (t,z,x) \mapsto X_{i}(t,z,x,\epsilon)$, which is, by construction, a bounded holomorphic
function from $\mathcal{E}_{i}$ into the Banach space $\mathbb{E}$ of holomorphic and bounded functions on $(\mathcal{T} \cap D(0,h''))
\times H_{\rho_{1}'} \times D(0,\rho_{1})$ equipped with the supremum norm, where $h''$,$\rho_{1}'$,$\rho_{1}$ are constants
appearing in Proposition 12. Bearing in mind the estimates (\ref{|X_i+1_minus_X_i|<}), we deduce that the cocycle
$\Delta_{i}(\epsilon) = G_{i+1}(\epsilon) - G_{i}(\epsilon)$ fulfills estimates of the form (\ref{s_exp_flat}) on
$Z_{i}=\mathcal{E}_{i+1} \cap \mathcal{E}_{i}$, where $s = \frac{hr_{1}+r_{2}}{r_{3}}$, for all $0 \leq i \leq \nu-1$. According to
Theorem \textbf{(RS)} stated above, we deduce the existence of a formal series $\hat{G}(\epsilon) \in \mathbb{E}[[\epsilon]]$ which is
the Gevrey asymptotic expansion of order $\frac{hr_{1}+r_{2}}{r_{3}}$ of $G_{i}(\epsilon)$ on $\mathcal{E}_{i}$, for all
$0 \leq i \leq \nu-1$. Let us define
\begin{equation}
\hat{G}(\epsilon) = \hat{X}(t,z,x,\epsilon) = \sum_{k \geq 0} H_{k}(t,z,x) \frac{\epsilon^{k}}{k!}. \label{defin_hatG_hatX}
\end{equation}
It only remains to show that $\hat{X}$ is a formal solution of the equation (\ref{main_CP_formal}). From the fact that $G_{i}(\epsilon)$ admits
$\hat{G}(\epsilon)$ as its asymptotic expansion at 0 on $\mathcal{E}_{i}$, one gets
\begin{equation}
\lim_{\epsilon \rightarrow 0, \epsilon \in \mathcal{E}_{i}}
\sup_{t \in \mathcal{T} \cap D(0,h''),z \in H_{\rho_{1}'},x \in D(0,\rho_{1})} | \partial_{\epsilon}^{l}X_{i}(t,z,x,\epsilon) - H_{l}(t,z,x) |=0
\label{lim_partial_l_epsilon_X_i}
\end{equation}
for all $l \geq 0$, all $0 \leq i \leq \nu-1$. Now, we choose an integer $i \in \{ 0, \ldots, \nu-1 \}$. By construction, the function
$X_{i}(t,z,x,\epsilon)$ solves the equation (\ref{main_CP}). We differentiate it $l$ times with respect to
$\epsilon$. By means of Leibniz's rule, we get that $\partial_{\epsilon}^{l}X_{i}(t,z,x,\epsilon)$ satisfies the identity
\begin{multline}
\sum_{h_{1}+h_{2}=l} \frac{l!}{h_{1}!h_{2}!}
\partial_{\epsilon}^{h_1}(\epsilon^{r_3}) (t^{2}\partial_{t}+t)^{r_2}\partial_{x}^{S}\partial_{\epsilon}^{h_2}
X_{i}(t,z,x,\epsilon) + (-i\partial_{z}+1)^{r_1}\partial_{x}^{S}\partial_{\epsilon}^{l}X_{i}(t,z,x,\epsilon)\\
= \sum_{(s,k_{0},k_{1},k_{2}) \in \mathcal{S}} t^{s} (\sum_{h_{1}+h_{2}=l} \frac{l!}{h_{1}!h_{2}!}
\partial_{\epsilon}^{h_1}b_{s,k_{0},k_{1},k_{2}}(z,x,\epsilon)(\partial_{t}^{k_0}\partial_{z}^{k_1}\partial_{x}^{k_2}
\partial_{\epsilon}^{h_2}X_{i})(t,z,x,\epsilon))\\
+  \sum_{(l_{0},l_{1}) \in \mathcal{N}} t^{l_{0}+l_{1}-1} \sum_{h_{0}+h_{1}+\ldots+h_{l_1}=l} \frac{l!}{h_{0}! \cdots h_{l_1}!}
\partial_{\epsilon}^{h_0}c_{l_{0},l_{1}}(z,x,\epsilon) \Pi_{j=1}^{l_1} (\partial_{\epsilon}^{h_j}X_{i})(t,z,x,\epsilon)
\label{main_CP_diff_l_times}
\end{multline}
for all $l \geq 1$, all $(t,z,x,\epsilon) \in (\mathcal{T} \cap D(0,h'')) \times H_{\rho_{1}'} \times D(0,\rho_{1}) \times \mathcal{E}_{i}$.
We let $\epsilon$ tend to 0 in the equality (\ref{main_CP_diff_l_times}) and with the help of (\ref{lim_partial_l_epsilon_X_i}), we get the
following recursions
\begin{multline}
(-i\partial_{z}+1)^{r_1}\partial_{x}^{S}\frac{H_{l}(t,z,x)}{l!}\\
= \sum_{(s,k_{0},k_{1},k_{2}) \in \mathcal{S}} t^{s} (\sum_{h_{1}+h_{2}=l}
\frac{(\partial_{\epsilon}^{h_1}b_{s,k_{0},k_{1},k_{2}})(z,x,0)}{h_{1}!}
\frac{\partial_{t}^{k_0}\partial_{z}^{k_1}\partial_{x}^{k_2}H_{h_2}(t,z,x)}{h_{2}!})\\
+  \sum_{(l_{0},l_{1}) \in \mathcal{N}} t^{l_{0}+l_{1}-1} \sum_{h_{0}+h_{1}+\ldots+h_{l_1}=l}
\frac{(\partial_{\epsilon}^{h_0}c_{l_{0},l_{1}})(z,x,0)}{h_{0}!} \Pi_{j=1}^{l_1}\frac{H_{h_j}(t,z,x)}{h_{j}!}
\label{main_CP_diff_l_times_epsilon_equal_zero_1}
\end{multline}
for all $0 \leq l <r_{3}$, all $(t,z,x) \in (\mathcal{T} \cap D(0,h'')) \times H_{\rho_{1}'} \times D(0,\rho_{1})$, and
\begin{multline}
(t^{2}\partial_{t}+t)^{r_2}\partial_{x}^{S}\frac{H_{l-r_{3}}(t,z,x)}{(l-r_{3})!} +
(-i\partial_{z}+1)^{r_1}\partial_{x}^{S}\frac{H_{l}(t,z,x)}{l!}\\
= \sum_{(s,k_{0},k_{1},k_{2}) \in \mathcal{S}} t^{s} (\sum_{h_{1}+h_{2}=l}
\frac{(\partial_{\epsilon}^{h_1}b_{s,k_{0},k_{1},k_{2}})(z,x,0)}{h_{1}!}
\frac{\partial_{t}^{k_0}\partial_{z}^{k_1}\partial_{x}^{k_2}H_{h_2}(t,z,x)}{h_{2}!})\\
+  \sum_{(l_{0},l_{1}) \in \mathcal{N}} t^{l_{0}+l_{1}-1} \sum_{h_{0}+h_{1}+\ldots+h_{l_1}=l}
\frac{(\partial_{\epsilon}^{h_0}c_{l_{0},l_{1}})(z,x,0)}{h_{0}!} \Pi_{j=1}^{l_1}\frac{H_{h_j}(t,z,x)}{h_{j}!}
\label{main_CP_diff_l_times_epsilon_equal_zero_2}
\end{multline}
for every $l \geq r_{3}$ and all $(t,z,x) \in (\mathcal{T} \cap D(0,h'')) \times H_{\rho_{1}'} \times D(0,\rho_{1})$. Since the functions
$b_{s,k_{0},k_{1},k_{2}}(z,x,\epsilon)$ and $c_{l_{0},l_{1}}(z,x,\epsilon)$ are analytic with respect to $\epsilon$ near the origin in
$\mathbb{C}$, we get that
\begin{equation}
b_{s,k_{0},k_{1},k_{2}}(z,x,\epsilon) = \sum_{h \geq 0} \frac{(\partial_{\epsilon}^{h}b_{s,k_{0},k_{1},k_{2}})(z,x,0)}{h!} \epsilon^{h} \ \ , \ \
c_{l_{0},l_{1}}(z,x,\epsilon) = \sum_{h \geq 0} \frac{(\partial_{\epsilon}^{h}c_{l_{0},l_{1}})(z,x,0)}{h!} \epsilon^{h}
\label{Taylor_b_sk_and_c_l}
\end{equation}
for all $(s,k_{0},k_{1},k_{2}) \in \mathcal{S}$, $(l_{0},l_{1}) \in \mathcal{N}$ and all
$(z,x,\epsilon) \in H_{\rho_{1}'} \times D(0,\rho_{1}) \times D(0,\epsilon_{0})$. Finally, gathering the recursions
(\ref{main_CP_diff_l_times_epsilon_equal_zero_1}),  (\ref{main_CP_diff_l_times_epsilon_equal_zero_2}) and the expansions
(\ref{Taylor_b_sk_and_c_l}), one can see that the formal series (\ref{defin_hatG_hatX}) solves the equation (\ref{main_CP_formal}).
\end{proof}

\end{document}